\documentclass[12pt]{article}

\usepackage{a4}   %%% DIN A4
\usepackage{amssymb,amsmath,mathrsfs,textcomp}   %%% mathematical writing
\usepackage[amsmath,thmmarks]{ntheorem}   %%% theorem-like evironments
\usepackage{yhmath}   %%% here: makes the command \widehat{} more extensible;
\usepackage{indentfirst}   %%% indent in the first line of each section 
\usepackage{tocloft}   %%% manipulating ToC
\usepackage[hyperfootnotes=false,colorlinks=true,linkcolor=black,citecolor=black]{hyperref}   % referecing of page/theorem/... numbers; no reference for footnotes

\theoremstyle{nonumberplain}
\theorembodyfont{\rm}
\theoremseparator{.}
\newtheorem{definition}{Definition}

\theoremstyle{plain}
\theorembodyfont{\rm}
\theoremseparator{.}
\newtheorem{example}{Example}
\newtheorem{question}{Question}

\theoremstyle{plain}
\theorembodyfont{\slshape}
\theoremseparator{.}
\newtheorem{lemma}{Lemma}[section]

\newtheorem{proposition}{Proposition}[section]
\newtheorem{theorem}{Theorem}[section]

\theoremstyle{plain}
\theorembodyfont{\slshape}
\theoremseparator{$^\prime$.}
\newtheorem{theorem_prime}{Theorem}[section]

\theoremstyle{nonumberplain}
\theorembodyfont{\slshape}
\theoremseparator{.}

\theoremstyle{nonumberplain}
\theoremheaderfont{\bf}
\theorembodyfont{\rm} 
\theoremseparator{.}
\theoremsymbol{$\Box$}
\newtheorem{proof}{Proof}

\makeatletter
\renewcommand*{\@seccntformat}[1]{\csname the#1\endcsname.\quad}
\makeatother

\newcommand{\N}{\mathbb{N}}

\newcommand{\R}{\mathbb{R}}
\newcommand{\C}{\mathbb{C}}

\newcommand{\CP}{\mathbb{CP}}

\newcommand{\coloneqq}{\mathrel{\mathop:}=}
\newcommand{\eqqcolon}{=\mathrel{\mathop:}}

\providecommand{\abs}[1]{\lvert#1\rvert}
\providecommand{\gabs}[1]{\big\lvert#1\big\rvert}
\providecommand{\Gabs}[1]{\Big\lvert#1\Big\rvert}
\providecommand{\norm}[1]{\lVert#1\rVert}

\DeclareMathOperator{\Reg}{Reg}
\DeclareMathOperator{\Sing}{Sing}
\DeclareMathOperator{\Ker}{Ker}
\DeclareMathOperator{\supp}{supp}
\DeclareMathOperator{\dist}{dist}
\DeclareMathOperator{\rank}{rank}
\DeclareMathOperator{\ebdim}{ebdim}

\DeclareMathOperator{\co}{co}
\DeclareMathOperator{\card}{\#}
\DeclareMathOperator{\MA}{MA}
\DeclareMathOperator{\interior}{int}

\renewcommand{\Re}{\text{Re}\,}

\newcommand{\Lev}{\text{Lev}}

\providecommand{\chiup}{\raisebox{0.3ex}{$\chi$}}

\makeatletter
\def\moverlay{\mathpalette\mov@rlay}
\def\mov@rlay#1#2{\leavevmode\vtop{%
   \baselineskip\z@skip \lineskiplimit-\maxdimen
   \ialign{\hfil$\m@th#1##$\hfil\cr#2\crcr}}}
\newcommand{\charfusion}[3][\mathord]{
    #1{\ifx#1\mathop\vphantom{#2}\fi
        \mathpalette\mov@rlay{#2\cr#3}
      }
    \ifx#1\mathop\expandafter\displaylimits\fi}
\makeatother

\newcommand{\cupdot}{\charfusion[\mathbin]{\cup}{\cdot}}
\newcommand{\bigcupdot}{\charfusion[\mathop]{\bigcup}{\cdot}}

\begin{document}

%
%
%%% title %%%
\title{On defining functions for unbounded pseudoconvex domains}
\author{Tobias Harz$^{(\ast)}$, Nikolay Shcherbina and Giuseppe Tomassini}
\maketitle
\nopagebreak

%
%
%%% abstract %%%
\small\noindent{\bf Abstract.} 
We show that every strictly pseudoconvex domain $\Omega$ with smooth boundary in a complex manifold $\mathcal{M}$ admits a global defining function, i.e., a smooth plurisubharmonic function $\varphi \colon U \to \R$ defined on an open neighbourhood $U \subset \mathcal{M}$ of $\overline{\Omega}$ such that $\Omega = \{\varphi < 0\}$, $d\varphi \neq 0$ on $b\Omega$ and $\varphi$ is strictly plurisubharmonic near $b\Omega$. We then introduce the notion of the core $\mathfrak{c}(\Omega)$ of an arbitrary domain $\Omega \subset \mathcal{M}$ as the set of all points where every smooth and bounded from above plurisubharmonic function on $\Omega$ fails to be strictly plurisubharmonic. If $\Omega$ is not relatively compact in $\mathcal{M}$, then in general $\mathfrak{c}(\Omega)$ is nonempty, even in the case when $\mathcal{M}$ is Stein. It is shown that every strictly pseudoconvex domain $\Omega \subset \mathcal{M}$ with smooth boundary admits a global defining function that is strictly plurisubharmonic precisely in the complement of $\mathfrak{c}(\Omega)$. We then investigate properties of the core. Among other results we prove 1-pseudoconcavity of the core, we show that in general the core does not possess an analytic structure, and we investigate Liouville type properties of the core. \normalsize

%
%
%%% Subject classification %%
\renewcommand{\thefootnote}{}\footnote{$^{(\ast)}$The first author was supported in part by the DFG grant SH 456/1-1.}

%
%
%%% table of contents %%%
\setcounter{tocdepth}{1}
\tableofcontents

\section{Introduction}
In this paper we are dealing with the question of existence of defining functions for strictly pseudoconvex domains $\Omega$ with smooth boundary $b\Omega$ in a complex manifold $\mathcal{M}$. More precisely, we are interested in the existence of global defining functions, namely, defining functions that are defined in a neighbourhood of the closure $\overline{\Omega}$ (we will also be concerned with the more general situations of strictly $q$-pseudoconvex domains in complex manifolds and strictly hyper-$q$-pseudoconvex domains in complex spaces). In what follows, a real-valued function $\varphi$ will be called a \emph{defining function} for $\Omega$ if it has the following properties: 

\renewcommand{\arraystretch}{1.5}
\begin{tabular}{cl} 
(I)   & $\varphi$ is a smooth function on an open neighbourhood $U \subset \mathcal{M}$ of $b\Omega$. \\
(II)  & $\varphi$ is strictly plurisubharmonic in U. \\
(III) & $\Omega \cap U = \{\varphi < 0\}$ and $d\varphi \neq 0$ on $b\Omega$.
\end{tabular}
\renewcommand{\arraystretch}{1}

\noindent It is well known that defining functions always exist whenever $\Omega \subset \mathcal{M}$ is relatively compact, and there are different proofs available for this fact, see, for example, \cite{FritzscheGrauert02}, \cite{FornaessStensones87}, \cite{Grauert62}, \cite{MorrowRossi75}. In fact, a careful investigation of the corresponding proof shows that the method presented in \cite{Grauert62} still works, with only minor changes, even without assuming relative compactness of $\Omega$. In particular, every  strictly pseudoconvex domain with smooth boundary in a complex manifold admits a defining function.

If $\Omega$ is a relatively compact domain in a Stein manifold $\mathcal{M}$, then in fact more is known. In this case one can choose $\varphi$ to be defined not only near $b\Omega$ but on a neighbourhood of the whole of $\overline{\Omega}$ (see, for example, Lemma 1.3 in \cite{MorrowRossi75}). For arbitrary domains and manifolds this is not longer true in general, as it is illustrated by the following examples.

\begin{example} \label{ex_compactanalytic}
Let $\mathcal{M}$ be the blow-up of $\C^{n+1}$ at the origin, i.e., $\mathcal{M} \coloneqq \{(z,x) \in \C^{n+1} \times \CP^n : z \in l(x)\}$, where $l(x) \subset \C^{n+1}$ denotes the complex line determined by $x \in \CP^n$. Then 
\[ \Omega \coloneqq \big\{(z,x) \in \mathcal{M} : \norm{z} < 1 \big\} \subset\subset \mathcal{M} \]
is a strictly pseudoconvex domain with smooth boundary in $\mathcal{M}$ containing the compact analytic set $E \coloneqq \{0\} \times \CP^n$. Let $\varphi$ be a plurisubharmonic function defined on a neighbourhood of $\overline{\Omega}$ such that $\Omega = \{\varphi < 0\}$. Then $\varphi$ is bounded from above on $E$, hence it is a constant by the maximum principle. In particular, $\varphi$ is not strictly plurisubharmonic at the points of $E$.
\end{example}
 
\begin{example} \label{ex_complexline}
Let $f \colon \C \to \C$ be an entire function and
\[ \Omega:= \big\{(z,w) \in \C^2: \log\abs{w - f(z)} + C_1\big(\abs{z}^2 + \abs{w}^2\big) < C_2 \big\} \subset \C^2, \]
where $C_1$ and $C_2$ are constants and $C_1>0$. For almost all constants $C_2$, $\Omega$ is an unbounded strictly pseudoconvex domain with smooth boundary in $\C^2$ containing the complex line $L \coloneqq \{(z,f(z)) \in \C^2 : z \in \C\}$. Let $\varphi$ be a plurisubharmonic function defined on a neighbourhood of $\overline{\Omega}$ such that $\Omega = \{\varphi < 0\}$. Then $\varphi$ is subharmonic and bounded from above on $L$, hence it is a constant by Liouville's theorem. In particular, $\varphi$ is not strictly plurisubharmonic at the points of $L$.
\end{example}

As the above examples show, we cannot longer expect $\varphi$ to be strictly plurisubharmonic on a neighbourhood of the whole of $\overline{\Omega}$ as soon as $\mathcal{M}$ fails to be Stein, or $\Omega$ fails to be relatively compact in $\mathcal{M}$. Hence we will call a real-valued function $\varphi$ a \textit{global defining function} for $\Omega$ if it has the following properties:

\renewcommand{\arraystretch}{1.5}
\begin{tabular}{cl} 
(I)   & $\varphi$ is a smooth function on an open neighbourhood $U \subset \mathcal{M}$ of $\overline{\Omega}$. \\
(II)  & $\varphi$ is plurisubharmonic in $U$ and strictly plurisubharmonic near $b\Omega$. \\
(III) & $\Omega = \{\varphi < 0\}$ and $d\varphi \neq 0$ on $b\Omega$.
\end{tabular}
\renewcommand{\arraystretch}{1}

\noindent Observe that instead of imposing (I) and (II), it is equivalent to claim that $\varphi$ is a smooth plurisubharmonic function on $\overline{\Omega}$ such that $\varphi$ is strictly plurisubharmonic near $b\Omega$. Moreover, after possibly shrinking $U$ and composing $\varphi$ with a suitable convex function, we can always assume that $\varphi$ is bounded.

As the main result of our article we will prove that every strictly pseudoconvex domain $\Omega$ with smooth boundary in a complex manifold $\mathcal{M}$ admits a global defining function. In view of this result and the examples above, it is then meaningful to consider the set of all points in $\Omega$ where every global defining function for $\Omega$ fails to be strictly plurisubharmonic. We will show that this set coincides with the core of $\Omega$, which we introduce in the following definition.

\begin{definition}
Let $\mathcal{M}$ be a complex manifold and let $\Omega \subset \mathcal{M}$ be a domain. Then the set
\begin{equation*}\begin{split}
\mathfrak{c}(\Omega) \coloneqq \big\{&z \in \Omega : \text{every smooth plurisubharmonic function on $\Omega$ that is} \\& \text{bounded from above fails to be strictly plurisubharmonic in } z \big\}
\end{split}\end{equation*}
will be called the \emph{core} of $\Omega$.
\end{definition}

\noindent \textbf{Remark.} Similar definitions in different settings have also been introduced in \cite{HarveyLawson12}, \cite{HarveyLawson13} and \cite{SlodkowskiTomassini04}.

\noindent As we shall see, every domain $\Omega \subset \mathcal{M}$ admits a smooth and bounded plurisubharmonic function that is strictly plurisubharmonic precisely in $\Omega \setminus \mathfrak{c}(\Omega)$. In the special case of global defining functions we get the following version of our main theorem.

\noindent \textbf{Main Theorem.} \label{thm_maintheorem} {\it Every strictly pseudoconvex domain $\Omega$ with smooth boundary in a complex manifold $\mathcal{M}$ admits a bounded global defining function that is strictly plurisubharmonic outside $\mathfrak{c}(\Omega)$.}

In a sense, the above theorem gives a complete answer to the question of existence of global defining functions. On the other hand, the presence of the core $\mathfrak{c}(\Omega)$ is a phenomenon that deserves further investigation. The major part of our article will be devoted to this topic.

The content of the paper is organized as follows: In Section \ref{sec_deffct} we prove the existence of global defining functions for strictly pseudoconvex domains $\Omega$ in arbitrary complex spaces. In the same context we also prove a theorem that guarantees the existence of smooth plurisubharmonic functions defined in a neighbourhood of $\overline{\Omega}$ which are strictly plurisubharmonic near $b\Omega$ and have arbitrary bounded from below and smooth boundary data. Analogous results are shown for strictly $q$-pseudoconvex domains in complex manifolds and strictly hyper-$q$-pseudoconvex domains in complex spaces. Moreover, we show that every strictly pseudoconvex domain in a complex manifold (not necessarily relatively compact or with smooth boundary) admits a neighbourhood basis consisting of strictly pseudoconvex domains with smooth boundary (a slightly weaker result is also proven in the setting of complex spaces).

In Section \ref{sec_core} we study elementary properties of the core $\mathfrak{c}(\Omega)$ and discuss some examples. As the most notable result we prove that the core is always $1$-pseudoconcave, and we show that this result cannot be further improved in general. In particular, plurisubharmonic functions have the local maximum property on $\mathfrak{c}(\Omega)$.

In Sections \ref{sec_noanalyticcore} and \ref{sec_liouville} we investigate the structure of the core of $\Omega$. We start by observing that in each example of Section \ref{sec_core}, as well as in the Examples \ref{ex_compactanalytic} and \ref{ex_complexline} above, every connected component $Z$ of $\mathfrak{c}(\Omega)$ has the following two properties: 
\begin{itemize}
  \item $Z$ satisfies a Liouville type theorem, i.e., every smooth and bounded from above plurisubharmonic function on $\Omega$ is constant on $Z$. 
  \item $Z$ possesses an analytic structure, i.e., there exists a dense subset of $Z$ that is the union of nonconstant holomorphic discs contained in $Z$ (in fact, in the above examples Z is always a complex manifold).
\end{itemize}
Clearly, every set $Z \subset \Omega$ with the above two properties has to be contained in $\mathfrak{c}(\Omega)$. We then want to know whether it is true in general that every connected component of $\mathfrak{c}(\Omega)$ has a Liouville type property and possesses an analytic structure. We will show that in this general formulation the answers to both of the above questions are negative. In fact, in Section \ref{sec_liouville} we construct examples of strictly pseudoconvex domains $\Omega \subset \C^n$ with smooth boundary such that $\mathfrak{c}(\Omega)$ is connected, but not every smooth and bounded from above plurisubharmonic function on $\Omega$ is constant on $\mathfrak{c}(\Omega)$. However, in these examples it is still the case that $\mathfrak{c}(\Omega)$ is the disjoint union of unbounded connected sets $L_\alpha \subset \mathfrak{c}(\Omega)$, $\alpha \in A$, where each set $L_\alpha$ satisfies a Liouville type theorem. Moreover, in the case when $\dim_\C \mathcal{M} = 2$ we can show that if $Z$ is a connected component of $\mathfrak{c}(\Omega)$ such that there exists a smooth and bounded from above plurisubharmonic function on $\Omega$ that is not constant on $Z$, then there exist uncountably many disjoint complex curves $\gamma_\alpha \subset Z$ such that each $\gamma_\alpha$ does satisfy a Liouville type theorem, see Theorem \ref{thm_foliatedcore} for the precise result. Concerning the analytic structure of $\mathfrak{c}(\Omega)$, we give in Section \ref{sec_noanalyticcore} an example of a strictly pseudoconvex domain with smooth boundary such that $\mathfrak{c}(\Omega)$ contains no analytic variety of positive dimension. However, by construction, $\mathfrak{c}(\Omega)$ will still be the limit in the Hausdorff metric of a sequence of analytic varieties (in fact, it will be a Wermer type set as constructed in \cite{HarzShcherbinaTomassini12}). We also prove a Liouville theorem for Wermer type sets.

Section \ref{sec_geometricproperties} deals with some geometric properties of Wermer type sets. In particular, we prove that the Wermer type sets from \cite{HarzShcherbinaTomassini12} are $(M, 1/2)$-H\"older continuous for some constant $M > 0$. We use these results to give a second proof for the smoothing procedure that is needed in the constructions of Section 4. 

In Section \ref{sec_higherordercore} we introduce a series of stronger notions of the core of a domain $\Omega \subset \mathcal{M}$, by requiring not only failure of strict plurisubharmonicity of smooth and bounded from above plurisubharmonic functions $\varphi$ on $\Omega$, but instead by prescribing an upper bound $k$ for the rank of the complex Hessian of $\varphi$, with $k$ possibly different from $\dim_\C \mathcal{M}-1$. We show that this sharpening of the definition does not lead to better results on pseudoconcavity of the core.

Finally, Section \ref{sec_questions} contains a series of open questions related to the content of this paper.
\section{Existence of global defining functions} \label{sec_deffct}
We prove the existence of global defining functions in a number of different settings. We first consider in \ref{sec_deffctmanifolds} the case of strictly $q$-pseudoconvex domains in complex manifolds. Later on we deal in Section \ref{sec_deffctspaces} with the situation of strictly hyper-$q$-pseudoconvex domains in complex spaces. 

\subsection{Global defining functions in complex manifolds} \label{sec_deffctmanifolds}
In this section we prove the existence of global defining functions in the setting of complex manifolds. Our focus will lie on strictly pseudoconvex domains, but when it is possible we formulate the results in the more general context of strictly $q$-pseudoconvex domains. We also discuss to which extent smoothness of $b\Omega$ is needed in our results. We start by recalling some definitions and by fixing our notation. 

Let $\mathcal{M}$ be a complex manifold of complex dimension $n \coloneqq \dim_\C\mathcal{M}$. The holomorphic tangent space to $\mathcal{M}$ at $z \in \mathcal{M}$ is denoted by $T_z(\mathcal{M})$, and we write $T(\mathcal{M})$ for the holomorphic tangent bundle of $\mathcal{M}$. If $\varphi \colon \mathcal{M} \to \R$ is a smooth function, then we define $(\partial\varphi)_z, (\overline{\partial}\varphi)_z \colon T_z(\mathcal{M}) \to \C$ and $\Lev(\varphi)(z,\,\cdot\,) \colon T_z(\mathcal{M}) \to \R$ in local holomorphic coordinates $h = (z_1, \ldots, z_n)$ by
\begin{gather*} 
 (\partial\varphi)_z(\xi) \coloneqq \sum_{j=1}^n \frac{ \partial(\varphi \circ h^{-1})}{\partial z_j}(h(z)) \xi_j, \quad (\overline{\partial}\varphi)_z(\xi) \coloneqq \sum_{j=1}^n \frac{ \partial(\varphi \circ h^{-1})}{\partial \overline{z}_j}(h(z)) \overline{\xi}_j, \\ 
 \Lev(\varphi)(z,\xi) \coloneqq \sum_{j,k=1}^n \frac{\partial^2 (\varphi \circ h^{-1})}{\partial z_j \partial \overline{z}_k}(h(z)) \xi_j\overline{\xi}_k, 
\end{gather*}
where $\xi = \sum_{j=1}^n \xi_j \, (\partial/\partial z_j)$. Moreover, $(d\varphi)_z(\xi) \coloneqq (\partial \varphi)_z(\xi) + (\overline{\partial}\varphi)_z(\xi) \colon T_z(\mathcal{M}) \to \R$ denotes the real differential of $\varphi$ in $z$. Further, we write $H_z(\varphi)$ for the complex subspace of $T_z(\mathcal{M})$ defined by $H_z(\varphi) \coloneqq \{\xi \in T_z(\mathcal{M}) : (\partial \varphi)_z(\xi) = 0\}$. If $h$ is a hermitian metric on $\mathcal{M}$, then for every $z \in \mathcal{M}$ we denote by $\norm{\,\cdot\,}_{h_z}$ and $\norm{\,\cdot\,}_{h_z^\ast}$ the induced norms on $T_z(\mathcal{M})$ and on the dual space $T_z^\ast(\mathcal{M})$, respectively. If the context is clear, then we sometimes omit the index $z$ and simply write $\norm{\,\cdot\,}_h$ and $\norm{\,\cdot\,}_{h^\ast}$. (Throughout this article the term ``smooth'' always means ``$\mathcal{C}^\infty$-smooth''. Of course, the above definitions of the various differentials and of the Levi form are possible for $\mathcal{C}^1$-smooth and $\mathcal{C}^2$-smooth functions, respectively.) 

An upper semicontinuous function $\varphi \colon \mathcal{M} \to [-\infty, \infty)$ is called plurisubharmonic if for every holomorphic mapping $f \colon G \to \mathcal{M}$ of an open set $G \subset \C$ into $\mathcal{M}$ the composition $\varphi \circ f$ is subharmonic on $G$. It is called strictly plurisubharmonic if for every compactly supported smooth function $\theta \colon \mathcal{M} \to \R$ there exists some number $\varepsilon_{0} > 0$ such that $\varphi + \varepsilon\theta$ is plurisubharmonic whenever $\abs{\varepsilon} \le \varepsilon_0$. If the function $\varphi$ is $\mathcal{C}^2$-smooth, then it is (strictly) plurisubharmonic if and only if $\Lev(\varphi)(z,\,\cdot\,)$ has precisely $n$ (positive) nonnegative eigenvalues for every $z \in \mathcal{M}$. An open set $\Omega \subset \mathcal{M}$ is called strictly pseudoconvex at $z \in b\Omega$ if there exist an open neighbourhood $U_z \subset \mathcal{M}$ of $z$ and a continuous strictly plurisubharmonic function $\varphi_z \colon U_z \to \R$ such that $\Omega \cap U_z = \{\varphi_z < 0\}$. It is called $\mathcal{C}^s$-smooth at $z \in b\Omega$, $s \ge 1$, if there exist an open neighbourhood $U_z \subset \mathcal{M}$ of $z$ and a $\mathcal{C}^s$-smooth function $\tilde{\varphi}_z \colon U_z \to \R$ such that $\Omega \cap U_z = \{\tilde{\varphi_z} < 0\}$ and $d\tilde{\varphi}_z \neq 0$ on $b\Omega \cap U_z$. The open set $\Omega$ is called strictly pseudoconvex or $\mathcal{C}^s$-smooth if it is strictly pseudoconvex or $\mathcal{C}^s$-smooth at each boundary point, respectively. 

Let $q \in \N_0 \coloneqq \{0\} \cup \N$. A $\mathcal{C}^2$-smooth function $\varphi \colon \mathcal{M} \to \R$ is called (strictly) $q$-plurisubharmonic if $\Lev(\varphi)(z,\,\cdot\,)$ has at least $n-q$ (positive) nonnegative eigenvalues for every $z \in \mathcal{M}$. Observe that $q$-plurisubharmonic functions are also $(q+1)$-plurisubharmonic for every $q \in \N_0$. Moreover, the $0$-plurisubharmonic functions are precisely the $\mathcal{C}^2$-smooth plurisubharmonic functions, and a function is $q$-plurisubharmonic for $q \ge n$ if and only if it is $\mathcal{C}^2$-smooth. An open set $\Omega \subset \mathcal{M}$ is called strictly $q$-pseudoconvex at $z \in b\Omega$ if there exist an open neighbourhood $U_z \subset \mathcal{M}$ of $z$ and a strictly $q$-plurisubharmonic function $\varphi_z \colon U_z \to \R$ such that $\Omega \cap U_z = \{\varphi_z < 0\}$. Moreover, the set $\Omega$ is called strictly $q$-pseudoconvex if it is strictly $q$-pseudoconvex at each boundary point. 

Before we start our studies on global defining functions, we want to mention that it is not completely trivial to see that a strictly pseudoconvex domain with $\mathcal{C}^2$-smooth boundary can locally near each boundary point be defined by a $\mathcal{C}^2$-smooth strictly plurisubharmonic function. Since we were not able to find a proof of this fact in the literature, we state here the following proposition.

\begin{proposition} \label{thm_strpscregularity}
Let $\mathcal{M}$ be a complex manifold, let $\Omega \subset \mathcal{M}$ be open and let $b\Omega$ be $\mathcal{C}^s$-smooth at $z_0 \in b\Omega$, $s \ge 2$. Assume that there exist an open neighbourhood $U \subset \mathcal{M}$ of $z_0$ and a strictly plurisubharmonic function $\psi \colon U \to [-\infty, \infty)$ such that $\Omega \cap U = \{\psi < 0\}$. Then, after possibly shrinking $U$, there exists a $\mathcal{C}^s$-smooth strictly plurisubharmonic function $\varphi \colon U \to \R$ such that $\Omega \cap U = \{\varphi < 0\}$ and $d\varphi \neq 0$ on $b\Omega \cap U$.
\end{proposition}
\begin{proof}
Observe that the statement is trivial in the case $\dim_\C \mathcal{M} = 1$. Thus we may assume that $n \coloneqq \dim_\C \mathcal{M} \ge 2$. By assumption, we can find an open neighbourhood $U \subset \mathcal{M}$ of $z_0$, a strictly plurisubharmonic function $\psi \colon U \to [-\infty,\infty)$ such that $\Omega \cap U = \{\psi < 0\}$ and a $\mathcal{C}^s$-smooth function $\widetilde{\varphi} \colon U \to \R$ such that $\Omega \cap U = \{\widetilde{\varphi} < 0\}$ and $d\widetilde{\varphi} \neq 0$ on $b\Omega \cap U$. After possibly shrinking $U$, and after introducing suitable holomorphic coordinates around $z_0$, we can assume that $U \subset\subset \C^n$, $z_0 = 0$, $\psi$ and $\widetilde{\varphi}$ are defined in a neighbourhood of $\overline{U}$ and the Taylor expansion of $\widetilde{\varphi}$ around $0$ has the form
\begin{equation} \label{equ_taylorphi}
\widetilde{\varphi}(\xi) = \Re \xi_1 + \Lev(\widetilde{\varphi})(0,\xi) + o(\norm{\xi}^2).
\end{equation}
For every $z \in U$, let $\dist(z,b\Omega) \coloneqq \inf_{z' \in b\Omega \cap U} \norm{z-z'}$ and for $\zeta \in b\Omega \cap U$ let $N_\Omega(\zeta)$ be the outward unit normal vector to $b\Omega$ at $\zeta$.

We claim that $\psi \equiv 0$ on $b\Omega \cap U$. Indeed, for every plurisubharmonic function $u$ defined near some point $\zeta \in \C^n$ and for every $w \in \C^n$ it holds true that $u(\zeta) = \limsup_{t \to 0+} u(\zeta + tw)$, see, for example, Proposition 7.4 in \cite{FornaessStensones87}. In particular, $\psi(\zeta) = \limsup_{t \to 0+} \psi(\zeta - tN_\Omega(\zeta)) \le 0$ for every $\zeta \in b\Omega \cap U$. The fact that $\psi \ge 0$ on $b\Omega \cap U$ is clear by the choice of $\psi$.

As the next step we claim that, after possibly shrinking $U$, there exist numbers $l,L > 0$ such that $\widetilde{\varphi}(z) \ge l \dist(z,b\Omega)$ and $\psi(z) \le L \dist(z,b\Omega)$ for every $z \in U \setminus \Omega$. Clearly, we only need to show the assertion on $\psi$, since the inequality for $\widetilde{\varphi}$ follows immediately from the fact that $d\widetilde{\varphi} \neq 0$ on $b\Omega \cap U$. The proof is similar to that of the Hopf Lemma: First we can assume, after possibly shrinking $U$, that the orthogonal projection $\pi \colon U \to b\Omega \cap U$ along the normal vectors $N_\Omega(\zeta)$ is well defined. In particular, $z = \pi(z) + \dist(z,b\Omega)N_\Omega(\pi(z))$ for every $z \in U \setminus \Omega$. For every $\zeta \in b\Omega \cap U$ and every $r > 0$, let $\zeta_r \coloneqq \zeta -rN_\Omega(\zeta)$. By $\mathcal{C}^2$-smoothness of $b\Omega \cap U$, we can then choose $r>0$ so small that for every $\zeta \in b\Omega \cap B^n(0,r)$ one has
\begin{enumerate}
  \item[(i)] $B^n(\zeta, 4r) \subset U$ and $b\Omega \cap B^n(\zeta, 4r)$ is the graph of a $\mathcal{C}^2$-smooth function over some open subset of $\zeta + T^\R_\zeta(b\Omega)$,
  \item[(ii)] $B^n(\zeta_{2r},2r) \subset \Omega \cap U$ and $\overline{B^n(\zeta_{2r},2r)} \cap b\Omega = \{\zeta\}$,
\end{enumerate}
where $B^n(a,r) \coloneqq \{z \in \C^n : \norm{z-a} < r\}$ and $T^\R_\zeta(b\Omega)$ denotes the real tangent space to $b\Omega$ at $\zeta$. For every $\zeta \in b\Omega \cap B^n(0,r)$ let $G_\zeta \coloneqq B^n(\zeta,r) \setminus \overline{\Omega}$ and let $h_\zeta \colon \overline{G_\zeta} \to \R$ be the function
\[h_\zeta(z) \coloneqq \frac{1}{r^{2n-2}} - \frac{1}{\norm{z-\zeta_r}^{2n-2}}. \]
Observe that $h_\zeta$ is harmonic on $G_\zeta$ and continuous on $\overline{G_\zeta}$, $h_\zeta(\zeta) = 0$, $h_\zeta > 0$ on $bG_\zeta \setminus \{\zeta\}$ and there exists a constant $c > 0$ such that $h_\zeta > c$ on $bG_\zeta \setminus b\Omega$ for every $\zeta \in b\Omega \cap B^n(0,r)$. Choose $C > 0$ so large that $\psi \le C$ on $U$ and set $M \coloneqq C/c$. Then, since $\psi \equiv 0$ on $b\Omega \cap U$, we have $\psi \le Mh_\zeta$ on $bG_\zeta$ for every $\zeta \in b\Omega \cap B^n(0,r)$. By subharmonicity of $\psi$, it follows that $\psi \le Mh_\zeta$ on $\overline{G_\zeta}$. In particular,
\begin{equation*} \begin{split}
\psi(\zeta+tN_\Omega(\zeta)) &\le Mh_\zeta(\zeta+tN_\Omega(\zeta)) = M \Big(\frac{1}{r^{2n-2}} - \frac{1}{(r+t)^{2n-2}} \Big) \\ &\le M(2n-2)\frac{1}{r^{2n-1}} t \eqqcolon L t
\end{split} \end{equation*}
for every $t \in (0,r)$ and $\zeta \in b\Omega \cap B^n(0,r)$. This shows that $\psi(z) \le L\dist(z,b\Omega)$ for every $z \in \{\zeta + tN_\Omega(\zeta) \in \C^n : \zeta \in b\Omega \cap B^n(0,r), t \in [0,r)\}$.

Now assume, to get a contradiction, that there exists $\xi_0 \in H_{0}(\widetilde{\varphi})$ such that $\Lev(\widetilde{\varphi})(0,\xi_0) \le 0$ and $\norm{\xi_0} = 1$. Choose $\varepsilon > 0$ such that $\widetilde{\psi} \coloneqq \psi - \varepsilon \norm{\,\cdot \,}^2$ is still plurisubharmonic on $U$. We claim that $\widetilde{\psi} < 0$ on the punctured complex disc $\Delta_{\xi_0}(0,\delta) \setminus \{0\} \coloneqq \{\lambda\xi_0 : \lambda \in \C, 0 < \abs{\lambda} < \delta\}$, provided $\delta > 0$ is chosen small enough. Since $\widetilde{\psi}$ is subharmonic on $\Delta_{\xi_0}(0,\delta)$, and since $\widetilde{\psi}(0) = 0$, this will be a contradiction to the maximum principle. Indeed, if $\lambda \in \C \setminus\{0\}$ is chosen in such a way that $\lambda\xi_0 \in \overline{\Omega}$, then the statement is trivial. But otherwise we can use $(\ref{equ_taylorphi})$ and the estimates on $\psi$ and $\widetilde{\varphi}$ which where given above to see that
\[ \widetilde{\psi}(\lambda\xi_0) = \big( \frac{\psi}{\widetilde{\varphi}} \widetilde{\varphi} - \varepsilon \norm{\,\cdot\,}^2\big)(\lambda\xi_0) \le \frac{L}{l}o(\abs{\lambda}^2) - \varepsilon\abs{\lambda}^2, \]
which is negative if $0 < \abs{\lambda} << 1$. This shows that $\Lev(\widetilde{\varphi})(0,\,\cdot\,)$ is positive definit on $H_0(\widetilde{\varphi})$.

To conclude the proof of the proposition, choose a smooth function $\chiup \colon \R \to \R$ such that $\chiup(0) = 0$, $\chiup'(0) = 1$ and $\chiup''(0) = k $. It follows then by a standard argument that for $k > 0$ large enough the function $\varphi \coloneqq \chiup \circ \widetilde{\varphi}$ is strictly plurisubharmonic near $0$ as desired (for a version of this argument see, for example, the proof of Lemma \ref{thm_uniformqpsh} below).
\end{proof}

\noindent \textbf{Remarks.} 1) The above definition of strictly pseudoconvex open sets in complex manifolds is the same as the one given in \cite{Narasimhan62}. In particular, the strictly plurisubharmonic functions $\varphi_z$ that define $\Omega$ near a given point $z \in b\Omega$ are assumed to be continuous. Observe that by dropping the assumption on continuity of the local defining functions $\varphi_z$ one obtains a class of sets that is strictly larger than the class of strictly pseudoconvex sets. Indeed, the function $u(z) \coloneqq \sum_{j=1}^\infty 2^{-j}\log\abs{z-1/j}$ is well defined and subharmonic on $\C$ such that $u(0) \neq -\infty$, and thus $\psi(z,w) \coloneqq u(z) + (\abs{z}^2 + \abs{w}^2) - u(0)$ is a strictly plurisubharmonic function on $\C^2$. Consider the open set $\Omega \coloneqq \{\psi < 0\}$. Then $L \coloneqq \{0\} \times \C \subset b\Omega$, since $\{1/j\} \times \C \subset \Omega$ for every $j \in \N$. In particular, since $\psi(z,w) = \abs{w}^2 \not\equiv 0$ on $L$, there exists no neighbourhood of $b\Omega$ on which $\psi$ is continuous. Assume, to get a contradiction, that there exist an open neighbourhood $U \subset\subset \C^2$ of $0 \in b\Omega$ and a continuous strictly plurisubharmonic function $\varphi \colon U \to \R$ such that $\Omega \cap U = \{\varphi < 0\}$. Let then $U' \subset\subset U$ be open such that $0 \in U'$ and let $\lambda \colon \C^2 \to (-\infty,0]$ be smooth such that $\overline{U'} = \{\lambda = 0\}$. After possibly shrinking $U$, we can find $\varepsilon > 0$ such that $\varphi' \coloneqq \varphi + \varepsilon\lambda$ is still plurisubharmonic on $U$. Since $\varphi$ is continuous, we have $\varphi \equiv 0$ on $L \subset b\Omega$, and thus there exists $c > 0$ such that $\varphi' < -c$ near $L \cap bU$. It follows that $\varphi'|_{L \cap U}$ is a nonconstant subharmonic function that attains a maximum at $0 \in L \cap U$, which is a contradiction. Observe, in particular, that the boundary of a sublevel set of a not necessarily continuous strictly plurisubharmonic function may contain non trivial analytic sets.

2) The described above problem cannot happen if $b\Omega$ satisfies some mild regularity assumptions. Namely, the following analogue of Proposition \ref{thm_strpscregularity} holds true in the $\mathcal{C}^0$-smooth category:

\begingroup \leftskip3.5ex 
{\it Let $\mathcal{M}$ be a complex manifold, let $\Omega \subset \mathcal{M}$ be open and let $b\Omega$ be $\mathcal{C}^0$-smooth at $z_0 \in b\Omega$. Assume that there exist an open neighbourhood $U \subset \mathcal{M}$ of $z_0$ and a strictly plurisubharmonic function $\psi \colon U \to [-\infty, \infty)$ such that $\Omega \cap U = \{\psi < 0\}$. Then, after possibly shrinking $U$, there exists a continuous strictly plurisubharmonic function $\varphi \colon U \to \R$ such that $\Omega \cap U = \{\varphi < 0\}$.} 
\par \endgroup

\noindent Indeed, after possibly shrinking $U$, we can assume that $U \subset \C^n$. By continuity of $b\Omega$ at $z_0$, there exists $w \in \C^n \setminus \{0\}$ such that, after possibly further shrinking $U$, $z + tw \in \Omega$ for every $z \in b\Omega \cap U$ and $t \in (0,1)$. Then, by the same argument as above, it follows that $\psi(z) = \limsup_{t \to 0+} \psi(z + tw) \le 0$ for every $z \in b\Omega \cap U$. Thus $\psi \equiv 0$ on $b\Omega \cap U$. The existence of the function $\varphi$ now follows from Theorem 2.5 in \cite{Richberg68}.

Now we begin to prove the existence of global defining functions. We will formulate our results in the general context of strictly $q$-pseudoconvex domains, since the essential part of our proof will be the same in both cases $q = 0$ and $q > 0$. However, at a certain point of our construction a technical problem will occur in the case $q > 0$, which is not present if $q = 0$. This problem is related to the fact that the sum of two $q$-plurisubharmonic functions $\varphi_1, \varphi_2 \colon U \to \R$ on an open set $U \subset \mathcal{M}$ will in general be again $q$-plurisubharmonic only if both $\Lev(\varphi_1)(z,\,\cdot\,)$ and $\Lev(\varphi_2)(z,\,\cdot\,)$ are positive definit on the same $(n-q)$-dimensional subspaces of $T_z(\mathcal{M})$ for every $z \in U$. Thus in the case $q > 0$ we need to keep track of the directions of positivity of the Levi forms of the $q$-plurisubharmonic functions involved in our construction. That is why before stating our theorems on global defining functions we first prove the following lemma which deals with this particular problem of the case $q > 0$. 

\begin{lemma} \label{thm_uniformqpsh}
Let $\mathcal{M}$ be a complex manifold of dimension $n$ equipped with a hermitian metric $h$ and let $\Omega \subset \mathcal{M}$ be a strictly $q$-pseudocon\-vex domain with smooth boundary for some $q \in \{0,1, \ldots, n-1\}$. Then for every smooth function $\varphi \colon V \to \R$ defined on an open neighbourhood $V \subset \mathcal{M}$ of $b\Omega$ such that $\Omega \cap V = \{\varphi < 0\}$ and $d\varphi \neq 0$ on $b\Omega$, there exist a neighbourhood $V' \subset V$ of $b\Omega$ and for each $z \in V'$ an $(n-q)$-dimensional complex subspace $L_z \subset T_z(\mathcal{M})$ such that the following assertion holds true: for every open set $U \subset \subset \mathcal{M}$ there exist a strictly increasing and strictly convex smooth function $\mu \colon \R \to \R$ satisfying $\mu(0) = 0$ and a constant $c > 0$ such that $\Lev(\mu \circ \varphi)(z, \xi) \ge c\norm{\xi}_{h_z}$ for every $z \in V' \cap U$ and $\xi \in L_z$.
\end{lemma}
\begin{proof}
Let $\varphi \colon V \to \R$ be a smooth function defined on an open neighbourhood $V \subset \mathcal{M}$ of $b\Omega$ such that $\Omega \cap V = \{\varphi < 0\}$ and $d\varphi \neq 0$ on $b\Omega$. After possibly shrinking $V$, we can assume that $d\varphi \neq 0$ on $V$, and then for every $z \in V$ we write $N_z$ for the orthogonal complement of $H_z(\varphi)$ in $T_z(\mathcal{M})$ with respect to $h_z$.

For every $l = 1, 2, \ldots, n$, denote by $G_l(\mathcal{M})$ the Grassmann bundle of dimension $l$ over $\mathcal{M}$, i.e., for every $z \in \mathcal{M}$ the fiber $G_l(\mathcal{M})_z = \{L \subset T_z(\mathcal{M}) : (z,L) \in G_l(\mathcal{M})\}$ consists of all complex $l$-dimensional subspaces of $T_z(\mathcal{M})$. Write $\pi = \pi_l \colon G_l(\mathcal{M}) \to \mathcal{M}$ for the canonical projection of $G_l(\mathcal{M})$ onto $\mathcal{M}$. Since $\Omega$ is strictly $q$-pseudoconvex with smooth boundary, there exists a closed in $\mathcal{M}$ neighbourhood $V' \subset V$ of $b\Omega$ and a closed subset $\widetilde{\mathcal{L}} \subset G_{n-1-q}(\mathcal{M})$ such that $\pi(\widetilde{\mathcal{L}}) = V'$ with the following properties: for every $z \in V'$ and every $\widetilde{L} \in \widetilde{\mathcal{L}}_z \coloneqq \{\widetilde{L} \subset T_z(\mathcal{M}) : (z,\widetilde{L}) \in \widetilde{\mathcal{L}}\}$ we have $(\partial \varphi)_z(\,\cdot\,) \equiv 0$ on $\widetilde{L}$ and $\Lev(\varphi)(z,\,\cdot\,) > 0$ on $\widetilde{L} \setminus \{0\}$. Set
\[ \mathcal{L} \coloneqq \big\{(z,L) \in G_{n-q}(\mathcal{M})|_{V'} : L = \widetilde{L} \oplus N_z \text{ for some } \widetilde{L} \in \widetilde{\mathcal{L}}_z\big\}. \]
We claim that $V'$ and any choice of $\{L_z\}_{z \in V'}$ such that $L_z \in \mathcal{L}_z$ for every $z \in V'$ are a neighbourhood of $b\Omega$ and a family of complex subspaces as desired.

Indeed, let $U \subset \subset \mathcal{M}$ be open. Define a map $\tau = \tau_l \colon G_l(\mathcal{M}) \to \mathcal{P}(T(\mathcal{M}))$ from $G_l(\mathcal{M})$ to the set of subsets of $T(\mathcal{M})$ by $\tau((z,L)) \coloneqq \bigcup_{\xi \in L, \norm{\xi}_{h_z} = 1} (z,\xi)$. Let $\widetilde{\mathcal{S}} \coloneqq \tau(\widetilde{\mathcal{L}})$, $\mathcal{S} \coloneqq \tau(\mathcal{L})$ and $\mathcal{S}_0 \coloneqq \{(z,\xi) \in \mathcal{S} : \Lev(\varphi)(z,\xi) \le 0\}$. Observe that, by construction, $\widetilde{\mathcal{S}} = \{(z,\xi) \in \mathcal{S} : (\partial\varphi)_z(\xi) = 0\} \subset \mathcal{S} \setminus \mathcal{S}_0$, and $\mathcal{S}_0$ is closed in $T(\mathcal{M})$. In particular, for every $z \in V'$ we have that $\delta_0(z) \coloneqq \min_{\xi \in \mathcal{S}_{0,z}} \abs{(\partial\varphi)_z(\xi)} > 0$, where $S_{0,z} \coloneqq \{\xi \in T_z(\mathcal{M}) : (z,\xi) \in \mathcal{S}_0\}$. Moreover, since $\mathcal{S} \setminus \mathcal{S}_0$ is an open neighbourhood of $\widetilde{\mathcal{S}}$ in $\mathcal{S}$, one sees easily that it is possible to choose a continuous function $\delta \colon V' \to (0,\infty)$ such that $\delta(z) < \delta_0(z)$ for every $z \in V'$. Let $C \colon V' \to \R$ be a continuous function such that $\Lev(\varphi)(z,\xi) > C(z)$ for every $(z,\xi) \in \mathcal{S}$. Now choose $k > 0$ so large that $C(z) + 2k\delta^2(z) > 0$ on $V' \cap \overline{U}$ and define $\mu \colon \R \to \R$ as $\mu(t) \coloneqq  te^{kt}$. Then $\Lev(\mu \circ \varphi)(z,\xi) = \Lev(\varphi)(z,\xi) + 2k\abs{(\partial\varphi)_z(\xi)}^2 > 0$ for every $z \in V' \cap \overline{U}$ and $\xi \in \mathcal{S}_{0,z}$, and clearly $\Lev(\mu \circ \varphi)(z,\xi) > 0$ for every $z \in V' \cap \overline{U}$ and $\xi \in \mathcal{S} \setminus \mathcal{S}_0$. Since $\mathcal{S}$ is closed in $T(\mathcal{M})$, and since $\mathcal{S} = \tau(\mathcal{L})$, it follows that there exists a constant $c > 0$ such that $\Lev(\mu \circ \varphi)(z,\,\cdot\,) \ge c\norm{\,\cdot\,}^2$ on $L$ for every $z \in V' \cap \overline{U}$ and $L \in \mathcal{L}_z$.
\end{proof}

After these preparations we can now prove the first two theorems of this section.

\begin{theorem} \label{thm_qbdvalue_manifolds}
Let $\mathcal{M}$ be a complex manifold of dimension $n$, let $\Omega \subset \mathcal{M}$ be a strictly $q$-pseu\-do\-con\-vex domain with smooth boundary for some $q \in \{0, 1, \ldots, n-1\}$ and let $f \colon b\Omega \to \R$ be a smooth function that is bounded from below. Then there exists a smooth $q$-plurisubharmonic function $F$ defined on an open neighbourhood of $\overline{\Omega}$ such that $F|_{b\Omega} = f$ and $F$ is strictly $q$-plurisubharmonic near $b\Omega$.
\end{theorem}
\begin{proof}
Since $f$ is bounded from below, we can assume without loss of generality that $f>0$. Let $\widetilde{F} \colon \mathcal{M} \to (0,\infty)$ be a smooth extension of $f$. Choose open sets $U_j' \subset\subset U_j \subset\subset \mathcal{M}$ such that $\{U_j'\}_{j=1}^\infty$ covers $b\Omega$ and $\{U_j\}_{j=1}^\infty$ is locally finite. 

Let $\beta \colon (0,\infty) \to (0,\infty)$ be a strictly increasing and strictly convex smooth function such that $\beta(t) \coloneqq e^{-1/t}$ for small values of $t$, and let $\tilde{\beta} \colon \R \to [0, \infty)$ be the  smooth extension of $\beta$ such that $\tilde{\beta}|_{(-\infty,0]} \equiv 0$. We will construct a family $\{\chiup_j\}_{j=1}^\infty$ of smooth functions $\chiup_j \colon \mathcal{M} \to [0,1]$ such that $\{\chiup_j > 0\} = U_j'$, $\sum_{j=1}^\infty \chiup_j \le 1$ on $\mathcal{M}$, $\sum_{j=1}^\infty \chiup_j \equiv 1$ near $b\Omega$ and such that the trivial extension $g_j \colon \mathcal{M}  \to [0,\infty)$ of the function $\beta^{-1} \circ (\widetilde{F}\chiup_j) \colon U_j' \to \R$ by 0 is smooth on $\mathcal{M}$. For this purpose let $\delta_j \colon \mathcal{M} \to \R$ be smooth such that $U_j' = \{\delta_j > 0\}$ and $\mathcal{M} \setminus \overline{U_j'} = \{\delta_j < 0\}$ and let $\psi_j \coloneqq \tilde{\beta} \circ \tilde{\beta} \circ \delta_j$. Further, let $\theta \colon \mathcal{M} \to [0,\infty)$ be smooth such that $\theta > 0$ on $\mathcal{M} \setminus \bigcup_{j=1}^\infty U_j'$ and such that $\theta \equiv 0$ near $b\Omega$. Then choosing $\chiup_j \coloneqq \psi_j/\big(\theta + \sum_{k=1}^\infty\psi_k\big)$ and writing $\sigma \coloneqq \widetilde{F}/\big(\theta + \sum_{k=1}^\infty\psi_k\big)$ we get for points in $U_j'$ close to $bU_j'$ that $\beta^{-1} \circ(\widetilde{F}\chiup_j) = \beta^{-1}\big(\sigma \cdot(\beta \circ \beta \circ \delta_j)\big) = \big[-\log\big(\sigma \cdot (\beta \circ \beta \circ \delta_j)\big)\big]^{-1} = \big[-\log \sigma - \log(\beta \circ \beta \circ \delta_j)\big]^{-1} = \big[1/(\beta \circ \delta_j) - \log \sigma\big]^{-1} = (\beta \circ \delta_j)/\big(1 - (\beta \circ \delta_j)\log\sigma\big)$. Hence $\beta^{-1} \circ(\widetilde{F}\chi_j)$ extends smoothly to $\mathcal{M}$ by $0$, since $\beta$ extends smoothly to $\tilde{\beta}$. The other properties are clear from the construction.

Fix a hermitian metric $h$ on $\mathcal{M}$. Let $\varphi \colon V \to \R$ be a smooth function defined on an open neighbourhood $V \subset \mathcal{M}$ of $b\Omega$ such that $\Omega \cap V = \{\varphi < 0\}$ and $d\varphi \neq 0$ on $b\Omega$. By Lemma \ref{thm_uniformqpsh}, there exist an open neighbourhood $V' \subset V$ of $b\Omega$ and for every $z \in V'$ an $(n-q)$-dimensional complex subspace $L_z \subset T_z(\mathcal{M})$ with the following properties: for each $j \in \N$ there exist a number $c_j > 0$ and a strictly increasing strictly convex smooth function $\mu_j \colon \R \to \R$ satisfying $\mu_j(0) = 0$ such that the function $\varphi_j \coloneqq \mu_j \circ \varphi$ satisfies $\Lev(\varphi_j)(z,\xi) \ge c_j\norm{\xi}_{h_z}$ for every $z \in V' \cap U_j$ and $\xi \in L_z$. Moreover, without loss of generality, we can assume that $U_j \subset\subset V'$ for every $j \in \N$.

Fix $j \in \N$. Let $\lambda_j \colon \mathcal{M} \to (-\infty,0]$ be smooth such that $\overline{U_j'} = \{\lambda_j =0\}$. Then choose $\varepsilon_j > 0$ so small and $C_j > 0$ so large that $\Lev\big(g_j + C_j(\varphi_j + \varepsilon_j\lambda_j)\big)(z,\,\cdot\,)$ is positive definit on $L_z$ for every $z \in U_j$. Observe that, by construction, $g_j + C_j(\varphi_j + \varepsilon_j\lambda_j) < 0$ on $bU_j \cap \overline{\Omega}$, hence the function $\tilde{\beta} \circ (g_j + C_j(\varphi_j + \varepsilon_j\lambda_j))|_{U_j}$ vanishes near this set and thus its trivial extension by $0$ to the open neighbourhood $\mathcal{U}_j \coloneqq \mathcal{M} \setminus \{z \in bU_j : (g_j + C_j(\varphi_j + \varepsilon_j\lambda_j))(z) \ge 0\}$ of $\overline{\Omega}$ defines a smooth $q$-plurisubharmonic function $F_j \colon \mathcal{U}_j \to [0,\infty)$ such that $F_j|_{b\Omega} = f\chiup_j$ and $F_j \equiv 0$ outside $U_j$. Moreover, $W_j \coloneqq \{F_j > 0\} \subset U_j$ is an open neighbourhood of $b\Omega \cap U_j'$ such that $F_j$ is strictly $q$-plurisubharmonic on $W_j$. In particular, $\Lev(F_j)(z,\,\cdot\,) > 0$ on $L_z \setminus \{0\}$ for every $z \in W_j$ and $\Lev(F_j)(z,\,\cdot\,) \equiv 0$ if $z \notin W_j$.

Set $F \coloneqq \sum_{j=1}^\infty F_j$. Then $F$ is a well defined smooth function on the open neighbourhood $\mathcal{U} \coloneqq \bigcap_{j=1}^\infty \mathcal{U}_j \subset \mathcal{M}$ of $\overline{\Omega}$. By construction, $F|_{b\Omega} = f$. Moreover, $\Lev(F)(z,\,\cdot\,) > 0$ on $L_z \setminus \{0\}$ for every $z \in W \coloneqq \bigcup_{j=1}^\infty W_j \supset b\Omega$ and $\Lev(F)(z,\,\cdot\,) \equiv 0$ if $z \notin W$. Hence $F$ is a function as desired.
\end{proof}

\begin{theorem} \label{thm_qdeffct_manifolds}
Let $\mathcal{M}$ be a complex manifold of dimension $n$ and let $\Omega \subset \mathcal{M}$ be a strictly $q$-pseudo\-con\-vex domain with smooth boundary for some $q \in \{0, 1, \ldots, {n-1}\}$. Then there exists a smooth $q$-plu\-ri\-sub\-har\-monic function $\varphi$ defined on an open neighbourhood of $\overline{\Omega}$ such that $\Omega = \{\varphi < 0\}$, $d\varphi \neq 0$ on $b\Omega$ and $\varphi$ is strictly $q$-plurisubharmonic near $b\Omega$.
\end{theorem}
\begin{proof}
Let $\varphi \coloneqq F - 1$, where $F$ is the function from Theorem \ref{thm_qbdvalue_manifolds} corresponding to the boundary values $f \equiv 1$. Then $\varphi$ is a smooth $q$-plurisubharmonic function on an open neighbourhood of $\overline{\Omega}$ that vanishes identically on $b\Omega$ and that is strictly $q$-plurisubharmonic near $b\Omega$. Observe that in the construction of $F$ we can choose $\widetilde{F}$ such that $\Omega = \{\widetilde{F} < 1\}$ and $\mathcal{M} \setminus \overline{\Omega} = \{\widetilde{F} > 1\}$. Moreover, after possibly shrinking $\mathcal{U}$, we can assume that $\sum_{j=1}^\infty \chiup_j \equiv 1$ on $\mathcal{U} \setminus \overline{\Omega}$. For $z \in \mathcal{M}$ let $I(z) \coloneqq \{j \in \N : z \in W_j\}$ and $J(z) \coloneqq \{j \in \N : z \in U_j'\}$. Then
\begin{equation*} \begin {split} F(z) &= \sum_{j \in I(z)} F_j(z) = \sum_{j \in I(z)} (\tilde{\beta} \circ (g_j + C_j(\varphi_j + \varepsilon_j\lambda_j)))(z) \le \sum_{j \in I(z)} (\tilde{\beta} \circ g_j)(z) \\ &= \sum_{j \in I(z)} (\widetilde{F}\chiup_j)(z) \le \widetilde{F}(z) < 1 \;\; \text{ for } z \in \Omega \end{split} \end{equation*}
(here the sum over the empty index set is understood to be zero), and
\begin{equation*} \begin {split} F(z) &\ge \sum_{j \in J(z)} F_j(z) = \sum_{j \in J(z)} (\tilde{\beta} \circ (g_j + C_j(\varphi_j + \varepsilon_j\lambda_j)))(z) \ge \sum_{j \in J(z)} (\tilde{\beta} \circ g_j)(z) \\ &= \sum_{j \in J(z)} (\widetilde{F}\chiup_j)(z) = \widetilde{F}(z) > 1 \;\; \text{ for } z \in \mathcal{U} \setminus \overline{\Omega}. \end{split} \end{equation*}
This shows that $\Omega = \{F < 1\}$, i.e., $\Omega = \{\varphi < 0\}$. Finally, we have $d\varphi \neq 0$ on $b\Omega$, provided that the numbers $C_j$ which appear in the construction of $F$ are chosen large enough (in fact, since $b\Omega$ is smooth, the non-vanishing of $d\varphi$ along $b\Omega$ is automatically satisfied, see, for example, the proof of Proposition 1.5.16 in \cite{HenkinLeiterer84}, which can be adapted easily to the case of $q$-plurisubharmonic functions).
\end{proof}

\noindent \textbf{Remarks.} 1) The assumption in Theorem \ref{thm_qbdvalue_manifolds} that $f$ is bounded from below is crucial. In fact, it was shown in Example 8.2 of \cite{ShcherbinaTomassini99} that there exist an unbounded strictly pseudoconvex domain $\Omega \subset \C^2$ with smooth boundary and a smooth function $f \colon b\Omega \to \R$ that is not bounded from below such that the only plurisubharmonic function $F \colon \Omega \to [-\infty, \infty)$ satisfying $\limsup_{z \to z_0} F(z) \le f(z_0)$ for every $z_0 \in b\Omega$ is the function $F \equiv -\infty$.

\noindent 2) The function $F$ from Theorem \ref{thm_qbdvalue_manifolds} is strictly $q$-plurisubharmonic on the open neighbourhood $W$ of $b\Omega$ and $F$ is a constant on $\Omega \setminus W$. It is clear from the construction that for every open set $\omega \subset \Omega$ such that $\overline{\omega} \subset \Omega$ we can choose $F$ in such a way that $F$ is constant on $\omega$.

\noindent 3) Let $h$ be a hermitian metric on $\mathcal{M}$ and let $\nu, \mu \colon b\Omega \to (0,\infty)$ be positive continuous functions. Then $F$ can be chosen in such a way that $\norm{(dF)_z}_{h^\ast_z} \ge \nu(z)$ for every $z \in b\Omega$ and $\Lev(F)(z,\,\cdot\,) \ge \mu(z)\norm{\,\cdot\,}^2_{h_z}$ on $L_z$ for every $z \in b\Omega$. Indeed, for every $j \in \N$ let $U_j'' \subset\subset U_j'$ be an open set such that $\{U_j''\}_{j=1}^\infty$ still covers $b\Omega$. Now in the construction of $F$ we can choose for every $j \in \N$ the numbers $\varepsilon_j > 0$ so small and $C_j > 0$ so large that $(dF_j)_z(N_\Omega(z)) \ge 0$ for every $z \in b\Omega$, $(dF_j)_z(N_\Omega(z)) \ge \nu(z)$ for every $z \in b\Omega \cap U_j''$ and $\Lev(F_j)(z, \,\cdot\,) \ge \mu(z) \norm{\,\cdot\,}^2_{h_z}$ on $L_z$ for every $z \in b\Omega \cap U_j''$, where $N_\Omega(z)$ denotes the outward unit normal to $b\Omega$ at $z$ with respect to $h$. Then $F$ is a function as desired.
 
\noindent 4) The statements of Theorems \ref{thm_qbdvalue_manifolds} and \ref{thm_qdeffct_manifolds} as well as the above remarks remain true if $\mathcal{C}^\infty$-smoothness is replaced by $\mathcal{C}^s$-smoothness for $s \ge 2$. If $q = 0$ and if for each point $z \in b\Omega$ there exists an open neighbourhood $U \subset \mathcal{M}$ of $z$ and a $\mathcal{C}^1$-smooth strictly plurisubharmonic function $\varphi \colon U \to \R$ such that $\Omega \cap U = \{\varphi < 0\}$ and $d\varphi \neq 0$ on $b\Omega \cap U$ (note that this is a stronger assumption than $\Omega$ being strictly pseudoconvex with $\mathcal{C}^1$-smooth boundary), and if $f \colon b\Omega \to \R$ is $\mathcal{C}^2$-smooth (i.e., for every $z \in b\Omega$ there exists an open neighbourhood $U_z \subset \mathcal{M}$ of $z$ and a $\mathcal{C}^2$-smooth function $F_z \colon U_z \to \R$ such that $F_z$ coincides with $f$ on $b\Omega \cap U_z$), then a statement analoguous to Theorem \ref{thm_qbdvalue_manifolds} holds true with $\mathcal{C}^1$-smooth $F$. Further, if $\Omega$ is just strictly pseudoconvex (with no smoothness assumptions on $b\Omega$) and if $f \colon b\Omega \to \R$ is $\mathcal{C}^2$-smooth, then there always exists a continuous plurisubharmonic function $F$ as in Theorem \ref{thm_qbdvalue_manifolds}. Analoguous generalizations are possible for Theorem \ref{thm_qdeffct_manifolds} (but of course no assertion on the differential of $\varphi$ is imposed if $s=0$). Moreover, when considering the case of possibly nonsmooth boundaries, it is also worth mentioning that we do not need connectedness of the set $\Omega$ in the proofs of Theorem \ref{thm_qbdvalue_manifolds} and Theorem \ref{thm_qdeffct_manifolds}. In particular, every strictly pseudoconvex open set in a complex manifold admits a continuous global defining function.

\noindent 5) Finally, we want to mention without giving the details of the proof that it is possible to weaken the assumptions on smoothness of $b\Omega$ even further. Indeed, in Theorem \ref{thm_qbdvalue_manifolds} it suffices to assume that $\Omega$ can be represented locally near each boundary point as the sublevel set of a $\mathcal{C}^\infty$-smooth strictly $q$-plurisubharmonic function with possibly vanishing differential along $b\Omega$ (or, more general, as the sublevel set of a $\mathcal{C}^s$-smooth strictly $q$-plurisubharmonic function for some $s \ge 2$ or $(q,s) = (0,1)$, but then the function $F$ from Theorem \ref{thm_qbdvalue_manifolds} will only be $\mathcal{C}^s$-smooth in general). Domains of this type were considered, for example, in \cite{HenkinLeiterer84} and \cite{HenkinLeiterer88}. If $q=0$, this is clear. In the case $q > 0$ this is a consequence of the following fact: If $\Omega \subset \mathcal{M}$ is open, $z \in b\Omega$, $U \subset \mathcal{M}$ is an open neighbourhood of $z$ and $\varphi_1, \varphi_2 \colon U \to \R$ are $\mathcal{C}^2$-smooth functions such that $\Omega \cap U = \{\varphi_1 < 0\} = \{\varphi_2 < 0\}$, then for every $\xi \in H_z(\varphi_1) = H_z(\varphi_2)$ (see again Proposition 1.5.16 in \cite{HenkinLeiterer84} for the fact that $(d\varphi_1)_z = 0$ if and only if $(d\varphi_2)_z=0$) we have $\Lev(\varphi_1)(z,\xi) \ge 0$ if and only if $\Lev(\varphi_2)(z,\xi) \ge 0$ (see, for example, the proof of Proposition 15 in \cite{AndreottiGrauert62}). In particular, the sum $\varphi_1 + \varphi_2$ is strictly $q$-plurisubharmonic near $z$ if both $\varphi_1$ and $\varphi_2$ are strictly $q$-plurisubharmonic near $z$ and $(d\varphi_1)_z = (d\varphi_2)_z = 0$. Thus, if $\Sigma(b\Omega)$ denotes the set of points $z \in b\Omega$ such that $b\Omega$ is smooth in $z$, then the function $F = \sum_{j=1}^\infty F_j$ that appears in the proof of Theorem \ref{thm_qbdvalue_manifolds} will be automatically strictly $q$-plurisubharmonic in a neighbourhood of $b\Omega \setminus \Sigma(b\Omega)$, and strict $q$-plurisubharmonicity near the remaining part of $b\Omega$ can be achieved as before. The same weakening of assumptions is possible in Theorem \ref{thm_qdeffct_manifolds}, but then the constructed function $\varphi$ can be guaranteed to have nonvanishing differential only along $\Sigma(b\Omega)$. %Remark 4 still applies to this more general situation, only in Remark 3 one has to observe that the size of $\norm{(d\rho)_z}_{h^\ast}$ cannot be controlled arbitrarily near $b\Omega \setminus \Sigma(b\Omega)$.

Our next goal is to show that the core is the only obstruction for strict plurisubharmonicity of global defining functions, i.e., we want to construct a global defining function that is strictly plurisubharmonic precisely in the complement of $\mathfrak{c}(\Omega)$ (in particular, we now work in the case $q=0$). This will give a stronger version of the statement of Theorem 2.2, namely, the Main Theorem (see page \pageref{thm_maintheorem} in the Introduction). For this we need an auxiliary lemma which uses the following notion of smooth maximum: Let $\delta > 0$ and let $\chi_\delta \colon \R \to \R$ be a smooth function such that $\chi$ is strictly convex for $\abs{t} < \delta/2$ and $\chi_\delta(t) = \abs{t}$ for $\abs{t} \ge \delta/2$. Then we define a smooth maximum by
\[ \label{def_smoothmax} \widetilde{\max}_\delta(x,y) \coloneqq \frac{x+y + \chi_{\delta}(x-y)}{2}. \]
Observe that the smooth maximum of two smooth (strictly) plurisubharmonic functions is again a smooth (strictly) plurisubharmonic function (see, for example, Corollary 4.14 in \cite{HenkinLeiterer88}). Moreover, $\widetilde{\max}_\delta(x,y) = \max(x,y)$ if $\abs{x-y} \ge \delta$.

\begin{lemma} \label{thm_coreequality}
Let $\Omega$ be a strictly pseudoconvex domain with smooth boundary in a complex manifold $\mathcal{M}$. Let $\mathfrak{c}_\ast(\Omega)$ denote the set of all points in $\Omega$ where every smooth global defining function for $\Omega$ fails to be strictly plurisubharmonic. Then $\mathfrak{c}_\ast(\Omega) = \mathfrak{c}(\Omega)$.
\end{lemma}
\begin{proof}
It is obvious that $\mathfrak{c}(\Omega) \subset \mathfrak{c}_\ast(\Omega)$. Assume, to get a contradiction, that $\mathfrak{c}_\ast(\Omega) \setminus \mathfrak{c}(\Omega) \neq \varnothing$, i.e., there exist $p \in \mathfrak{c}_\ast(\Omega)$ and a bounded from above smooth plurisubharmonic function $\varphi_1$ on $\Omega$ such that $\varphi_1$ is strictly plurisubharmonic at $p$. Let $\varphi_2 \colon \overline{\Omega} \to \R$ be a smooth global defining function for $\Omega$. Choose constants $C_1, C_2 > 0$ such that $\varphi_1 - C_1 < C_2\varphi_2 - 1$ near $b\Omega$ and $C_2 \varphi_2(p) < \varphi_1(p) - C_1 - 1$. Then $\widetilde{\max}_1(\varphi_1 - C_1, C_2\varphi_2)$ is a smooth global defining function for $\Omega$ that is strictly plurisubharmonic in $p$, which contradicts the fact that $p \in \mathfrak{c}_\ast(\Omega)$.
\end{proof}

\noindent \textbf{Main Theorem.} Every strictly pseudoconvex domain $\Omega$ with smooth boundary in a complex manifold $ \mathcal{M}$ admits a bounded global defining function that is strictly plurisubharmonic outside $\mathfrak{c}(\Omega)$.
\begin{proof}
By the previous lemma, for every $p \in \Omega \setminus \mathfrak{c}(\Omega)$ there exists a smooth global defining function $\psi_p$ for $\Omega$ that is strictly plurisubharmonic on an open neighbourhood $V_p \subset\subset \Omega \setminus \mathfrak{c}(\Omega)$ of $p$. Let $\{p_j\}_{j=1}^\infty$ be a sequence of points $p_j \in \Omega$ such that $\bigcup_{j=1}^\infty V_{p_j} = \Omega \setminus \mathfrak{c}(\Omega)$. Without loss of generality we can assume that each set $V_{p_j}$ is contained in some coordinate patch of $\mathcal{M}$. Choose a sequence $\{\delta_j\}_{j=1}^\infty$ of positive numbers $\delta_j$ such that $\delta_j\psi_{p_j} > -1/2$ on $V_{p_j}$ for every $j \in \N$. Moreover, let $\{\varepsilon_j\}_{j=1}^\infty$ be a second sequence of suitably chosen positive numbers $\varepsilon_j$ and define $\varphi_1 \coloneqq \sum_{j=1}^\infty \varepsilon_j \widetilde{\max}_{1/2}(\delta_j\psi_{p_j}, -1)$. If $\{\varepsilon_j\}_{j=1}^\infty$ converges to zero fast enough, then $\varphi_1$ is a smooth plurisubharmonic function on $\overline{\Omega}$ such that $\varphi_1$ is strictly plurisubharmonic outside $\mathfrak{c}(\Omega)$, $b\Omega = \{\varphi_1 = 0\}$ and $0 \ge \varphi_1 > -1$. By construction, $\varphi_1$ has nonvanishing differential along $b\Omega$, hence a smooth extension $\varphi$ of $\varphi_1$ to a small enough open neighbourhood $\mathcal{U} \subset \mathcal{M}$ of $\overline{\Omega}$ will be a global defining function as desired.
\end{proof}

\noindent \textbf{Remarks.} 1) In the same way as described in the remarks after Theorem \ref{thm_qdeffct_manifolds}, we can prescribe along $b\Omega$ the size of the differential and the Levi form of the global defining function constructed in the Main Theorem.

\noindent \label{def_coreweaklysmooth} 2) As for the case of $\mathcal{C}^\infty$-smooth functions, we can define the sets
\begin{equation*}\begin{split}
\mathfrak{c}^s(\Omega) \coloneqq \big\{&z \in \Omega :\text{every $\mathcal{C}^s$-smooth plurisubharmonic function on $\Omega$ that is} \\& \,\text{bounded from above fails to be strictly plurisubharmonic in } z \big\}
\end{split}\end{equation*}
for every $s \in \N_0^\infty \coloneqq \{0\} \cup \N \cup \{\infty\}$. Then statements analoguous to Lemma \ref{thm_coreequality} and the Main Theorem hold for every $s \in \N_0^\infty$. Observe, however, that it is not clear whether in general $\mathfrak{c}^{s_1}(\Omega) = \mathfrak{c}^{s_2}(\Omega)$ for $s_1 \neq s_2$.

\noindent 3) One can also define yet another version of the core as
\begin{equation*}\begin{split}
\tilde{\mathfrak{c}}(\Omega) \coloneqq \big\{&z \in \Omega :\text{every plurisubharmonic function on $\Omega$ that is bounded} \\& \text{from above and not identically $-\infty$ on any connected compo-} \\& \text{nent of } \Omega \text{ fails to be strictly plurisubharmonic in } z \big\}.
%\tilde{\mathfrak{c}}(\Omega) \coloneqq \big\{&z \in \Omega :\text{every plurisubharmonic function on $\Omega$ that is bounded from} \\& \text{above and not identically $-\infty$ fails to be strictly plurisubharmonic in } z \big\}.
\end{split}\end{equation*}
Observe that this definition leads to a weaker notion, i.e., in general we have $\tilde{\mathfrak{c}}(\Omega) \subsetneq \mathfrak{c}(\Omega)$. For example, the function $\varphi(z,w) \coloneqq \log\abs{w-f(z)} + C_1(\abs{z}^2 + \abs{w}^2)$ is strictly plurisubharmonic and bounded from above on the domain $\Omega$ from Example \ref{ex_complexline}. Hence in this case we have $\tilde{\mathfrak{c}}(\Omega) = \varnothing$, but $\mathfrak{c}(\Omega) \neq \varnothing$. We do not know if there exists a complex manifold $\mathcal{M}$ and a strictly pseudoconvex domain $\Omega \subset \mathcal{M}$ such that $\tilde{\mathfrak{c}}(\Omega) \neq \varnothing$.

\noindent 4) A result analoguous to the Main Theorem holds also true if $b\Omega$ is only smooth in the weaker sense as it is described in Remark 5 after Theorem \ref{thm_qdeffct_manifolds}. Indeed, to extend $\varphi_1$ from $\overline{\Omega}$ to an open neighbourhood of $\overline{\Omega}$ let then $\varphi_2$ be a global defining function for $\Omega$ as constructed in Theorem \ref{thm_qdeffct_manifolds}. In particular, $\varphi_2$ is defined on an open neighbourhood $U$ of $\overline{\Omega}$, $\varphi_2 \ge -1$ and $\varphi_2$ is strictly plurisubharmonic on $\{\varphi_2 > -1\}$. Then
\[\varphi(z) \coloneqq \left\{ \begin{array}{c@{\,,\quad}l} 2\varphi_2(z) & z \in U \setminus \overline{\Omega} \\ \widetilde{\max}_{1/4}(\varphi_1(z) - 1/2, 2\varphi_2(z)) & z \in \overline{\Omega} \cap \{\varphi_2 > -1\} \\ \varphi_1(z) -1/2 & z \in \overline{\Omega} \cap \{\varphi_2 = -1\} \end{array} \right.\]
is a function as desired.

Following \cite{SlodkowskiTomassini04}, we introduce the following notion of minimal functions for a domain $\Omega \subset \mathcal{M}$.
\begin{definition} \label{def_minimal}
Let $\mathcal{M}$ be a complex manifold and let $\Omega \subset \mathcal{M}$ be a domain. A smooth and bounded from above plurisubharmonic function $\varphi \colon \Omega \to \R$ will be called {\it minimal} if $\varphi$ is strictly plurisubharmonic outside $\mathfrak{c}(\Omega)$. 
\end{definition}

\noindent Our Main Theorem can then be rephrased as follows: every strictly pseudoconvex domain with smooth boundary in a complex manifold admits a bounded minimal global defining function. Moreover, by using similar arguments as in the proof of the Main Theorem, it also follows that every domain in a complex manifold admits a bounded minimal function.

As in the case of plurisubharmonic functions, it now would also be possible to introduce for every domain $\Omega$ in a complex manifold $\mathcal{M}$ the core $\mathfrak{c}(\Omega,q)$ with respect to the class of $q$-plurisubharmonic functions, namely,
\begin{equation*}\begin{split}
\mathfrak{c}(\Omega,q) \coloneqq \big\{&z \in \Omega :\text{every smooth $q$-plurisubharmonic function on $\Omega$ that is} \\& \text{bounded from above fails to be strictly $q$-plurisubharmonic in } z \big\}.
\end{split}\end{equation*}

However, we do not know whether this definition is meaningful, in the sense that we do not have any examples of domains $\Omega \subset \mathcal{M}$ such that $\mathfrak{c}(\Omega, q) \neq \varnothing$ for $q > 0$. Indeed, for domains in Stein manifolds the set $\mathfrak{c}(\Omega, q)$ is always empty for every $q > 0$ as it is shown in the following proposition.

\begin{proposition} \label{thm_emptycore}
Every Stein manifold $\mathcal{M}$ admits a bounded smooth $1$-plurisubhar\-monic function. In particular, $\mathfrak{c}(\Omega,q) = \varnothing$ for every $\Omega \subset \mathcal{M}$ and every $q > 0$.
\end{proposition} 
\begin{proof}
Let $\psi \colon \mathcal{M} \to \R$ be  a smooth strictly plurisubharmonic function. After replacing $\psi$ by $e^\psi$ if necessary, we can assume without loss of generality that $\psi \ge 0$. Define $\chiup \colon (-1, \infty) \to \R$ as $\chiup(t) \coloneqq -1/(1+t)$ and consider the bounded smooth function $\varphi \coloneqq \chiup \circ \psi$. Then
\[ \Lev(\varphi)(z,\xi) = \chiup''(\psi(z))\abs{(\partial \psi)_z(\xi)}^2 + \chiup'(\psi(z))\Lev(\psi)(z,\xi) \]
for every $z \in \mathcal{M}$ and $\xi \in T_z(\mathcal{M})$. In particular, $\Lev(\varphi)(z,\,\cdot\,) > 0$ on the at least $(\dim_\C \mathcal{M}-1)$-dimensional subspace $H_z(\psi) = \{\xi \in \C^n : (\partial\psi)_z(\xi) = 0\}$.
\end{proof}

One might expect that at least compact analytic subsets $A \subset \Omega$ of pure dimension $q+1$ are always contained in $\mathfrak{c}(\Omega,q)$. However, this is not necessarily the case as it is shown by the following example.

\begin{example}
As in Example \ref{ex_compactanalytic}, let $\mathcal{M} \coloneqq \{(z,x) \in \C^3 \times \CP^2 : z_ix_j = z_jx_i, \,i,j= 0,1,2\}$ be the blow-up of $\C^3$ at the origin. For every $j = 0,1,2$, define mappings $h_j \colon U_j \to \C^3$ on the dense open subsets $U_j \coloneqq \{(z,x) \in \mathcal{M} : x_j \neq 0\}$ as 
\[ h_j(z,x) \coloneqq \Big(\frac{x_0}{x_j}, \ldots, \frac{x_{j-1}}{x_j}, z_j, \frac{x_{j+1}}{x_j}, \ldots, \frac{x_2}{x_j}\Big). \]
Each map $h_j$ is a homeomorphism with inverse 
\begin{equation*} \begin{split}  h_j^{-1}(w_0,w_1,w_2) \coloneqq \big((w_jw_0, \ldots, w_jw_{j-1}, w_j, &w_jw_{j+1}, \ldots, w_jw_2), \\ &[w_0 : \ldots : w_{j-1} : 1 : w_{j+1} : \ldots : w_2]\big) \end{split} \end{equation*}
and the tupel $\{(U_j,h_j) : j = 0,1,2\}$ defines a complex structure on $\mathcal{M}$. For every $j = 0,1,2$, define a smooth function $\varphi_j \colon \mathcal{M} \to \R$ as
\[ \varphi_j(z,x) \coloneqq -\frac{\abs{x_j}^2}{\abs{z_j}^2\abs{x_j}^2 + \abs{x_0}^2 + \abs{x_1}^2 + \abs{x_2}^2}. \]
Then
\[ (\varphi_j \circ h_j^{-1})(w_0,w_1,w_2) = - \frac{1}{1 + \abs{w_0}^2 + \abs{w_1}^2 + \abs{w_2}^2}, \]
and as in the proof of Proposition \ref{thm_emptycore} we see that this function is strictly $1$-plurisubhar\-monic on $\C^3$. Hence $\varphi_j$ is $1$-plurisubharmonic on $\mathcal{M}$ and strictly $1$-plurisubharmonic on $U_j$. Now let $\Omega \subset\subset \mathcal{M}$ be the strictly pseudoconvex domain with smooth boundary defined by
\[ \Omega \coloneqq \big\{(z,x) \in \mathcal{M} : \norm{z} < 1 \big\}. \]
Then the above computations show that for every $(z,x) \in \Omega$ there exists a smooth $1$-plurisubharmonic function on $\Omega$ that is bounded from above and that is strictly $1$-plurisubharmonic near $(z,x)$, i.e., $\mathfrak{c}(\Omega,1) = \varnothing$. In particular, the pure $2$-dimensional compact analytic set $\{0\} \times \CP^2 \subset\Omega$ is not contained in $\mathfrak{c}(\Omega,1)$.
\end{example}

Observe also that in general no analogue of the Main Theorem holds true in the case of strictly $q$-pseudoconvex domains $\Omega$ if $q>0$, i.e., in general it is not possible to have a global defining function for $\Omega$ as in Theorem \ref{thm_qdeffct_manifolds} that is strictly $q$-plurisubharmonic outside $\mathfrak{c}(\Omega, q)$. Indeed, the domain $\Omega$ from the last example satisfies $\mathfrak{c}(\Omega,1) = \varnothing$, but there exists no smooth strictly $1$-plurisubharmonic function on $\Omega$, since there exists no such function on $\CP^2$.

We now give one more application of the constructions that were carried out in the proof of Theorem \ref{thm_qbdvalue_manifolds}.

\begin{theorem} \label{thm_nbhbasis_manifolds}
Let $\mathcal{M}$ be a complex manifold and let $\Omega \subset \mathcal{M}$ be a strictly pseudoconvex open set (not necessarily relatively compact or with smooth boundary). Let $U \subset \mathcal{M}$ be an arbitrary open neighbourhood of $b \Omega$. Then the following assertions hold true:
\begin{enumerate}
  \item[(1)] There exists a strictly pseudoconvex open set $\Omega' \subset \mathcal{M}$ with smooth boundary such that $\Omega \setminus U \subset \Omega'$, $\overline{\Omega'} \subset \Omega$ and $\mathfrak{c}(\Omega') = \mathfrak{c}(\Omega)$.
  \item[(2)] There exists a strictly pseudoconvex open set $\Omega'' \subset \mathcal{M}$ with smooth boundary such that $\overline{\Omega} \subset \Omega''$, $\overline{\Omega''} \subset \Omega \cup U$ and $\mathfrak{c}(\Omega'') = \mathfrak{c}(\Omega)$.
\end{enumerate}
In particular, $\overline{\Omega}$ admits a neighbourhood basis consisting of strictly pseudoconvex open sets with smooth boundary. Moreover, if $\Omega$ is a domain, then one can also choose $\Omega'$ and $\Omega''$ to be domains.
\end{theorem} 
\begin{proof}
Fix an open neighbourhood $U \subset \mathcal{M}$ of $b\Omega$.

\noindent (1) We first show the existence of the strictly pseudoconvex set $\Omega'$. Let $\omega \subset \Omega$ be an arbitrary but fixed open set such that $\Omega \setminus U \subset \omega$ and $\overline{\omega} \subset \Omega$. By Theorem \ref{thm_qdeffct_manifolds} and the related Remarks 2 and 4, there exists a continuous plurisubharmonic function $\varphi$ defined near $\overline{\Omega}$ such that $\Omega = \{\varphi < 0\}$, $\varphi \ge -1$, $\varphi \equiv -1$ on $\omega$ and $\varphi$ is strictly plurisubharmonic on $\{\varphi > -1\}$. Applying Richberg's smoothing procedure (see, for example, Theorem I.5.21 in \cite{DemaillyXX}), we can then find a continuous plurisubharmonic function $\widetilde{\varphi}$ defined near $\overline{\Omega}$ such that $\widetilde{\varphi} \ge \varphi$, $\widetilde{\varphi} \equiv -1$ on $\omega$, $\widetilde{\varphi}$ is smooth and strictly plurisubharmonic on $\{\widetilde{\varphi} > -1\}$, and $\abs{\widetilde{\varphi} - \varphi} < 1/2$. Let $c \in (-1,-1/2)$ be a regular value of $\widetilde{\varphi}$ and set $\Omega' \coloneqq \{\widetilde{\varphi} < c\}$. Then $\Omega'$ is a strictly pseudoconvex open set such that $\Omega \setminus U \subset \Omega'$ and $\overline{\Omega'} \subset \Omega$.

It remains to show that $\mathfrak{c}(\Omega') = \mathfrak{c}(\Omega)$. Since $\Omega' \subset \Omega$, it follows immediately that $\mathfrak{c}(\Omega') \subset \mathfrak{c}(\Omega)$. On the other hand, observe that for small enough $\delta > 0$ the function $\varphi_2 \coloneqq \widetilde{\max}_\delta(\widetilde{\varphi}-c,-(c+1)/2)$ is smooth plurisubharmonic and bounded from above on $\Omega$, strictly plurisubharmonic near $\Omega \setminus \Omega'$ and $\Omega' = \{\varphi_2 < 0\}$. In particular, this shows that $\mathfrak{c}(\Omega) \subset \Omega'$. By repeating the same arguments as in the proof of Lemma \ref{thm_coreequality}, it now follows easily that $\mathfrak{c}(\Omega) \subset \mathfrak{c}(\Omega')$.

(2) We now show the existence of the strictly pseudoconvex set $\Omega''$. After possibly shrinking $U$, let $\varphi$ be a continuous plurisubharmonic function defined on a neighbourhood of $\overline{\Omega \cup U}$ such that $\Omega = \{\varphi < 0\}$, $\varphi \ge -1$, $\varphi > -1/2$ on $U$ and $\varphi$ is strictly plurisubharmonic on $\{\varphi > -1\}$. Without loss of generality we can assume that $\varphi > 0$ outside $\overline{\Omega}$ (in fact, the function $\varphi$ from Theorem \ref{thm_qdeffct_manifolds} has this property by construction).

We claim that there exists a strictly pseudoconvex open set $\widetilde{\Omega}'' \subset \mathcal{M}$ (not necessarily with smooth boundary) such that $\overline{\Omega} \subset \widetilde{\Omega}'' \subset \Omega \cup U$. The proof is essentially the same as in Lemma 2 of \cite{Tomassini83}: Choose a locally finite open covering $\{U_j\}_{j=1}^\infty$ of $b\Omega$ by open sets $U_j \subset\subset U$. For every $j \in \N$, let $\eta_j \colon \mathcal{M} \to (-\infty, 0]$ be a smooth function such that $\{\eta_j < 0\} = U_j$. Set $\phi \coloneqq \varphi + \sum_{j=1}^\infty \varepsilon_j\eta_j$ with positive constants $\varepsilon_j$, $j \in \N$. Clearly, $\phi > 0$ on $b(\Omega \cup U)$, $\phi < 0$ on $\overline{\Omega}$ and if the numbers $\varepsilon_j$ are chosen small enough, then $\phi$ is still strictly plurisubharmonic on $U$. Set $\widetilde{\Omega}'' \coloneqq \{\phi < 0\}$.

Note that, by a suitable choice of the numbers $\varepsilon_j$, we can also guarantee that $\mathfrak{c}(\widetilde{\Omega}'') = \mathfrak{c}(\Omega)$. Indeed, since $\Omega \subset \widetilde{\Omega}''$, it is immediately clear that $\mathfrak{c}(\Omega) \subset \mathfrak{c}(\widetilde{\Omega}'')$. Further, observe that in the construction of $\phi$ we can choose the numbers $\varepsilon_j$ so small that $\phi > -1/2$ on $U$. Thus we can use the same smoothing procedure as in part $(1)$ (choosing $c = -1/2$) to obtain a smooth and bounded from above plurisubharmonic function $\phi_2 \colon \widetilde{\Omega}'' \to [-1/4,\infty)$ such that $\phi_2 > 0$ on $U$ and $\phi_2$ is strictly plurisubharmonic on $\{\phi_2 > -1/4\}$. Then, as before, the same argument as in the proof of Lemma \ref{thm_coreequality} shows that $\mathfrak{c}(\widetilde{\Omega}'') \subset \mathfrak{c}(\Omega)$.

Now we can apply part 1 of the theorem to the strictly pseudoconvex set $\widetilde{\Omega}''$ and an open neighbourhood of $b\widetilde{\Omega}''$ that does not intersect $\Omega$ to obtain a set $\Omega''$ as desired. This completes the proof of $(2)$.

The last two properties claimed in the theorem are obvious by the construction.
\end{proof}

\noindent \textbf{Remark.} A similar result holds true for strictly $q$-pseudoconvex open sets (or domains) $\Omega \subset \mathcal{M}$. However, in the case $q > 0$ the regularity of $b\Omega'$ and $b\Omega''$ will in general be only as good as the regularity of $b\Omega$, since it is not always possible to make a $q$-plurisubharmonic smoothing of a $q$-plurisubharmonic function, see \cite{DiederichFornaess85}. Moreover, in the general situation we do not make any claim about the cores of the sets $\Omega'$ and $\Omega''$.

At the end of this section we want to prove again, but in a different way, the existence of global defining functions for strictly pseudoconvex domains $\Omega$ with $\mathcal{C}^\infty$-smooth boundary. We first prove the existence of defining functions for $\Omega$ that have prescribed differentials along the boundary of $\Omega$. In a next step we use this result to construct a global defining function for $\Omega$.

\begin{lemma} \label{thm_deffctgradnorm}
Let $\Omega$ be a strictly pseudoconvex domain with smooth boundary in a complex manifold $\mathcal{M}$. Let $h$ be a hermitian metric on $\mathcal{M}$ and let $f \colon b\Omega \to (0,\infty)$ be a smooth positive function. Then there exists an open neighbourhood $V \subset \mathcal{M}$ of $b\Omega$ and a smooth strictly plurisubharmonic function $\psi \colon V \to \R$ such that $\Omega \cap V = \{\psi<0\}$ and $\norm{d \psi}_{h^\ast} = f$ on $b\Omega$.
\end{lemma}
\begin{proof}
Let $\rho \colon V \to \R$ be a smooth function on an open neighbourhood $V \subset \mathcal{M}$ of $b\Omega$ such that $\Omega \cap V = \{\rho < 0\}$ and $d\rho \neq 0$ on $b\Omega$. Let $q \colon b\Omega \to (0,\infty)$ be a positive smooth function that we consider to be fixed, but that will be further specified later on. Choose smooth extensions $F \colon V \to (0,\infty)$ of $f / \norm{d\rho}_{h^\ast} \colon b\Omega \to (0,\infty)$ and $Q \colon V \to (0,\infty)$ of $q \colon b\Omega \to (0,\infty)$, respectively, and define $\psi \colon V \to \R$ as 
\[ \psi(z) \coloneqq F(z)\rho(z) + Q(z)\rho(z)^2. \]
Then $\psi$ is smooth, $\norm{d\psi(z)}_{h^\ast} = f(z)$ for every $z \in b\Omega$ and, after possibly shrinking $V$, $\Omega \cap V = \{\psi<0\}$. By smoothness of $\psi$, it only remains to show that $\psi$ is strictly plurisubharmonic at every point $z \in b\Omega$. We claim that this is always the case, provided that the function $q$ is chosen large enough (observe that the Levi form of $F \cdot \rho$ in $z \in b\Omega$ is automatically positive definit on the complex tangent space $T_z^\C(b\Omega)$ of $b\Omega$ in $z$, since $\Omega$ is strictly pseudoconvex).

Indeed, a straightforward calculation shows that for every $z \in b\Omega$ and $\xi \in T_z(\mathcal{M})$
\begin{equation} \label{equ_deffctgrad} \Lev(\psi)(z,\xi) = F(z)\Lev(\rho)(z,\xi) + 2\Re[(\partial \rho)_z(\xi) \cdot (\overline{\partial}F)_z(\xi)] + 2q(z)\abs{(\partial \rho)_z(\xi)}^2. \end{equation}
Since $T_z^\C(b\Omega) = H_z(\rho)$, we have $\Lev(\psi)(z, \,\cdot\,) = F(z)\Lev(\rho)(z, \,\cdot\,)$ on $T_z^\C(b\Omega)$, and by strict pseudoconvexity of $\Omega$ we know that $\Lev(\rho)(z, \,\cdot\,)$ is positive definit on $T_z^\C(b\Omega)$ for every $z \in b\Omega$. Let $K \coloneqq \{(z,\xi) \in T(\mathcal{M})|_{b\Omega} : \norm{\xi}_{h_z} = 1\}$ and define $K_0 \subset K$ to be the subset $K_0 \coloneqq \{(z,\xi) \in K : F(z)\Lev(\rho)(z,\xi) + 2\Re[(\partial \rho)_z(\xi) \cdot (\overline{\partial}F)_z(\xi)] \le 0 \}$. Since $\rho$ and $F$ are smooth, we can choose a smooth function $C \colon b\Omega \to \R$ such that $F(z)\Lev(\rho)(z,\xi) + 2\Re[(\partial \rho)_z(\xi) \cdot (\overline{\partial}F)_z(\xi)] > C(z)$ for every $(z,\xi) \in K$. Moreover, observe that, by construction, $(\partial \rho)_z(\xi) \neq 0$ for every $(z,\xi) \in K_0$. Hence we can further choose a positive smooth function $\varepsilon \colon b\Omega \to (0,\infty)$ such that $\abs{(\partial \rho)_z(\xi)}^2 > \varepsilon(z)$ for every $(z,\xi) \in K_0$. Now assume that $q \colon b\Omega \to (0,\infty)$ is chosen so large that $C + 2q\varepsilon > 0$ on $b\Omega$. Then we conclude from $(\ref{equ_deffctgrad})$ and the choice of $C$ that $\Lev(\psi)(z,\xi) > 0$ on $K_0$. But it is clear from the choice of $K_0$ that $\Lev(\psi)(z,\xi) > 0$ on $K \setminus K_0$. Thus $\psi$ is strictly plurisubharmonic at every point $z \in b\Omega$ as claimed.
\end{proof}

\setcounter{theorem_prime}{1}
\begin{theorem_prime} \label{thm_manifolds_deffct_prime}
Let $\mathcal{M}$ be a complex manifold and let $\Omega \subset \mathcal{M}$ be a strictly pseudoconvex domain with smooth boundary. Then there exists a smooth plurisubharmonic function $\varphi$ defined on an open neighbourhood of $\overline{\Omega}$ such that $\Omega = \{\varphi < 0\}$, $d\varphi \neq 0$ on $b\Omega$ and $\varphi$ is strictly plurisubharmonic near $b\Omega$.
\end{theorem_prime}
\begin{proof}
As in Theorem 5 of \cite{SimioniucTomassini08} we can choose a countable locally finite covering $\{U_j\}_{j=1}^\infty$ of $b\Omega$ by open subsets $U_j \subset\subset \mathcal{M}$ such that there exist biholomorphisms $\phi_j \colon U_j \to U_j'$ onto open subsets $U_j' \subset \C^n$, strictly convex bounded domains $G_j' \subset \subset U_j'$ with smooth boundaries and a smooth partition of unity $\{\theta_j\}_{j=1}^\infty$ on $b\Omega$ subordinated to $\{b\Omega \cap U_j\}_{j=1}^\infty$ such that $G_j' \subset \phi_j(\Omega \cap U_j)$ and $\supp \theta_j' \subset \subset bG_j' \cap \phi_j(b\Omega \cap U_j)$, where $\theta_j' \coloneqq \theta_j \circ \phi_j^{-1}$ on $\phi_j(b\Omega \cap U_j)$. Moreover, by strict pseudoconvexity of $b\Omega$, we can assume that $\overline{\Omega} \cap \bigcup_{j=1}^\infty U_j$ is contained in a one-sided neighbourhood $U \subset \mathcal{M}$ of $b\Omega$ that is filled with analytic discs attached to $b\Omega$. For every $j \in \N$, let $g_j' \colon bG_j' \to [0,1]$ be the smooth extension of $\theta_j' \colon bG_j' \cap \phi_j(b\Omega \cap U_j) \to [0,1]$ by $0$, let $S_j' \coloneqq \supp \theta_j' = \supp g_j' \subset bG_j'$ and let $Z_j' \coloneqq bG_j' \setminus S_j'$. Let $f_j' \colon \overline{G_j'} \to (-\infty,1]$ be the strictly plurisubharmonic solution of the following Dirichlet problem for the complex Monge-Amp\`ere equation,
\[ \left\{\begin{array}{l} f_j'|_{bG_j'} = g_j' \\ \MA[f_j'] \equiv 1 \end{array} \right.. \]
The existence and uniqueness as well as smoothness of $f_j'$ is guaranteed by Theorem 1.1 in \cite{CaffarelliKohnNirenbergSpruck85}. Observe that, by strict convexity of $bG_j'$, the set $D_j' \coloneqq \{z' \in \overline{G_j'} : \text{there exists a complex line } L_{z'} \ni z' \text{ such that } L_{z'} \cap S_j' = \varnothing \}$ is an open neighbourhood of $Z_j'$ in $\overline{G_j'}$. By the maximum principle, we have $f_j' \le 0$ on $D_j'$. Hence the function $\tilde{f}_j' \coloneqq \max(0,f_j')$ satisfies $\tilde{f}_j' \equiv 0$ on $D_j'$. For every $j \in \N$, let $X_j \subset b\Omega$ be an open set such that $X_j \subset\subset \{\theta_j > 0\}$ and such that $\{X_j\}_{j=1}^\infty$ covers $b\Omega$. Further, let $W_j'$, $j \in \N$, be an open neighbourhood of $X_j' \coloneqq \phi_j(X_j) \subset\subset \{\theta_j' > 0\}$ in $\overline{G_j'}$ such that $f_j' > c_j > 0$ on $W_j'$ for some $c_j > 0$. Then $\tilde{f}_j'$ is strictly plurisubharmonic on a relatively open neighbourhood of $\overline{W_j'}$ in $\overline{G_j'}$.

Fix $j \in \N$. Without loss of generality we can assume that $0 \in G_j'$. In particular, there exists $\varepsilon_{j,1} > 0$ such that $G_{j,\varepsilon_j}' \coloneqq (1+\varepsilon_j)G_j'^{(-\varepsilon_j^2)}$ satisfies $G_j' \subset\subset G_{j,\varepsilon_j}' \subset\subset U_j'$ for every positive $\varepsilon_j \le \varepsilon_{j,1}$, where $G_j'^{(-\varepsilon_j^2)} \coloneqq G_j' \setminus \bigcup_{z' \in bG_j'} B^n(z',\varepsilon_j^2)$. Define a smooth plurisubharmonic function $\tilde{f}_{j,\varepsilon_j}' \colon G_{j,\varepsilon_j}' \to [0,1]$ by $\tilde{f}_{j,\varepsilon_j}'(z') \coloneqq \big(\tilde{f}_j' \ast \delta_{\varepsilon_j^2}\big)(z'/(1+\varepsilon_j))$, where for $\gamma>0$ we denote by $\delta_\gamma$ some fixed smooth nonnegative function depending only on $\norm{z}$ such that $\supp \delta_\gamma = \overline{B^n(0,\gamma)}$ and $\int_{\C^n} \delta_\gamma = 1$. It follows from the constructions of $G_{j,\varepsilon_j}'$ and $\tilde{f}_{j,\varepsilon_j}'$ that there exists $\varepsilon_{j,2} > 0$ such that for every positive $\varepsilon_j \le \varepsilon_{j,2}$ the set $D_{j,\varepsilon_j}' \coloneqq (1+\varepsilon_j)D_j'^{(-\varepsilon_j^2)} \subset \C^n$ is an open neighbourhood of $bG_j' \setminus \phi_j(b\Omega \cap U_j)$ and $\tilde{f}_{j,\varepsilon_j}' \equiv 0$ on $D_{j,\varepsilon_j}'$. In particular, the trivial extension of $\tilde{f}_{j,\varepsilon_j}' \circ \phi_j \colon \overline{G}_j \to [0,1]$ by $0$ defines a smooth plurisubharmonic function $F_{j,\varepsilon_j} \colon \overline{\Omega} \to [0,1]$, where $G_j \coloneqq \phi_j^{-1}(G_j')$. Moreover, there exists $\varepsilon_{j,3} > 0$ such that for every $\varepsilon_j \le \varepsilon_{j,3}$ the function $\tilde{f}_{j,\varepsilon_j}'$ is strictly plurisubharmonic on $W_j'$, and hence the function $F_{j,\varepsilon_j}$ is strictly plurisubharmonic on $W_j \coloneqq \phi_j^{-1}(W_j')$. Finally, for $\varepsilon_j \to 0$ the function $\tilde{f}_{j,\varepsilon_j}'|_{bG_j'}$ converges uniformly to $g_j'$, i.e., $F_{j,\varepsilon_j}|_{b\Omega}$ converges uniformly to $\theta_j$. 

For every $j \in \N$, let $\varepsilon_{j,0} \coloneqq \min\{\varepsilon_{j,1}, \varepsilon_{j,2}, \varepsilon_{j,3}\}$. Consider the sets $e$ and $d$ of sequences of nonnegative numbers defined by $e \coloneqq \big\{\varepsilon = \{\varepsilon_j\}_{j=1}^\infty : 0 < \varepsilon_j \le \varepsilon_{j,0}\big\}$ and $d \coloneqq \big\{\delta = \{\delta_j\}_{j=1}^\infty : 0 < \delta_j \le 1/2\big\}$. For every $(\varepsilon, \delta) \in e \times d$, define a function $F_{\varepsilon, \delta} \colon \overline{\Omega} \to [0,1]$ as $F_{\varepsilon, \delta} \coloneqq \sum_{j=1}^\infty F_{j,\varepsilon_j}(1 - \delta_j)$. Since $\supp F_{j,\varepsilon_j} \subset\subset U_j$ for every $j \in \N$, and since $\{U_j\}_{j=1}^\infty$ is locally finite, each of the functions $F_{\varepsilon,\delta}$ is a well defined smooth and plurisubharmonic function such that $\supp F_{\varepsilon, \delta} \subset U$. Moreover, $F_{\varepsilon,\delta}$ is strictly plurisubharmonic on $W \coloneqq \bigcup_{j=1}^\infty W_j$. By construction, each set $W_j$ is an open neighbourhood of $X_j$ in $\overline{\Omega}$, and since $\{X_j\}_{j=1}^\infty$ covers $b\Omega$, it follows that $W$ is an open neighbourhood of $b\Omega$ in $\overline{\Omega}$. Moreover, we claim that the following assertion holds true: for every continuous function $k \colon b\Omega \to (0,\infty)$ we can chose $(\varepsilon, \delta) \in e \times d$ such that
\begin{equation} \label{equ_bdestimate} 1-k < F_{\varepsilon, \delta} < 1 \text{ on } b\Omega. \end{equation}
Indeed, for every $\delta \in d$ define functions $\delta_{min}, \delta_{max} \colon b\Omega \to (0, 1/2]$ as $\delta_{min}(z) \coloneqq \min\{\delta_j : z \in U_j\}$ and $\delta_{max} \coloneqq \max\{\delta_j : z \in U_j\}$, respectively, and for every $\varepsilon \in e$ let $F_\varepsilon \coloneqq \sum_{j=1}^\infty F_{j,\varepsilon_j}$. Since $\{U_j\}_{j=1}^\infty$ is locally finite, and since $k$ is continuous, we can choose $\delta \in d$ so small that $1-k/2 < 1-\delta_{max}$. Then, since for $\varepsilon_j \to 0$ the function $F_{j,\varepsilon_j}|_{b\Omega \cap U_j}$ converges uniformly to $\theta_j|_{b\Omega \cap U_j}$ for every $j \in \N$, and since $F_{j,\varepsilon_j}|_{b\Omega \setminus U_j} = \theta_j|_{b\Omega \setminus U_j} \equiv 0$ for every $j \in \N$ and $\varepsilon_j > 0$, we can choose $\varepsilon \in e$ so small that $(1-k)/(1-k/2) < F_\varepsilon < 1/(1-\delta_{min})$ on $b\Omega$. Now observe that, by definition of $F_{\varepsilon, \delta}$, we have $(1 - \delta_{\max})F_\varepsilon \le F_{\varepsilon,\delta} \le (1-\delta_{min})F_\varepsilon$, hence it follows that $1-k < F_{\varepsilon, \delta} < 1$ on $b\Omega$ as claimed. Finally, note that the inequality $F_{\varepsilon, \delta} < 1$ on $b\Omega$ implies that $F_{\varepsilon, \delta} < 1$ on $\Omega$, since $\supp F_{\varepsilon, \delta} \subset U$, $U$ is filled by analytic discs attached to $b\Omega$, and $F_{\varepsilon, \delta}$ is smooth and plurisubharmonic on $\overline{\Omega}$.

Now define a continuous function $\nu \colon \overline{\Omega} \to (0,\infty)$ as $\nu \coloneqq \sup_{(\varepsilon,\delta) \in e \times d} \norm{dF_{\varepsilon, \delta}}_{h^\ast}$ and observe that indeed $\nu(z) < \infty$ for every $z \in \overline{\Omega}$. Let $\psi \colon V \to \R$ be a smooth strictly plurisubharmonic function on an open neighbourhood $V \subset \mathcal{M}$ of $b\Omega$ such that $\Omega \cap V = \{\psi < 0\}$ and $\norm{d\psi}_{h^\ast} > 1 + \nu$. The existence of such a function $\psi$ follows immediately from Lemma \ref{thm_deffctgradnorm}. Let $k \colon b\Omega \to (0,\infty)$ be a sufficiently small continuous function and let $(\varepsilon, \delta) \in e \times d$ be chosen in such a way that $(\ref{equ_bdestimate})$ holds true. Since $\norm{d\psi}_{h^\ast} > 1 + \norm{d F_{\varepsilon,\delta}}_{h^\ast}$ on $b\Omega$, it is easy to see that we have $\psi < F_{\varepsilon,\delta} - 1$ on $b(V \cap W) \cap \Omega$, provided that $k$ is chosen small enough. Hence the function $\tilde{\varphi} \colon \overline{\Omega} \cup V \to \R$ defined by
\[\tilde{\varphi}(z) \coloneqq \left\{ \begin{array}{c@{\,,\quad}l} \psi(z) & z \in V \setminus \Omega \\ \max(\psi(z), F_{\varepsilon,\delta}(z)-1) & z \in (V \cap W) \cap \Omega \\ F_{\varepsilon,\delta}(z)-1 & z \in \Omega \setminus (V \cap W) \end{array} \right.\]
is a continuous plurisubharmonic function such that $\tilde{\varphi} = \psi$ near $b\Omega$ and $\Omega = \{\tilde{\varphi}  < 0\}$. That is why $\tilde{\varphi}$ has all the properties that we seek, except, possibly, for smoothness in points of the set $A \coloneqq \{z \in V \cap W \cap \Omega : \psi(z) = F_{\varepsilon, \delta}(z) \}$. But both $\psi$ and $F_{\varepsilon, \delta}$ are strictly plurisubharmonic on $V \cap W \cap \Omega$. Hence, if $\omega$ is an arbitrary fixed neighbourhood of $A$ such that $\overline{\omega} \subset \Omega$, we can apply Richberg's smoothing method to obtain from $\tilde{\varphi}$ a smooth plurisubharmonic function $\varphi \colon V \cup \overline{\Omega}$ such that $\varphi = \tilde{\varphi}$ outside $\omega$ and such that still $\Omega = \{\varphi < 0\}$ (see, for example, Theorem I.5.21 in \cite{DemaillyXX} for a version of Richberg's smoothing procedure that is strong enough for our purpose). Then $\varphi$ is a function as desired.
\end{proof}

\subsection{Global defining functions in complex spaces} \label{sec_deffctspaces}
In this section we extend the above results to the setting of complex spaces. However, at least at the following two points our results are weaker when compared to the case of complex manifolds. First, we are not able to establish a general existence theorem for global defining functions of smoothly strictly $q$-pseudoconvex domains if $q > 0$. Instead, we have to restrict ourselves to the case of strictly hyper-$q$-pseudoconvex domains. Secondly, if $\Omega$ is a smoothly strictly pseudoconvex domain (i.e., $q=0$) in an arbitrary complex space, a subtle technical problem concerning the regularity of the desired function arises, when one tries to construct a smoothly global defining function that is smoothly strictly plurisubharmonic outside $\mathfrak{c}(\Omega)$. We start by gathering the necessary definitions and results.

Let $X = (X, \mathcal{O}_X)$ be a complex space (all complex spaces are assumed to be reduced and paracompact). A holomorphic chart for $X$ is a tupel $(U, \tau, A, G)$ where $U \subset X$ is open, $A$ is an analytic subset of a domain $G \subset \C^n$  and $\tau \colon U \xrightarrow{\sim} A$ is biholomorphic. For every $x \in X$, let $T_x(X)$ denote the Zariski tangent space of $X$ at $x$, i.e., $T_x(X) \coloneqq (\mathfrak{m}_x/\mathfrak{m}_x^2)^\ast$ where $\mathfrak{m}_x \subset \mathcal{O}_x$ is the maximal ideal of germs of holomorphic functions that vanish in $x$. If $f \colon X \to Y$ is a holomorphic map between complex spaces $X$ and $Y$, then we write $f_\ast = f_{\ast,x} \colon T_x(X) \to T_{f(x)}(Y)$ for the induced differential map. Let $\varphi \colon X \to \R$ be a smooth function, let $(U, \tau ,A, G)$ be a holomorphic chart for $X$ around $x$ and let $\hat{\varphi} \colon G \to \R$ be a smooth function such that $\varphi = \hat{\varphi} \circ \tau$ on $U$ (see below for the definition of smooth functions on complex spaces). Then we can define functionals $(d\varphi)_x \colon T_x(X) \to \R$ and $(\partial \varphi)_x, (\overline{\partial} \varphi)_x \colon T_x(X) \to \C$ by setting
\begin{gather*}
(\partial \varphi)_x(\xi) \coloneqq (\partial \hat{\varphi})_{\tau(x)}(\tau_\ast\xi), \quad (\overline{\partial} \varphi)_x(\xi) \coloneqq (\overline{\partial} \hat{\varphi})_{\tau(x)}(\tau_\ast\xi), \\ (d\varphi)_x(\xi) \coloneqq (\partial \varphi)_x(\xi) + (\overline{\partial} \varphi)_x(\xi)
\end{gather*}
for every $\xi \in T_x(X)$. Indeed, by part 1 of the Proposition in \cite{Vajaitu93}, this definition is independent of the smooth extension $\hat{\varphi}$, and by assertion (1) in Section 1 of \cite{Grauert62} it is also independent of the holomorphic chart $(U, \tau ,A, G)$. In particular, $H_x(\varphi) \coloneqq \{ \xi \in T_x(X) : (\partial \varphi)_x(\xi) = 0\}$ is a well defined subspace of $T_x(X)$. In the same way we want to define $\Lev(\varphi)(x,\,\cdot\,) \colon T_x(X) \to \R$ as
\begin{equation} \label{equ_defLeviform} \Lev(\varphi)(x,\xi) \coloneqq \Lev(\hat{\varphi})\big(\tau(x),\tau_\ast\xi\big). \end{equation}
However, as it is shown by Example 1 in \cite{Vajaitu93}, the number $\Lev(\varphi)(x,\xi)$ defined in this way will in general depend on the choice of the smooth extension $\hat{\varphi}$. In fact, in order for $(\ref{equ_defLeviform})$ to be well defined, we need to require that $X$ is locally irreducible at $x$, see part 2 of the Proposition in \cite{Vajaitu93}. (We do not know if the functionals $(\partial \varphi)_x, (\overline{\partial} \varphi)_x$ and $(d\varphi)_x$ can be well defined in general if $\varphi$ is only assumed to be $\mathcal{C}^1$-smooth. We also do not know whether on locally irreducible complex spaces the Levi form $\Lev(\varphi)(x,\,\cdot\,)$ can be well defined for arbitrary $\mathcal{C}^2$-smooth functions $\varphi$.)

A function $\varphi \colon X \to \R$ is called smooth or (strictly) plurisubharmonic, if for every $x \in X$ there exist a holomorphic chart $(U, \tau, A, G)$ around $x$ and a smooth or (strictly) plurisubharmonic function $\hat{\varphi} \colon G \to \R$ such that $\varphi|_U = \hat{\varphi} \circ \tau$, respectively. Observe that it is not clear from the definition whether a smooth and (strictly) plurisubharmonic function $\varphi \colon X \to \R$ does admit local extensions $\hat{\varphi}$ as above that are both smooth and (strictly) plurisubharmonic at the same time. In fact, this is not true in general, see, for example, Warning 1.5 in \cite{Smith86} and Example 2 in \cite{Vajaitu93}. If around each point $x \in X$ the function $\varphi$ admits holomorphic charts and local extensions $\hat{\varphi}$ that are smooth and (strictly) plurisubharmonic, then $\varphi$ will be called smoothly (strictly) plurisubharmonic. A domain $\Omega \subset X$ is called strictly pseudoconvex if for every $x \in b\Omega$ there exist an open neighbourhood $U_x \subset X$ of $x$ and a continuous strictly plurisubharmonic function $\varphi_x \colon U_x \to \R$ such that $\Omega \cap U_x = \{\varphi_x < 0\}$. The domain $\Omega$ will be called smoothly strictly pseudoconvex if for every $x \in b\Omega$ the function $\varphi_x \colon U_x \to \R$ can be chosen to be smoothly strictly plurisubharmonic. (In the same way we can define the notions of $\mathcal{C}^s$-smoothly (strictly) plurisubharmonic functions and $\mathcal{C}^s$-smoothly strictly pseudoconvex domains for every $s \in \N_0^\infty$. Note that a function $\varphi \colon X \to \R$ is $\mathcal{C}^0$-smooth and (strictly) plurisubharmonic if and only if it is $\mathcal{C}^0$-smoothly (strictly) plurisubharmonic, see Theorem 2.4 in \cite{Richberg68}.)

Let $q \in \N_0$. A function $\varphi \colon X \to \R$ is called (strictly) $q$-plurisubharmonic, if for every $x \in X$ there exist a holomorphic chart $(U, \tau, A, G)$ around $x$ and a $\mathcal{C}^2$-smooth (strictly) $q$-plurisubharmonic function $\hat{\varphi} \colon G \to \R$ such that $\varphi|_U = \hat{\varphi} \circ \tau$. It is called smoothly (strictly) plurisubharmonic if around each point $x \in X$ the function $\varphi$ admits holomorphic charts and local extensions $\hat{\varphi}$ that are smooth and (strictly) $q$-plurisubharmonic. A domain $\Omega \subset X$ is called (smoothly) strictly $q$-pseudoconvex if for every $x \in b\Omega$ there exist an open neighbourhood $U_x \subset X$ of $x$ and a (smoothly) strictly $q$-plurisubharmonic function $\varphi_x \colon U_x \to \R$ such that $\Omega \cap U_x = \{\varphi_x < 0\}$. (Analoguous definitions are possible for $\mathcal{C}^s$-smoothly (strictly) $q$-plurisubharmonic functions and $\mathcal{C}^s$-smoothly strictly $q$-pseudoconvex domains, whenever $s > 2$.)

Finally, we will say that the boundary $b\Omega$ is smooth in $x \in b\Omega$, if there exists a smooth function $\varphi \colon U \to \R$ defined on an open neighbourhood $U \subset X$ of $x$ such that $\Omega \cap U = \{\varphi < 0\}$ and $(d\varphi)_x \neq 0$. Observe that $b\Omega$ is smooth in $x \in b\Omega$ if and only if in every small enough minimal holomorphic chart around $x$ (i.e., every small enough chart $(U, \tau, A, G)$ around $x$ such that $G \subset \C^{\ebdim_ x X}$, where $\ebdim_ xX = \dim_\C T_x(X)$ denotes the embedding dimension of $X$ at $x$) $\Omega$ is the intersection of $X$  with a smoothly bounded subdomain of the ambient $\C^n$. A function $f \colon b\Omega \to \R$ will be called smooth if $f$ is the restriction of a smooth function defined on an open neighbourhood $U \subset X$ of $b\Omega$. In the case when $X$ is a manifold and $b\Omega$ is smooth this definition coincides with the usual one. (Again, analoguous definitions of $\mathcal{C}^s$-smooth boundaries and $\mathcal{C}^s$-smooth functions can be given for every $s \in \N_0^\infty$.)

Now we can formulate the main results of this section which generalize Theorem \ref{thm_qbdvalue_manifolds} and Theorem \ref{thm_qdeffct_manifolds} to the case of smoothly strictly pseudoconvex domains in complex spaces.

\begin{theorem} \label{thm_bdvalue_spaces}
Let $X$ be a complex space, let $\Omega \subset X$ be a smoothly strictly pseudoconvex domain and let $f \colon b\Omega \to \R$ be a smooth function that is bounded from below. Then there exists a smoothly plurisubharmonic function $F$ defined on an open neighbourhood of $\overline{\Omega}$ such that $F|_{b\Omega} = f$ and $F$ is smoothly strictly plurisubharmonic near $b\Omega$.
\end{theorem}

\begin{theorem} \label{thm_deffct_spaces}
Let $X$ be a complex space and let $\Omega \subset X$ be a smoothly strictly pseudoconvex domain. Then there exists a smoothly plurisubharmonic function $\varphi$ defined on an open neighbourhood of $\overline{\Omega}$ such that $\Omega = \{\varphi < 0\}$ and $\varphi$ is smoothly strictly plurisubharmonic near $b\Omega$.
\end{theorem}

We would also like to prove results analoguous to Theorem \ref{thm_qbdvalue_manifolds} and Theorem \ref{thm_qdeffct_manifolds} for the case of smoothly strictly $q$-pseudoconvex domains in complex spaces. However, we do not know whether this is possible in general if $q > 0$. The problem is essentially the following: If $\Omega$ is a domain in a complex manifold $\mathcal{M}$, $z \in b\Omega$, $U \subset \mathcal{M}$ is an open neighbourhood of $z$ and $\varphi_1, \varphi_2 \colon U \to \R$ are smooth functions such that $\Omega \cap U = \{\varphi_1 < 0\} = \{\varphi_2 < 0\}$, then for every $\xi \in H_z(\varphi_1)$ we have $\Lev(\varphi_1)(z,\xi) \ge 0$ if and only if $\Lev(\varphi_2)(z,\xi) \ge 0$ (see Remark 5 after Theorem \ref{thm_qdeffct_manifolds}). Thus when adding $\varphi_2$ to $\varphi_1$ we do not loose positivity of the Levi form on $H_z(\varphi_1)$, and if $H_z(\varphi_1) \neq T_z(\mathcal{M})$, then a possible loss of positivity in the direction normal to $H_z(\varphi_1)$ (with respect to some hermitian metric $h$ on $\mathcal{M}$) can be reacquired by composing the sum $\varphi_1 + \varphi_2$ with a smooth strictly increasing and strictly convex function $\chiup \colon \R \to \R$. However, this is not longer true in general in the setting of complex spaces as it is shown by the following example.

\begin{example}
Let $X \coloneqq \{(z,w) \in \C^2 : z^3 = w^2\}$ and let $\Omega \coloneqq X \setminus \{0\}$. Consider the smoothly $1$-plurisubharmonic functions $\varphi_1, \varphi_2 \colon X \to \R$ which are defined as the restrictions to $X$ of the functions $\hat{\varphi}_1(z,w) \coloneqq \abs{z+w}^2 - 2\abs{z-w}^2$ and $\hat{\varphi}_2(z,w) \coloneqq \abs{z-w}^2 - 2\abs{z+w}^2$ on $\C^2$, respectively. One easily verifies that in a small open neighbourhood $U \subset X$ of $0 \in X$ it holds true that $\Omega \cap U = \{x \in U : \varphi_1(x) < 0\} = \{x \in U : \varphi_2(x) < 0\}$ and hence $\Omega$ is a smoothly strictly $1$-pseudoconvex domain. Since $T_0(X) \simeq \C^2$, and since $X$ is locally irreducible, the Levi form at the origin of every smooth extension of $\varphi_1 + \varphi_2$ to an open neighbourhood of $0 \in \C^2$ coincides with the Levi form of $\hat{\varphi}_1 + \hat{\varphi}_2$ in $0$. However, $\Lev(\hat{\varphi}_1 + \hat{\varphi}_2)(0, \,\cdot\,)$ is negative definit on $\C^2 \simeq H_0(\varphi_1)$.
\end{example}

One might argue that the above example is of a rather pathological nature. On the other hand, observe that typically no assumptions about the smoothness of $b\Omega$ are made in the definition of smoothly strictly $q$-pseudoconvex domains in complex spaces (see, for example, \cite{AndreottiGrauert62}). Anyway, even if $b\Omega$ is assumed to be smooth at $x \in b\Omega$ we do not know whether the Levi forms at $x$ of two local defining functions for $\Omega$ around $x$ can be compared as it is done in the manifold case.

As a consequence of the described above problem, in the case $q > 0$ we prove generalizations of Theorem \ref{thm_qbdvalue_manifolds} and Theorem \ref{thm_qdeffct_manifolds} only for hyper-$q$-pseudoconvex domains instead of smoothly strictly $q$-pseudoconvex domains. Before stating the precise results we collect all necessary definitions.

First we remind the definition of hermitian metrics on complex spaces. Let $\pi \colon T(X) \to X$ be the Zariski tangent linear space, i.e., the underlying set of $T(X)$ is simply the disjoint union $\bigcup_{x \in X} T_x(X)$ (see, for example, \cite{Fischer76} for more details). A hermitian metric $h$ on $X$ is a smooth mapping $h \colon T(X) \times_\pi T(X) \to \C$ such that $h|_{T_x(X) \times T_x(X)}$ is a hermitian metric on $T_x(X)$ for every $x \in X$. If $h$ is a hermitian metric on $X$, then for every $x \in X$ we denote by $\norm{\,\cdot\,}_{h_x}$ and $\norm{\,\cdot\,}_{h_x^\ast}$ the induced norms on $T_x(X)$ and $T_x^\ast(X)$, respectively. If the context is clear, then we sometimes omit the index $x$ and simply write $\norm{\,\cdot\,}_h$ or $\norm{\,\cdot\,}_{h^\ast}$.

Let $X$ be a complex space endowed with a hermitian metric $h$. A smooth function $\varphi \colon X \to \R$ is called hyper-$q$-plurisubharmonic (respectively strictly hyper-$q$-plurisubharmonic) if for every complex subspace $Y \subset X$ and every $y \in Y$, every holomorphic chart $(U, \tau, A, G)$ for $Y$ around $y$ and every hermitian metric $\hat{h}$ on $G \subset \C^n$ which satisfies $h|_U = \tau^\ast\hat{h}$ there exist an open neighbourhood $G' \subset G$ of $\tau(y)$ and a smooth function $\hat{\varphi} \colon G' \to \R$ such that $\varphi = \hat{\varphi} \circ \tau$ on $U' \coloneqq \tau^{-1}(G') \subset U$ with the following property: for every $z \in G'$ the trace with respect to $\hat{h}$ of the restriction of the Levi form $\Lev(\hat{\varphi})(z,\,\cdot\,)$ to any $(q+1)$-dimensional subspace of $\C^n$ is nonnegative (respectively positive), i.e., for every $\hat{h}$-orthonormal collection of vectors $e_1, e_2, \ldots, e_{q+1} \subset \C^n$ we have that $\sum_{j=1}^{q+1} \Lev(\hat{\varphi})(z, e_j) \ge 0$ (respectively $> 0$). Observe that these definitions depend on the given hermitian metric $h$ and that in general the Levi form of $\hat{\varphi}$ is not uniquely determined by $\varphi$ (it is if $X$ is locally irreducible). Moreover, it is clear from the definition that every (strictly) hyper-$q$-plurisubharmonic function is (strictly) $q$-plurisubharmonic. The main advantage of the set of (strictly) hyper-$q$-plurisubharmonic functions over the set of all (strictly) $q$-plurisubharmonic functions is that the former set is closed under addition. (In the case of not necessarily strictly hyper-$q$-plurisubharmonic functions we need here that $\varphi$ has an extension $\hat{\varphi}$ as described above with respect to every extension $\hat{h}$ of $h$. Then the additional assumption on the complex subspaces $Y \subset X$ is imposed in order to guarantee that restrictions of hyper-$q$-plurisubharmonic functions to complex subspaces are again hyper-$q$-plurisubharmonic; we will not need this property in our constructions. Moreover, for not necessarily strictly hyper-$q$-plurisubharmonic functions it is also not clear if requiring the existence of the extension $\hat{\varphi}$ only with respect to one fixed chart of $Y$, instead of requiring it with respect to every chart of $Y$, yields an equivalent definition. All these complications do not arise in the case of strictly hyper-$q$-plurisubharmonic functions, and thus there is a less technical but equivalent definition in this situation, see, for example, Proposition 2.2 in \cite{FraboniNapier10}. In particular, it follows easily from the above remarks that for every strictly hyper-$q$-plurisubharmonic function $\varphi \colon X \to \R$ and every compactly supported smooth function $\theta  \colon X \to \R$ there exists $\varepsilon_0 > 0$ such that $\varphi + \varepsilon\theta$ is strictly hyper-$q$-plurisubharmonic for every $\abs{\varepsilon} \le \varepsilon_0$.) Finally, note that (strictly) hyper-$0$-plurisubharmonic just means smoothly (strictly) plurisubharmonic. The notion of hyper-q-plurisubharmonicity was first introduced in the context of complex manifolds by Grauert and Riemenschneider in \cite{GrauertRiemenschneider70}, the above definition for complex spaces is taken from \cite{FraboniNapier10} (actually Grauert and Riemenschneider use the term hyper-$(q+1)$-convex functions instead of strictly hyper-$q$-plurisubharmonic functions, but since we prefered the term of strict $q$-plurisubharmonicity over $(q+1)$-convexity before, we stick to this convention).

A domain $\Omega \subset X = (X,h)$ will be called strictly hyper-$q$-pseudoconvex if for every $x \in b\Omega$ there exist an open neighbourhood $U_x \subset X$ of $x$ and a strictly hyper-$q$-plurisubharmonic function $\varphi_x \colon U_x \to \R$ such that $\Omega \cap U_x = \{\varphi_x < 0\}$. Observe that $\Omega$ is strictly hyper-$0$-pseudoconvex if and only if it is smoothly strictly pseudoconvex. (Again analoguous definitions would be possible in the $\mathcal{C}^s$-smooth categories for every $s \ge 2$.)

Now we turn to the generalizations of Theorem \ref{thm_qbdvalue_manifolds} and Theorem \ref{thm_qdeffct_manifolds} to hyper-$q$-pseudoconvex domains in complex spaces. Note that for $q=0$ this will include the case of smoothly strictly pseudoconvex domains. Hence Theorems \ref{thm_bdvalue_spaces} and \ref{thm_deffct_spaces} are special cases of the next two theorems, and thus it suffices to prove only these more general results. \par

\setcounter{theorem_prime}{3} 
\begin{theorem_prime} \label{thm_bdvalue_spaces_prime}
Let $X$ be a complex space, let $\Omega \subset X$ be a strictly hyper-$q$-pseudo\-con\-vex domain and let $f \colon b\Omega \to \R$ be a smooth function that is bounded from below. Then there exists a hyper-$q$-plurisubharmonic function $F$ defined on an open neighbourhood of $\overline{\Omega}$ such that $F|_{b\Omega} = f$ and $F$ is strictly hyper-$q$-plurisubharmonic near $b\Omega$.
\end{theorem_prime}

%\noindent \textbf{Theorem \ref{thm_bdvalue_spaces}$\,^\prime$.}
%{\it Let $X$ be a complex space, let $\Omega \subset X$ be a strictly hyper-$q$-pseudo\-con\-vex domain and let $f \colon b\Omega \to \R$ be a smooth function that is bounded from below. Then there exists a hyper-$q$-plurisubharmonic function $F$ defined on an open neighbourhood of $\overline{\Omega}$ such that $F|_{b\Omega} = f$ and $F$ is strictly hyper-$q$-plurisubharmonic near $b\Omega$.} \par

\noindent \textbf{Proof.}
Since $f$ is bounded from below, we can assume without loss of generality that $f>0$. Let $\widetilde{F} \colon X \to (0,\infty)$ be a smooth extension of $f$. Let $\{U_j''\}_{j=1}^\infty$ be a locally finite covering of $b\Omega$ by open sets $U_j'' \subset\subset X$ such that for every $j \in \N$ there exists a strictly hyper-$q$-plurisubharmonic function $\varphi_j \colon U_j'' \to \R$ such that $\Omega \cap U_j'' = \{\varphi_j < 0\}$. Moreover, let $U_j' \subset\subset U_j \subset\subset U_j''$ be open sets such that $\{U_j'\}_{j=1}^\infty$ still covers $b\Omega$.

Let $\{\chiup_j\}_{j=1}^\infty$ be a family of smooth functions $\chiup_j \colon X \to [0,\infty)$ such that $\{\chiup_j > 0\} = U_j'$ for every $j \in \N$, $\sum_{j=1}^\infty \chiup_j \le 1$ on $X$ and $\sum_{j=1}^\infty \chiup_j \equiv 1$ near $b\Omega$. Let $\beta \colon (0,\infty) \to (0,\infty)$ be a strictly increasing and strictly convex smooth function such that $\beta(t) = e^{-1/t}$ for small values of $t$, and let $\tilde{\beta} \colon \R \to [0,\infty)$ be the smooth extension of $\beta$ such that $\tilde{\beta}|_{(-\infty,0]} \equiv 0$. As in the proof of Theorem \ref{thm_qbdvalue_manifolds}, by a proper choice of $\{\chiup_j\}_{j=1}^\infty$, we can guarantee that for every $j \in \N$ the trivial extension $g_j \colon X \to [0, \infty)$ of the function $\beta^{-1} \circ(\widetilde{F}\chiup_j) \colon U_j' \to (0,\infty)$ by $0$ is smooth on $X$. 

Let $\lambda_j \colon X \to (-\infty, 0]$ be smooth such that $\overline{U_j'} = \{\lambda_j = 0\}$. Then choose $\varepsilon_j > 0$ so small and $C_j > 0$ so large that $g_j + C_j(\varphi_j + \varepsilon_j\lambda_j)$ is still strictly hyper-$q$-plurisubharmonic on $U_j$. Observe that, by construction, $g_j + C_j(\varphi_j + \varepsilon_j\lambda_j) < 0$ on $bU_j \cap \overline{\Omega}$, hence the function $\tilde{\beta} \circ (g_j + C_j(\varphi_j + \varepsilon_j\lambda_j))|_{U_j}$ vanishes near this set and thus its trivial extension by $0$ to the open neighbourhood $\mathcal{U}_j \coloneqq X \setminus \{x \in bU_j : (g_j + C_j(\varphi_j + \varepsilon_j\lambda_j))(x) \ge 0\}$ of $\overline{\Omega}$ defines a hyper-$q$-plurisubharmonic function $F_j \colon \mathcal{U}_j \to [0,\infty)$ such that $F_j|_{b\Omega} = f\chiup_j$ and $F_j \equiv 0$ outside $U_j$. Moreover, $W_j \coloneqq \{F_j > 0\} \subset U_j$ is an open neighbourhood of $b\Omega \cap U_j'$ such that $F_j$ is strictly hyper-$q$-plurisubharmonic on $W_j$. Hence $F \coloneqq \sum_{j=1}^\infty F_j$ is hyper-$q$-plurisubharmonic on the open neighbourhood $\mathcal{U} \coloneqq \bigcap_{j=1}^\infty \mathcal{U}_j \subset X$ of $\,\overline{\Omega}$ such that $F|_{b\Omega} = f$ and $F$ is strictly hyper-$q$-plurisubharmonic on $W \coloneqq \bigcup_{j=1}^\infty W_j \supset b\Omega$. \hfill $\Box$ \par

\begin{theorem_prime} \label{thm_deffct_spaces_prime}
Let $X$ be a complex space and let $\Omega \subset X$ be a strictly hyper-$q$-pseudo\-con\-vex domain. Then there exists a hyper-$q$-plu\-ri\-sub\-har\-monic function $\varphi$ defined on an open neighbourhood of $\overline{\Omega}$ such that $\Omega = \{\varphi < 0\}$ and $\varphi$ is strictly hyper-$q$-plurisubharmonic near $b\Omega$.
\end{theorem_prime}
%\noindent \textbf{Theorem \ref{thm_deffct_spaces}$\,^\prime$.}
%{\it Let $X$ be a complex space and let $\Omega \subset X$ be a strictly hyper-$q$-pseudo\-con\-vex domain. Then there exists a hyper-$q$-plu\-ri\-sub\-har\-monic function $\varphi$ defined on an open neighbourhood of $\overline{\Omega}$ such that $\Omega = \{\varphi < 0\}$ and $\varphi$ is strictly hyper-$q$-plurisubharmonic near $b\Omega$.} \par

\noindent \textbf{Proof.}
Let $\varphi \coloneqq F - 1$ where $F$ is the function from Theorem $\ref{thm_bdvalue_spaces_prime}$ $\,^\prime$ corresponding to the boundary values $f \equiv 1$. Then $\varphi$ is a hyper-$q$-plurisubharmonic function on an open neighbourhood of $\overline{\Omega}$ that vanishes identically on $b\Omega$ and that is strictly hyper-$q$-plurisubharmonic near $b\Omega$. As before, in the construction of $F$ we can choose $\widetilde{F}$ such that $\Omega = \{\widetilde{F} < 1\}$, and this choice implies that $\Omega = \{F < 1\}$, i.e., $\Omega = \{\varphi < 0\}$. \hfill $\Box$ \par

\noindent \textbf{Remarks.} 1) In the last two theorems strict hyper-$q$-convexity is understood with respect to an arbitrary but fixed hermitian metric on $X$.

\noindent 2) Similar to what we had above, the function $F$ from Theorem \ref{thm_bdvalue_spaces_prime}$\,^\prime$ is strictly hyper-$q$-plurisubharmonic on the open neighbourhood $W$ of $b\Omega$ and it is constant on $\Omega \setminus W$. One sees immediately from our construction that for every open set $\omega \subset \Omega$ such that $\overline{\omega} \subset \Omega$ we can choose $F$ in such a way that it will be constant on $\omega$.

\noindent 3) The statements of Theorems \ref{thm_bdvalue_spaces_prime}$\,^\prime$ and \ref{thm_deffct_spaces_prime}$\,^\prime$ remain true if $\mathcal{C}^\infty$-smoothness is replaced by $\mathcal{C}^s$-smoothness for any $s \ge 2$. Also for $\mathcal{C}^s$-smoothly strictly pseudoconvex domains with $s \in \{0, 1\}$ the proofs still work, but in Theorem \ref{thm_bdvalue_spaces_prime}$\,^\prime$ we have to assume that the function $f \colon b\Omega \to \R$ is at least $\mathcal{C}^2$-smooth. In particular, every strictly pseudoconvex domain in a complex space admits a continuous global defining function.

It is also possible to generalize Theorem \ref{thm_nbhbasis_manifolds} to the case of complex spaces. The precise statement is contained in the following theorem.

\begin{theorem}
Let $X$ be a complex space and let $\Omega \subset X$ be a strictly pseudoconvex domain. Let $U \subset X$ be an arbitrary open neighbourhood of $b \Omega$. Then the following assertions hold true:
\begin{enumerate}
  \item[(1)] There exists a smoothly strictly pseudoconvex domain $\Omega' \subset X$ such that $\Omega \setminus U \subset \Omega'$ and $\overline{\Omega'} \subset \Omega$.
  \item[(2)] There exists a smoothly strictly pseudoconvex domain $\Omega'' \subset X$ such that $\overline{\Omega} \subset \Omega''$ and $\overline{\Omega''} \subset \Omega \cup U$.
\end{enumerate}
In particular, $\overline{\Omega}$ admits a neighbourhood basis consisting of smoothly strictly pseudoconvex domains.
\end{theorem}
\begin{proof}
This follows easily by adjusting the proof of Theorem \ref{thm_nbhbasis_manifolds} to the situation of complex spaces, compare for example with the proof of Theorem \ref{thm_bdvalue_spaces_prime}$\,^\prime$ above.
\end{proof}

\noindent \textbf{Remark.} A similar result holds true for strictly hyper-$q$-pseudoconvex domains $\Omega \subset X$. However, in the case $q > 0$ the regularity of $b\Omega'$ and $b\Omega''$ will in general be only as good as the regularity of $b\Omega$, since we do not know whether it is always possible to make a hyper-$q$-plurisubharmonic smoothing of a hyper-$q$-plurisubharmonic function.

As in the case of manifolds, we can now introduce the notion of the core of a smoothly strictly pseudoconvex domain in an arbitrary complex space.

\begin{definition} \label{def_corespaces}
Let $X$ be a complex space and let $\Omega \subset X$ be a domain. Then the set
\begin{equation*} \begin{split}
\mathfrak{c}(\Omega) \coloneqq \big\{ z \in & \;\Omega : \text{every smoothly plurisubharmonic function on $\Omega$ that is bounded} \\ &\text{from above fails to be smoothly strictly plurisubharmonic in } z \big\}
\end{split} \end{equation*}
will be called the \emph{core} of $\Omega$.
\end{definition}

Since the construction of smooth maximum still works in arbitrary complex spaces, one easily sees that a statement analoguous to Lemma \ref{thm_coreequality} can be proved for smoothly strictly pseudoconvex domains in complex spaces. However, a subtle technical problem concerning the regularity of the global defining function occurs if one tries to extend the Main Theorem to the setting of complex spaces. In fact, in the proof of the Main Theorem we construct the smooth plurisubharmonic function $\varphi_1$ as a limit of a series of smooth plurisubharmonic functions. Correspondingly, in the case of complex spaces the function $\varphi_1$ would be a limit of a series of smoothly plurisubharmonic functions. Observe though that when taking limits in the class of smoothly plurisubharmonic functions defined on a complex space, it is not clear in general in which cases the limit function will again be smoothly plurisubharmonic. Indeed, if we try to repeat the proof of the Main Theorem for complex spaces, then one can easily choose the sequence $\{\varepsilon_j\}_{j=1}^\infty$ in such a way that the function $\varphi_1$ is smooth and plurisubharmonic (here we need the equivalence of weakly plurisubharmonic and plurisubharmonic functions on complex spaces, see Theorem 5.3.1 in \cite{FornaessNarasimhan80}), but it is not clear whether $\varphi_1$ will also be smoothly plurisubharmonic. The problem is that given a point $x \in X$ and a sequence $\{\Psi_j\}$ of smoothly plurisubharmonic functions on $X$, then after a local embedding of $X$ into some $\C^n$ each $\Psi_j$ extends as a smooth plurisubharmonic function onto some neighbourhood $\hat{U}_j \subset \C^n$ of $x$, but it is not clear whether one can guarantee that also $\bigcap_{j=1}^\infty \hat{U}_j$ will contain some neighbourhood of $x$. For the function $\varphi_1$, we only know how to avoid this problem away from $\mathfrak{c}(\Omega)$, namely, we can at least show that $\varphi_1$ is smoothly strictly plurisubharmonic outside $\mathfrak{c}(\Omega)$. We sketch briefly the corresponding argument: By construction, $\varphi_1 = \sum_{j=1}^\infty \varepsilon_j\Psi_j$ for some smoothly plurisubharmonic functions $\Psi_j$ on $\Omega$. Fix arbitrary $x \in \Omega \setminus \mathfrak{c}(\Omega)$. After a local embedding of the complex space X, we can find smooth extensions $\hat{\Psi}_j$ of the functions $\Psi_j$ to a uniformly large neighbourhood of $x$ in the ambient $\C^n$ such that each function $\hat{\Psi}_j$ has nonnegative Levi form in $x$. Moreover, since $x \notin \mathfrak{c}(\Omega)$, and by the choice of the functions $\Psi_j$, at least one of the functions $\hat{\Psi}_j$ has positive Levi form in $x$. Thus if $\{\varepsilon_j\}$ is chosen suitably, then the function $\hat{\varphi}_1 \coloneqq \sum_{j=1}^\infty \varepsilon_j\hat{\Psi}_j$ is a smooth extension of $\varphi_1$ which has a positive Levi form in $x$, i.e., $\varphi_1$ is smoothly strictly plurisubharmonic in $x$. (Analoguous results hold in the $\mathcal{C}^s$-smooth categories for every $s \ge 2$. Moreover, in view of Theorem 2.4 from \cite{Richberg68}, a full analogue of the Main Theorem holds true for strictly pseudoconvex domains in complex spaces in the $\mathcal{C}^0$-smooth category.)

Finally, an analogue of Proposition \ref{thm_emptycore} holds true for arbitrary Stein spaces (for the existence of a smoothly strictly plurisubharmonic function on a Stein space see, for example, the Lemma in Section 3 of \cite{Narasimhan61}).

\section{The core of a domain} \label{sec_core}
In this section we investigate properties of the core $\mathfrak{c}(\Omega)$ of a domain $\Omega$ in a complex manifold $\mathcal{M}$. Clearly, $\mathfrak{c}(\Omega)$ is always a relatively closed subset of $\Omega$. If $\Omega$ is strictly pseudoconvex with smooth boundary, then Theorem \ref{thm_manifolds_deffct_prime}$\,^\prime$ implies that $\mathfrak{c}(\Omega)$ is also closed in $\mathcal{M}$. Moreover, $\mathfrak{c}(\Omega) = \varnothing$ if $\mathcal{M}$ is Stein and $\Omega$ is relatively compact in $\mathcal{M}$. As remarked above, every domain $\Omega \subset \mathcal{M}$ admits a smooth and bounded plurisubharmonic function $\varphi \colon \Omega \to \R$ that is strictly plurisubharmonic precisely in $\Omega \setminus \mathfrak{c}(\Omega)$ (see the remarks following the definition of minimal functions on page \pageref{def_minimal}).

Before we start our investigations of properties of the core, we want to make the following observation: If $\Omega \subset \mathcal{M}$ is a domain, and $\omega \subset \Omega$ is a subdomain such that $\mathfrak{c}(\Omega) \subset \omega$, then clearly $\mathfrak{c}(\omega) \subset \mathfrak{c}(\Omega)$. We do not know, however, whether the reverse inclusion holds also true here, i.e., we do not know if in general $\mathfrak{c}(\omega) = \mathfrak{c}(\Omega)$. In fact, in all examples of domains $\Omega \subset \mathcal{M}$ with nonempty core that we will construct here (and also in the Examples \ref{ex_compactanalytic} and \ref{ex_complexline} that have already been given in the Introduction), the above equality does indeed hold true for every subdomain $\omega \subset \Omega$ that contains $\mathfrak{c}(\Omega)$. Thus, loosely speaking, in all examples that we are able to construct, the presence of the core $\mathfrak{c}(\Omega) \subset \Omega$ is only related to intrinsic properties of $\mathfrak{c}(\Omega)$, but not to properties of $\Omega$.

These observations will lead us in Section \ref{sec_liouville} to define the notion of sets of {\it core type}. For the moment, we just want to point out, that in view of a possible dependence of $\mathfrak{c}(\Omega)$ on $\Omega$, it is desirable to construct not only examples of arbitrary domains with nonempty core, but also of domains with additional properties, like, for example, pseudoconvexity. This concern is further illustrated by the following theorem.

\begin{theorem}
Let $n \ge 2$. The following assertions hold true for domains $\Omega \subset \C^n$ with coordinates $(z_1, z_2, \ldots, z_n)$, $z_j = x_j+iy_j$:
\begin{enumerate}
  \item[(1)] There exists a domain $\Omega \subset \C^n$ such that $\mathfrak{c}(\Omega) = E \times \C^{n-1}$, where $E \subset \C$ is the set $E = [0,1] \times \R_{y_1}$.
  \item[(2)] Let $\Omega \subset \C^n$ be a pseudoconvex domain such that $\mathfrak{c}(\Omega) = E \times \C^k$ for some $k \in \N_{n-1}$ and some set $E \subset \C^{n-k}$. Then either $E$ is complete pluripolar or $E$ is open. In the later case $\Omega = E \times \C^k$.
  \item[(3)] Let $k \in \N_{n-1}$ be arbitrary but fixed. Then there exists a strictly pseudoconvex domain $\Omega \subset \C^n$ such that $\mathfrak{c}(\Omega) = E \times \C^k$ for a set $E \subset \C^{n-k}$ if and only if $E$ is closed and complete pluripolar.
\end{enumerate}
\end{theorem}

\noindent \textbf{Remarks.} 1) Statement (1) shows that the core of an arbitrary domain $\Omega \subset \C^n$ may divide $\Omega$ into several connected components, and, moreover, $\mathfrak{c}(\Omega)$ may have nonemtpy interior. However, it is not clear to us at the moment, whether a pseudoconvex domain $\Omega \subset \C^n$ can have nonempty and disconnected complement $\Omega \setminus \mathfrak{c}(\Omega)$, or if the core of a strictly pseudoconvex domain $\Omega \subset \C^n$ can have nonempty interior.

\noindent 2) Statement (3) shows that if the core of a strictly pseudoconvex domain is assumed to have a product structure as described above, then it is always complete pluripolar. In view of this result, and also of the examples that will be given below, one can raise the following question: Is it always true that the core of a strictly pseudoconvex domain is complete pluripolar? At the moment we are not able to give an answer to this question.

\noindent 3) Some part of the arguments which we will use to prove the statements (2) and (3) of the theorem are similar to the ones which appear in the proof of Theorem 1.11 in \cite{Aupetit83}. In fact, to some extent, the corresponding statements are already implicitly contained in the above mentioned result of Aupetit.

\begin{proof}
(i) We start with proving statement (1) of the theorem. For every $j \in \N$, let $\psi_j \colon B^n(0,j) \to \R$ be the smooth and strictly plurisubharmonic function defined by
\[ \psi_j(z_1, \ldots, z_n) \coloneqq x_1 - \frac{1}{2^{j-2}} + \frac{1}{j^22^{j-1}}\big(y_1^2 + \abs{z_2}^2 + \cdots + \abs{z_n}^2\big). \]
Choose a smooth function $\chiup_j \colon \R \to [0,\infty)$ such that $\chiup_j \equiv 0$ on $(-\infty, -1/2^j]$ and such that $\chiup_j$ is strictly increasing and strictly convex on $(-1/2^j, \infty)$. Set $\widetilde{\varphi}_j \coloneqq \chiup_j \circ \psi_j$. Then $\widetilde{\varphi}_j$ is a smooth plurisubharmonic function on $B^n(0,j)$ such that $\widetilde{\varphi}_j \equiv 0$ on $\{\psi_j \le -1/2^j\} \supset B^n(0,j) \cap \{x_1 \le 1/2^j\}$ and such that $\widetilde{\varphi}_j$ is strictly plurisubharmonic and positive on $\{\psi_j > -1/2^j\} \supset B^n(0,j) \cap \{x_1 > 3/2^j\}$. Thus
\[ \varphi_j(z) \coloneqq \left\{\begin{array}{c@{\,,\quad}l} \widetilde{\varphi}_j(z) & z \in B^n(0,j) \cap \{x_1 \ge 1/2^j\} \\ 0 & z \in \{x_1 < 1/2^j\} \end{array} \right. \]
is a smooth plurisubharmonic function on $W_j \coloneqq B^n(0,j) \cup \{x_1 < 1/2^j\}$ such that $\varphi_j$ is strictly plurisubharmonic and positive on $B^n(0,j) \cap \{x_1 > 3/2^j\}$. Observe that $W \coloneqq \bigcap_{j=1}^\infty W_j$ is a connected open neighbourhood of $\{x_1 \le 0\}$. Then one easily sees that for a sequence $\{\varepsilon_j\}_{j=1}^\infty$ of positive numbers that converges to zero fast enough, the function $\varphi \coloneqq \sum_{j=1}^\infty \varepsilon_j\varphi_j$ is smooth and plurisubharmonic on $W$ such that $\varphi \equiv 0$ on $\{x_1 \le 0\}$ and such that $\varphi$ is strictly plurisubharmonic and positive on $W \cap \{x_1 > 0\}$.

Now define a domain $\Omega \subset \C^n$ as
\[ \Omega \coloneqq [W + (1,0, \ldots,0)] \cap [-W] = \big\{z \in \C^n : (z_1-1,z_2, \ldots, z_n) \in W \text{ and } -z \in W \big\}. \]
Then $E \times \C^{n-1} \subset \Omega$, where $E \coloneqq [0,1] \times \R_{y_1} \subset \C$. By the Liouville theorem, every plurisubharmonic function $u$ on $\Omega$ that is bounded from above has to be constant on $\{z\} \times \C^{n-1}$ for every $z \in E$. Hence $u$ fails to be strictly plurisubharmonic at every point of $E \times \C^{n-1}$, i.e., $E \times \C^{n-1} \subset \mathfrak{c}(\Omega)$. On the other hand, $\Phi(z) \coloneqq \varphi(z_1-1,z_2, \ldots, z_n) + \varphi(-z)$ is a smooth plurisubharmonic function on $\Omega$ such that $\Phi$ is strictly plurisubharmonic on $\Omega \setminus (E \times \C^{n-1})$, i.e., $\mathfrak{c}(\Omega) \subset E \times \C^{n-1}$. Thus $\mathfrak{c}(\Omega) = E \times \C^{n-1}$, which completes the proof of part 1 of the theorem.

\noindent (ii) We now prove the statements (2) and (3) of the theorem. At first, let $\Omega \subset \C^n$ be an arbitrary domain such that $\mathfrak{c}(\Omega) = E \times \C^k$ for some set $E \subset \C^{n-k}$. Then, in view of Liouville's theorem, one can easily see that $E = \{z' \in \C^{n-k} : \{z'\} \times \C^k \subset \Omega\}$. For every $z'' \in \C^k$ let $V_{z''} \coloneqq \{z' \in \C^{n-k} : (z',z'') \in \Omega\}$ and define $\psi_{z''} \colon V_{z''} \to [-\infty,\infty)$ as $\psi_{z''}(z') \coloneqq -\log\mathcal{R}(z',z'')$, where $\mathcal{R}(z) \coloneqq \sup\{r>0 : \{z'\} \times B^k(z'',r) \subset \Omega\}$, $z = (z',z'') \in \C^{n-k} \times \C^k$. By definition, $\psi_0(z') = -\infty$ if and only if $\{z'\} \times \C^k \subset \Omega$. Thus $E = \{\psi_0 = -\infty\}$.

Assume now that $\Omega$ is pseudoconvex. Then $\psi_0$ is plurisubharmonic on $V_0$, since $\psi_0(z') = \sup_{w''\in \C^k, \norm{w''} = 1} [-\log \mathcal{R}_{(0,w'')}(z',0)]$, where for every $w \in \C^n$ the function $\mathcal{R}_w(z) \coloneqq \sup\{r >0 : z + \zeta w \in \Omega \text{ for every } \zeta \in \Delta(0,r)\}$ denotes the Hartogs radius of $\Omega$ in the $w$-direction; here $\Delta(a,r) \coloneqq \{z \in \C : \abs{z-a} < r\}$. Thus $E$ is complete pluripolar if $\psi_0 \not\equiv -\infty$ on every connected component of $V_0$. On the other hand, suppose that $\psi_0 \equiv -\infty$ on some connected component $U$ of $V_0$, i.e., $U \times \C^k \subset \Omega$. Assume, to get a contradiction, that $\Omega \neq U \times \C^k$. Then there exists $z'' \in \C^k$ such that $U$ is a proper subset of the connected component $V_{z''}'$ of $V_{z''}$ containing $U$. Since $\psi_{z''} \equiv -\infty$ on the open set $U$, it follows that $\psi_{z''} \equiv -\infty$ on $V_{z''}'$. Thus $V_{z''}' \times \C \subset \Omega$ and hence, by definition of $U$, we have $V_{z''}' \subset U$. This contradicts the fact that $U \subsetneq V_{z''}'$ and thus proves that $\Omega = U \times \C^k$. Another application of Liouville's theorem then shows that $E = U$, which completes the proof of statement $(2)$.

Now assume that $\Omega$ is even strictly pseudoconvex. Then, by what we have already proven, it follows that $E$ is complete pluripolar. Assume, to get a contradiction, that $E$ is not closed. Then there exist $p \in \C^{n-k} \setminus E$ and a sequence $\{p_j\}_{j=1}^\infty \subset E$ such that $\lim_{j \to \infty} p_j = p$. Since $E \times \C^k \subset \Omega$, it follows that $L \coloneqq \{p\} \times \C^k \subset \overline{\Omega}$. By Theorem \ref{thm_qdeffct_manifolds} and the related Remark 4, there exists a continuous plurisubharmonic function $\varphi$ on an open neighbourhood of $\overline{\Omega}$ such that $\Omega = \{\varphi < 0\}$. In particular, $\varphi \le 0$ on $\overline{\Omega}$. Thus, by Liouville's theorem, $\varphi \equiv c$ on $L$ for some constant $c \le 0$. If $c < 0$, then $L \subset \Omega$ and hence also $L \subset \mathfrak{c}(\Omega)$. This implies that $p \in E$, which contradicts the assumption on $p$. On the other hand, if $c = 0$, then $L \subset b\Omega$, which is not possible by strict pseudoconvexity of $b\Omega$. This shows that $E$ is closed in $\C^{n-k}$.

Finally, let $E \subset \C^{n-k}$ be a closed complete pluripolar set. Then, by Corollary 1 in \cite{Coltoiu90}, there exists a plurisubharmonic function $u$ on $\C^{n-k}$ such that $u$ is smooth on $\C^{n-k} \setminus E$ and $E = \{u = -\infty\}$. Define
\[ \Omega' \coloneqq \big\{(z',z'') \in \C^n : u(z') + \norm{z}^2 < C\big\}. \]
Then for generic $C \in \R$, $\Omega'$ is a strictly pseudoconvex open set with smooth boundary such that $E \times \C^k \subset \Omega'$. By Liouville's theorem, $E \times \C^k \subset \mathfrak{c}(\Omega')$. Moreover, let $v \colon \C^n \to [-\infty,\infty)$ be defined as $v(z) = u(z') + \norm{z}^2$. It is easy to see that if $\{\eta_j\}_{j=1}^\infty$ is a sequence of positive numbers that converges to zero fast enough, then $\widetilde{v} \coloneqq \sum_{j=1}^\infty \eta_j \, \widetilde{\max}_1(v - C, -j)$ is a smooth global defining function for $\Omega'$ which is strictly plurisubharmonic outside $E \times \C^k$. Thus we also have that $\mathfrak{c}(\Omega') \subset E \times \C^k$. This shows that $\mathfrak{c}(\Omega') = E \times \C^k$. The assertion of statement (3) then follows from the following lemma.
\end{proof}

\begin{lemma}
Let $\mathcal{M}$ be a connected complex manifold and let $\Omega' \subset \mathcal{M}$ be a strictly pseudoconvex open set (not necessarily connected or with smooth boundary). Then there exists a strictly pseudoconvex domain $\Omega \subset \mathcal{M}$ with smooth boundary such that $\overline{\Omega'} \subset \Omega$ and $\mathfrak{c}(\Omega) = \mathfrak{c}(\Omega')$.
\end{lemma}
\begin{proof}
We proceed in three steps.

\noindent \textsc{Step 1.} {\it Let $G_0,G_1 \subset \C^n$ be two strictly pseudoconvex domains with smooth boundary such that $\overline{G}_0 \cap \overline{G}_1 = \varnothing$. Let $\gamma \colon [0,1] \to \C^n$ be a smooth embedding such that $z_0 \coloneqq \gamma(0) \in bG_0$, $z_1 \coloneqq \gamma(1) \in bG_1$ and $\gamma(t) \in \C^n \setminus (\overline{G}_0 \cup \overline{G}_1)$ for $t \in (0,1)$. Let $\psi$ be a smooth plurisubharmonic function on an open neighbourhood $V \subset \C^n$ of $\overline{G}_0 \cup \overline{G}_1$ such that for $j = 0,1$ we have $\psi(z_j) \le 0$ and $\psi$ is strictly plurisubharmonic near $z_j$. Then for every open neighbourhood $\Gamma \subset \C^n$ of $\gamma([0,1])$ there exist a strictly pseudoconvex domain $\Omega \subset \C^n$ with smooth boundary and a smooth plurisubharmonic function $\varphi$ on an open neighbourhood $U \subset \C^n$ of $\overline{\Omega}$ such that the following assertions hold true:
\begin{itemize}
  \item[(i)] $U = V' \cup \Gamma'$ for some open neighbourhood $V' \subset V$ of $\overline{G}_0 \cup \overline{G}_1$ and some open neighbourhood $\Gamma' \subset \Gamma$ of $\gamma([0,1])$,
  \item[(ii)] $\Omega \setminus \Gamma = (G_0 \cup G_1) \setminus \Gamma$,
  \item[(iii)] $\varphi = \psi$ on $V'$, while $\varphi$ is strictly plurisubharmonic and less than $1$ on $\Gamma'$.
\end{itemize} }

\noindent \textsc{Proof.} Fix constants $\varepsilon_0, \delta_0 > 0$ such that $\overline{B^n(z_0, \varepsilon_0)} \cap \overline{B^n(z_1, \varepsilon_0)} = \varnothing$, $\psi$ is strictly plurisubharmonic and less than $1/2$ on $B^n(z_0,\varepsilon_0) \cup B^n(z_1, \varepsilon_0) \subset V \cap \Gamma$ and such that $\gamma([0,1])^{(\delta_0)} \subset \Gamma$, where for $K \subset \C^n$ and $d>0$ we let $K^{(d)} \coloneqq \bigcup_{z \in K} B^n(z,d)$.

Choose $s > 0$ so small that $\gamma([0,s])$ and $\gamma([1-s,1])$ are contained in $B^n(z_0,\varepsilon_0) \cup B^n(z_1,\varepsilon_0)$. Let $f \colon \gamma([0,1]) \to (-\infty, 1/2)$ be a smooth function such that for some constant $c \in (0,1)$ one has $f + c < \psi$ in $\gamma(0)$ and $\gamma(1)$, and $f > \psi + c$ in $\gamma(s)$ and $\gamma(1-s)$. Let $F \colon \C^n \to \R$ be a smooth extension of $f$. Since one can see easily that $\gamma([0,1])$ is contained in a closed embedded smooth real $1$-dimensional submanifold $M \subset \C^n$, it follows from Lemma 1 in \cite{Chirka69} that there exists a smooth strictly plurisubharmonic function $\theta \colon W \to \R$ on an open neighbourhood $W \subset \C^n$ of $\gamma([0,1])$ such that $\theta \equiv 0$ on $\gamma([0,1])$. Thus for $C > 0$ large enough, and after possibly shrinking $W$, the function $\rho \coloneqq F + C\theta$ is smooth and strictly plurisubharmonic on $W$ such that $\rho < 1/2$ and $\rho|_{\gamma([0,1])} = f$. 

Choose $\varepsilon \in (0, \varepsilon_0)$ so small that
\begin{itemize}
  \item $B^n(z_0,\varepsilon) \cup B^n(z_1,\varepsilon) \subset W$,  
  \item $\gamma([s,1-s]) \cap (\overline{B^n(z_0,\varepsilon)} \cup \overline{B^n(z_1, \varepsilon)}) = \varnothing$, and 
  \item $\rho + c < \psi$ on $B^n(z_0,\varepsilon) \cup B^n(z_1,\varepsilon)$.
\end{itemize}
Moreover, let $\delta \in (0, \min(\delta_0, \varepsilon)/2) $ be so small that 
\begin{itemize}
  \item $\gamma([0,1])^{(2\delta)} \subset W$, 
  \item $\gamma([0,1])^{(2\delta)} \cap (G_0 \cup G_1)^{(\delta)} \subset B^n(z_0,\varepsilon) \cup B^n(z_1, \varepsilon)$,
  \item the orthogonal projection $\pi \colon \gamma([0,1])^{(2\delta)} \to M$ along the normal directions of the manifold $M$ is well defined, and
  \item there exists a constant $a \in (0, s)$ such that $\pi^{-1}(\gamma([0,s+a) \cup (1-s-a,1])) \subset B^n(z_0, \varepsilon_0) \cup B^n(z_1, \varepsilon_0)$, $\pi^{-1}(\gamma((s-a, 1-s+a))) \cap (B^n(z_0,\varepsilon) \cup B^n(z_1, \varepsilon)) = \varnothing$ and $\rho > \psi + c$ on $\pi^{-1}(\gamma((s-a,s+a) \cup (1-s-a, 1-s+a)))$.
\end{itemize}
Let $V' \coloneqq V \cap (G_0 \cup G_1)^{(\delta)}$, let $\Gamma' \coloneqq B^n(z_0,\varepsilon) \cup \gamma([0,1])^{(2\delta)} \cup B^n(z_1,\varepsilon)$ and set $U \coloneqq V' \cup \Gamma'$. Then $\varphi \colon U \to \R$ defined as
\[\varphi \coloneqq \left\{ \begin{array}{c@{\,\text{ on }}l} \psi & V \cap (G_0 \cup G_1)^{(\delta)} \\ \widetilde{\max}_c(\psi, \rho) & B^n(z_0,\varepsilon) \cup \pi^{-1}(\gamma([0,s+a) \cup (1-s-a,1])) \cup B^n(z_1, \varepsilon) \\ \rho & \pi^{-1}(\gamma((s-a, 1-s+a))) \end{array} \right.\]
is a smooth plurisubharmonic function on $U$ such that $\varphi = \psi$ on $V'$, $\varphi < 1$ on $\Gamma'$ and $\varphi$ is strictly plurisubharmonic on $\Gamma'$.

It only remains to construct a strictly pseudoconvex domain $\Omega \subset \C^n$ with smooth boundary such that $\Omega \setminus \Gamma = (G_0 \cup G_1) \setminus \Gamma$ and $\overline{\Omega} \subset U$. To do so, fix $\tilde{\varepsilon} > 0$ so small that $\overline{B^n(z_0,\tilde{\varepsilon})} \cup \overline{B^n(z_1,\tilde{\varepsilon})} \subset U \cap \Gamma$. Then for $j = 0,1$ choose strictly pseudoconvex domains $\widetilde{G}_j \subset \C^n$ with smooth boundary such that $\widetilde{G}_j \subset G_j$, $\widetilde{G}_j \setminus B^n(z_j,\tilde{\varepsilon}) = G_j \setminus B^n(z_j, \tilde{\varepsilon})$ and such that near $z_j$ the domain $\widetilde{G}_j$ looks like a ball with $z_j$ as a boundary point. (The existence of the domains $\widetilde{G}_j$ is essentially an observation by H. Boas, which is based on a result due to Y. Eliashberg. A detailed proof of this fact, together with references to the results of Boas and Eliashberg, can be found, for example, in Corollary 4.1.46 of \cite{JarnickiPflug00}. Observe that our assertions on the domains $\widetilde{G}_j$ are slightly stronger than the ones formulated in the statement of the mentioned above corollary. However, the fact that $\widetilde{G}_j$ can be assumed to be strictly pseudoconvex with smooth boundary follows, for example, from the remark after Corollary 4.1.46 in \cite{JarnickiPflug00}, or from the construction of smooth maximum as described in Section \ref{sec_deffct} of this article.) Moreover, let $\widetilde{\gamma} \colon [0,1] \to \C^n$ be a smooth embedding such that $\widetilde{\gamma}(0) = z_0$, $\widetilde{\gamma}(1) = z_1$, $\widetilde{\gamma}([0,1]) \setminus (B^n(z_0,\tilde{\varepsilon}/2) \cup B^n(z_1,\tilde{\varepsilon}/2)) = \gamma([0,1]) \setminus (B^n(z_0,\tilde{\varepsilon}/2) \cup B^n(z_1,\tilde{\varepsilon}/2))$, and such that for some $r>0$ the curves $\widetilde{\gamma}([0,r])$ and $\widetilde{\gamma}([1-r,1])$ are segments of lines orthogonal to $b\widetilde{G}_0$ and $b\widetilde{G}_1$, respectively. Finally, choose $\tilde{\delta} \in (0, \min(\tilde{\varepsilon}/2, \delta))$. Then, by the corollary in Section 1 of \cite{Shcherbina83}, and after possibly further shrinking $\tilde{\delta}$, there exists a strictly pseudoconvex domain $\Omega \subset \C^n$ with smooth boundary such that
\[ \Omega \setminus (B^n(z_0,\tilde{\varepsilon}/2) \cup B^n(z_1, \tilde{\varepsilon}/2)) = (\widetilde{G}_0 \cup \widetilde{\gamma}([0,1])^{(\tilde{\delta})} \cup \widetilde{G}_1) \setminus (B^n(z_0,\tilde{\varepsilon}/2) \cup B^n(z_1, \tilde{\varepsilon}/2)). \]
(The corollary quoted above is only formulated for domains in $\C^2$, but the statement and the proof remain true also in the case of $\C^n$.) In particular, 
\[ \Omega \setminus (B^n(z_0,\tilde{\varepsilon}) \cup B^n(z_1, \tilde{\varepsilon})) = (G_0 \cup \gamma([0,1])^{(\tilde{\delta})} \cup G_1) \setminus (B^n(z_0,\tilde{\varepsilon}) \cup B^n(z_1, \tilde{\varepsilon})). \]
It now follows easily from the constructions that $\Omega$ is a domain as desired. This completes the proof of Step 1.

\noindent \textsc{Step 2.} {\it The statement of Step 1 remains true if $\C^n$ is replaced by an arbitrary complex manifold $\mathcal{M}$.}

\noindent \textsc{Proof.} Let $D_1, \ldots, D_N \subset \mathcal{M}$ be open coordinate patches such that $\gamma([0,1]) \subset \bigcup_{j=1}^N D_j$, $z_0 \in D_1$, $z_1 \in D_N$, $D_j \cap \gamma([0,1])$ is connected, $1 \le j \le N$, $D_j \cap D_{j+1} \cap \gamma([0,1]) \neq \varnothing$, $1 \le j \le N-1$, and $D_j \cap D_k = \varnothing$ if $\abs{j-k} > 1$. For every $j = 1, \ldots, N-1$, let $\widetilde{G}_j \subset\subset \Gamma \cap (D_j \cap D_{j+1})$ be a strictly pseudoconvex domain with smooth boundary such that the sets $G_0 \eqqcolon \widetilde{G}_0, \widetilde{G}_1, \ldots, \widetilde{G}_{N-1}, \widetilde{G}_N \coloneqq G_1$ have pairwise disjoint closures and there exists numbers $0 \eqqcolon t^1_0 < t^1_1 < t^2_0 < t^2_1 < \cdots < t^N_0 < t^N_1 \coloneqq 1$ such that $\bigcup_{j=1}^N \gamma((t^j_0,t^j_1)) \subset \mathcal{M} \setminus \bigcup_{j=0}^N \overline{\widetilde{G}}_j$ and $\gamma((t^j_1, t^{j+1}_0)) \subset \widetilde{G}_j$, $1 \le j \le N-1$. Define $\widetilde{\gamma}_j \colon [0,1] \to \mathcal{M}$ as $\widetilde{\gamma}_j(t) \coloneqq \gamma(t^j_0 + t(t^j_1 - t^j_0))$. Choose open neighbourhoods $\widetilde{V}_j \subset \mathcal{M}$ of $\overline{\widetilde{G}}_j$ and smooth functions $\widetilde{\psi}_j \colon \widetilde{V}_j \to \R$, $0 \le j \le N$, such that the sets $\widetilde{V}_0, \ldots, \widetilde{V}_N$ are pairwise disjoint, $\widetilde{V}_j \subset V$ and $\widetilde{\psi}_j \equiv \psi$ if $j \in \{0,N\}$, and $\widetilde{\psi}_j$ is a strictly plurisubharmonic global defining function for $\widetilde{G}_j$ such that $\widetilde{\psi}_j < 1$ on $\widetilde{V}_j$ if $j \in \{1, \ldots, N-1\}$. Finally, let $\widetilde{\Gamma}_j \subset \Gamma$, $1 \le j \le N$, be open neighbourhoods of $\widetilde{\gamma}_j([0,1])$ with pairwise disjoint closures. Then application of Step 1 to the tupel $(\widetilde{G}_{j-1}, \widetilde{G}_j, \widetilde{V}_{j-1}, \widetilde{V}_j, \widetilde{\gamma}_j, \widetilde{\Gamma}_j, \widetilde{\psi}_{j-1}, \widetilde{\psi}_j)$ for every $j = 1, \ldots, N$ gives the desired result.

\noindent \textsc{Step 3.} {\it The assertion of the lemma holds true.}

\noindent \textsc{Proof.} By Theorem \ref{thm_nbhbasis_manifolds}, there exists a strictly pseudoconvex open set $G \subset \mathcal{M}$ with smooth boundary such that $\overline{\Omega'} \subset G$ and $\mathfrak{c}(G) = \mathfrak{c}(\Omega')$. 
%(the equality $\mathfrak{c}(G) = \mathfrak{c}(\Omega')$ follows immediately if in both parts of the proof of Theorem \ref{thm_nbhbasis_manifolds} the function $\varphi$ is chosen to be minimal on $\Omega$)
Let $\{G_j\}_{j=1}^N$ be the different connected components of $G$, $N \in \N \cup \{\infty\}$. Fix an arbitrary increasing sequence $\{D_R\}_{R=1}^\infty$ of relatively compact domains $D_R \subset \mathcal{M}$ such that $\bigcup_{R=1}^\infty D_R = \mathcal{M}$. Since $bG$ is smooth, it is easy to see that there exist a family $\{\gamma_j\}_{j=1}^{N-1}$ of smooth embeddings $\gamma_j \colon [0,1] \to \mathcal{M}$ and natural numbers $\nu(j), \mu(j)$, $1 \le j < N$, such that $\gamma_j(0) \in bG_{\nu(j)}$, $\gamma_j(1) \in bG_{\mu(j)}$, $\gamma_j(t) \in \mathcal{M} \setminus \overline{G}$ for $t \in (0,1)$, $\gamma_j([0,1]) \cap \gamma_k([0,1]) = \varnothing$ if $j \neq k$, $\card \{1 \le j < N : D_R \cap \gamma_j([0,1]) \neq \varnothing \}$ is finite for every $R > 0$, and $G \cup \bigcup_{j=1}^{N-1} \gamma_j([0,1])$ is connected. Let $\psi$ be a minimal global defining function for $G$.

Choose open neighbourhoods $\Gamma_j \subset\subset \mathcal{M}$ of $\gamma_j([0,1])$, $1 \le j < N$, such that 
\begin{itemize}
  \item $\overline{\Gamma}_j \cap \overline{\Gamma}_k = \varnothing$ if $j \neq k$,
  \item $\overline{\Gamma}_j \cap \overline{G} \subset \overline{G}_{\nu(j)} \cup \overline{G}_{\mu(j)}$, and
  \item $\Gamma_j \cap \Omega' = \varnothing$.
\end{itemize}
Then for every $1 \le j < N$ we can apply Step 2 to obtain a strictly pseudoconvex domain $\Omega_j \subset \mathcal{M}$ with smooth boundary, an open set $\Gamma_j' \subset \Gamma_j$ and a smooth plurisubharmonic function $\varphi_j$ on $\Omega_j$ such that
\begin{itemize}
  \item $\Omega_j \setminus \Gamma_j' = (G_{\nu(j)} \cup G_{\mu(j)}) \setminus \Gamma_j'$,
  \item $\varphi_j = \psi$ on $\Omega_j \setminus \Gamma_j'$, while $\varphi_j$ is strictly plurisubharmonic and less than $1$ on $\Omega_j \cap \Gamma_j'$.
\end{itemize}
Then define a strictly pseudoconvex domain $\Omega \subset \C^n$ with smooth boundary as
\[ \Omega \coloneqq \Big[ G \setminus \bigcup_{j=1}^{N-1} \Gamma_j' \Big] \cup \bigcup_{j=1}^{N-1} \big[ \Omega_j \cap \Gamma_j' \big] \]
and a smooth plurisubharmonic function $\varphi \colon \Omega \to \R$ as
\[ \varphi \coloneqq \left\{ \begin{array}{c@{\,\text{ on }}l} \psi & \Omega \setminus \bigcup_{j=1}^{N-1} \Gamma_j' \\ \varphi_j & \Omega \cap \Gamma_j' \end{array} \right.. \]
By construction, $\varphi < 1$ on $\Omega$ and $\varphi$ is strictly plurisubharmonic outside $\mathfrak{c}(G) = \mathfrak{c}(\Omega')$. Thus $\mathfrak{c}(\Omega) \subset \mathfrak{c}(\Omega')$. Moreover, observe that, by construction of $\Omega$, one has $\overline{\Omega'} \subset \Omega$, hence $\mathfrak{c}(\Omega') \subset  \mathfrak{c}(\Omega)$. It follows that $\mathfrak{c}(\Omega) = \mathfrak{c}(\Omega')$, which completes the proof the lemma.
\end{proof} 

In order to get a better understanding of properties of the core we now consider some examples.

\begin{example} \label{ex_complexsubspace}
Fix $n \ge 2$ and let $1 \le q \le n-1$. Then for generic $C \in \R$
\[ \Omega \coloneqq \big\{(z,w) \in \C^{n-q} \times \C^q : \log\norm{z} + \big(\norm{z}^2 + \norm{w}^2\big) < C \big\} \]
is an unbounded strictly pseudoconvex domain with smooth boundary. By the Liouville theorem, every plurisubharmonic function $\varphi$ on $\Omega$ that is bounded from above has to be constant on $\{0\} \times \C^q$. Hence $\varphi$ fails to be strictly plurisubharmonic at every point of $\{0\} \times \C^q$, i.e., $\{0\} \times \C^q \subset \mathfrak{c}(\Omega)$. On the other hand, let $\psi \colon \C^n \to \R$ be defined as $\psi(z,w) = \log\norm{z} + (\norm{z}^2 + \norm{w}^2)$. As before, if $\{\eta_j\}_{j=1}^\infty$ is a sequence of positive numbers that converges to zero fast enough, then $\varphi \coloneqq \sum_{j=1}^\infty \eta_j \, \widetilde{\max}_1(\psi - C, -j)$ is a smooth global defining function for $\Omega$ that is strictly plurisubharmonic outside $\{0\} \times \C^q$. This shows that $\mathfrak{c}(\Omega) \subset \{0\} \times \C^q$, and hence $\mathfrak{c}(\Omega) = \{0\} \times \C^q$. In particular, $\mathfrak{c}(\Omega)$ is $q$-pseudoconcave in the sense of Rothstein (see below for more details on $q$-pseudoconcavity).
\end{example}

\begin{example}
Fix $n \ge 2$ and let $1 \le q \le n-1$. Further, fix pairwise distinct points $a_1, a_2, \ldots, a_N \in \C^{n-q}$, $N \ge 2$. Then for generic and large enough $C \in \R$
\[ \Omega \coloneqq \big\{(z,w) \in \C^{n-q} \times \C^q : \sum_{j=1}^N \log\norm{z-a_j} + \big(\norm{z}^2 + \norm{w}^2\big) < C \big\} \]
is an unbounded strictly pseudoconvex domain with smooth boundary. Using the same argument as before, it can be shown that $\mathfrak{c}(\Omega) = \bigcup_{j=1}^N \{a_j\} \times \C^q$. In particular, the core $\mathfrak{c}(\Omega)$ is not connected.
\end{example}

\begin{example} \label{ex_FatouBieberbach}
Let $\Omega' \subset \C^2_{z,w}$ be a Fatou-Bieberbach domain such that $\varnothing \neq \overline{\Omega'} \cap \{w=0\} \subset \Delta(0,1) \times \{0\}$ and $\overline{\Omega' \cap \{w=0\}} = \overline{\Omega'} \cap \{w=0\}$ (the existence of such a domain is guaranteed by Corollary 1.1 in \cite{Globevnik98}). Let $\varepsilon > 0$ and let $\psi \colon \overline{\Delta}(0,1+\varepsilon) \to (-\infty, -C)$ be a smooth superharmonic function, where $C > 0$ is chosen so large that $\{(z,w) \in \C^2 : \abs{z} = 1+\varepsilon, \abs{w} \le e^{\psi(z)}\} \subset \C^2 \setminus \overline{\Omega'}$. Let $\Phi \colon \C^2 \to \Omega'$ be a biholomorphism and define $\Omega \subset \C^2$ as
\[ \Omega \coloneqq \Phi^{-1}\big(\Omega' \cap \big\{(z,w)\in\C^2 : \abs{z} < 1+\varepsilon, \abs{w} < e^{\psi(z)}\big\}\big). \]
After possibly replacing $\Omega$ by one of its connected components, $\Omega$ is an unbounded strictly pseudoconvex domain with smooth boundary. Since $\varphi \coloneqq \norm{\,\cdot\,}^2 \circ \Phi \colon \Omega \to \R$ is a smooth strictly plurisubharmonic function on $\Omega$ that is bounded from above, we see that $\mathfrak{c}(\Omega) = \varnothing$. 
\end{example}

\begin{example}
Let $\Omega \subset \C^n$ be a Fatou-Bieberbach domain or a domain of the form $\Omega = D \times \C^k$ for some domain $D \subset \C^{n-k}$ (more generally $\Omega \subset \C^n$ could be any domain that is the union of holomorphic images of complex lines). Applying Liouville's theorem as before, we conclude that $\mathfrak{c}(\Omega) = \Omega$. It follows easily from our construction of global defining functions that the situation $\mathfrak{c}(\Omega) = \Omega$ cannot happen if $b\Omega$ has points of strict pseudoconvexity.
\end{example}

\begin{example} \label{ex_Riemannpolar}
Let $X$ be a compact Riemann surface and let $E \subset X$ be a polar subset. Then $X \setminus E$ is a Stein manifold of dimension 1, hence there exists a proper holomorphic embedding $F \colon X \setminus E \to \C^3$. Let $g_1, g_2 \colon \C^3 \to \C$ be holomorphic functions such that $F(X \setminus E) = \{z \in \C^3 : g_1(z) = g_2(z) = 0\}$ (see \cite{ForsterRamspott68}). Define 
\[ \Omega \coloneqq \big\{z \in \C^3 : \log\big(\abs{g_1(z)}^2 + \abs{g_2(z)}^2\big) + \norm{z}^2 < C \big\} \]
for generic $C \in \R$. Then, after possibly replacing $\Omega$ by a suitable connected component, $\Omega$ is a strictly pseudoconvex domain with smooth boundary such that $F(X \setminus E) \subset \Omega$. We claim that $\mathfrak{c}(\Omega) = F(X \setminus E)$. Indeed, if $\{\eta_j\}_{j=1}^\infty$ is a sequence of positive numbers that converges to zero fast enough, then the function $\varphi \colon \C^3 \to \R$ defined by $\varphi(z) \coloneqq \sum_{j=1}^\infty \eta_j \widetilde{\max}_1(\log(\abs{g_1(z)}^2 + \abs{g_2(z)}^2) + \norm{z}^2, -j)$ is a smooth plurisubharmonic function that is strictly plurisubharmonic in the complement of $F(X \setminus E)$ and that is bounded from above on $\Omega$, hence $\mathfrak{c}(\Omega) \subset F(X \setminus E)$. On the other hand, if $\varphi \colon \Omega \to \R$ is a plurisubharmonic function that is bounded from above, then $\psi \coloneqq \varphi \circ F|_{X \setminus E}$ extends to a bounded subharmonic function $\widehat{\psi}$ on $X$, and since $X$ is compact we conclude that $\widehat{\psi}$ is constant. This means that $\varphi$ is constant on $F(X \setminus E)$, hence $\varphi$ cannnot be strictly plurisubharmonic at any point of $F(X \setminus E)$. This proves that $F(X \setminus E) \subset \mathfrak{c}(\Omega)$, and hence $\mathfrak{c}(\Omega) = F(X \setminus E)$ as claimed.
\end{example}

\begin{example} \label{ex_CPhypersurface}
Let $H$ be a complex hypersurface in the complex projective space $\CP^n$. Then $\CP^n \setminus H$ is a Stein manifold (see, for example, Corollary V.3.4 in \cite{FritzscheGrauert02}), hence there exists a proper holomorphic embedding $F \colon \CP^n \setminus H \to \C^N$ for some $N \in \N$. Let $g_1, g_2, \ldots, g_k \colon \C^N \to \C$ be holomorphic functions such that $F(\CP^n \setminus H) = \{z \in \C^N : g_1(z) = g_2(z) = \cdots = g_k(z) = 0\}$. Define 
\[ \Omega \coloneqq \big\{z \in \C^N : \log\big(\abs{g_1(z)}^2 + \cdots + \abs{g_k(z)}^2\big) + \norm{z}^2 < C \big\} \]
for generic $C \in \R$. Then, after possibly replacing $\Omega$ by a suitable connected component, $\Omega$ is a strictly pseudoconvex domain with smooth boundary such that $F(\CP^n \setminus H) \subset \Omega$. As before we see that $\mathfrak{c}(\Omega) = F(\CP^n \setminus H)$.
\end{example}

Let now $\Omega \subset \C^n$ be an unbounded strictly pseudoconvex domain with smooth boundary such that the envelope of holomorphy $E(b\Omega)$ of $b\Omega$ is single-sheeted. Then we can define the \textit{CR-core} $\mathfrak{c}_{C\!R}(\Omega)$ of $\Omega$ as $\mathfrak{c}_{C\!R}(\Omega) \coloneqq \Omega \setminus E(b\Omega)$. In some simple cases (for example the domain $\Omega$ from Example \ref{ex_complexline} in the Introduction) one can observe that $\mathfrak{c}_{C\!R}(\Omega) = \mathfrak{c}(\Omega)$ (see also \cite{HarzShcherbinaTomassini12}). This is why one can be tempted to think that the equality $\mathfrak{c}_{C\!R}(\Omega) = \mathfrak{c}(\Omega)$ holds true for every domain $\Omega$ as above. However, we claim that this is false and, moreover, in general the sets $\mathfrak{c}_{C\!R}(\Omega)$ and $\mathfrak{c}(\Omega)$ are not related at all.

\begin{proposition}
\begin{enumerate}
  \item[(1)] There exists an unbounded strictly pseudoconvex domain $\Omega \subset \C^3$ with smooth boundary such that $\mathfrak{c}(\Omega) \neq \varnothing$ but $\mathfrak{c}_{C\!R}(\Omega) = \varnothing$.
  \item[(2)] There exists an unbounded strictly pseudoconvex domain $\Omega \subset \C^2$ with smooth boundary such that $\mathfrak{c}(\Omega) = \varnothing$ but $\mathfrak{c}_{C\!R}(\Omega) \neq \varnothing$.
\end{enumerate}
\end{proposition}
\begin{proof}
Let first $\Omega \subset \C^3$ be the domain from Example \ref{ex_complexsubspace} such that $\mathfrak{c}(\Omega) = \{0\} \times \C \neq \varnothing$. We claim that $\mathfrak{c}_{C\!R}(\Omega) = \varnothing$. Indeed, by strict pseudoconvexity of $b\Omega$, every CR function $f$ on $b\Omega$ extends to a holomorphic function $\tilde{f}$ on a one-sided neighbourhood $U \subset \Omega$ of $b\Omega$.  Further, every slice $S_c \coloneqq \Omega \cap \{w = c\}$ is a ball in $\C^2_z$, hence by Hartogs theorem on removability of compact singularities each function $\tilde{f}|_{S_c \cap U}$ extends to a holomorphic function $F_c \colon S_c \to \C$. An easy investigation of the proof of Hartogs theorem shows that the function $F \colon \Omega \to \R$ defined by $F(z,w) \coloneqq F_w(z)$ is holomorphic in the $w$-variable. But it is clear from the construction that $F$ is also holomorphic in the $z$-variables. By Hartogs theorem on separate analyticity, it follows that $F$ is a holomorphic extension of $\tilde{f}$. Since here $f$ was arbitrary, it follows that $\mathfrak{c}_{C\!R}(\Omega) = \varnothing$. 

On the other hand, let now $\Omega$ be the domain from Example \ref{ex_FatouBieberbach}. We have already seen that $\mathfrak{c}(\Omega) = \varnothing$, and we claim that $\mathfrak{c}_{C\!R}(\Omega) \neq \varnothing$. Indeed, let $\Omega^\ast \subset \C^2$ be a strictly pseudoconvex domain with smooth boundary such that $\overline{\Omega} \subset \Omega^\ast$ (the existence of such a domain 
$\Omega^\ast$ follows, for example, from Theorem \ref{thm_nbhbasis_manifolds}; the other direct way to see this is by repeating the construction of $\Omega$ with $\psi$ replaced by $\psi + \delta$ for a some small enough constant $\delta > 0$). Moreover, let $h \colon \Delta(0, 1+\varepsilon) \to \R$ be a harmonic function such that $h < \psi$. Then $V \coloneqq \Phi^{-1}(\Omega' \cap \{(z,w) \in \C^2 : \abs{z} < 1+\varepsilon, \abs{w} < e^{h(z)}\})$ is an unbounded open set with smooth Levi-flat boundary such that, after possibly replacing $V$ by a suitable connected component, $\overline{V} \subset \Omega$. In particular, $\Omega^\ast \setminus V$ is a pseudoconvex open set, hence there exists a holomorphic function $F \colon (\Omega^\ast \setminus V) \to \C$ that does not extend holomorphically to any larger domain. But, by construction, $b\Omega \subset \Omega^\ast \setminus V$, hence $F|_{b\Omega}$ is a CR function on $b\Omega$ that does not extend holomorphically to any point of $V$. Thus $V \subset \mathfrak{c}_{C\!R}(\Omega)$. 
\end{proof}

The rest of this section is devoted to the proof of $1$-pseudoconcavity of the core (see Theorem \ref{thm_pseudoconcave}). The main step of this proof is contained in the following lemma.
\begin{lemma} \label{thm_psctouch}
Let $\mathcal{M}$ be a complex manifold and let $\Omega$ be a domain in $\mathcal{M}$. Then it is not possible ``to touch'' $\mathfrak{c}(\Omega)$ by a strictly pseudoconvex hypersurface contained in $\Omega$. More precisely, one cannot find a domain $U \subset \Omega$ and a smooth real hypersurface $M \subset U$ such that $U \setminus M$ consists of two connected components $U_1$ and $U_2$, $M \cap \mathfrak{c}(\Omega) \neq \varnothing$, $U \cap \mathfrak{c}(\Omega) \subset \overline{U}_1$ and $U_1$ is strictly pseudoconvex at every point $p \in M$. 
\end{lemma}
\begin{proof}
Assume, to get a contradiction, that there exist $U$ and $M$ as above. Fix $p \in M \cap \mathfrak{c}(\Omega)$. After possibly shrinking $U$ and performing a local biholomorphic change of variables, we can assume that $U \subset \C^n$ and that $U_1$ is strictly convex at every point of $M$. By slightly enlarging $U_1$ we can choose a smooth real hypersurface $M' \subset U$ such that $U \setminus M'$ consists of two connected components $U_1'$ and $U_2'$, $U_1 \subset U_1'$, $M' \cap M = \{p\}$ and $U_1'$ is strictly convex at every point of $M'$. Moreover, we may assume without loss of generality that $p = 0$ and that the outward unit normal vector to $U_1'$ at $0$ equals $e_{y_n} \coloneqq (0, \ldots, 0, i) \in \C^n$, $z_j = x_j + iy_j$, $j= 1, 2, \ldots, n$. Let $\widetilde{G} \subset U$ be the domain bounded by $M''\coloneqq M' + \varepsilon_1 e_{y_n}$ and $\big\{y_n = \varepsilon_2 (\abs{z_1}^2 + \cdots + \abs{z_{n-1}}^2 + x_n^2) - \varepsilon_3\big\}$, where $\varepsilon_1, \varepsilon_2, \varepsilon_3$ are small positive constants, and let $G \subset \widetilde{G}$ be a domain obtained by smoothing the wedge of $\widetilde{G}$. Then for suitably choosen $\varepsilon_1, \varepsilon_2, \varepsilon_3$ and a good enough smoothing of $\widetilde{G}$ the domain $G$ is a strictly convex smoothly bounded domain in $U$ such that $bG \cap \{y_n > - \frac{\varepsilon_3}{2}\} \subset M'' \subset U \setminus \mathfrak{c}(\Omega)$. (A suitable smoothing of $\widetilde{G}$ is obtained as follows: Since the outward unit normal to $U_1'$ is $e_{y_n}$, there exists a smooth strictly concave function $f \colon \C^{n-1}_{z_1, \ldots, z_{n-1}} \times \R_{x_n} \to \R_{y_n}$ such that $M''$ is contained in the graph of $f$. Then for $\delta > 0$ small enough let $u \coloneqq \widetilde{\max}_\delta \big(y_n - f(z_1, \ldots, z_{n-1}, x_n)$, $\varepsilon_2 (\abs{z_1}^2 + \cdots + \abs{z_{n-1}}^2 + x_n^2) - \varepsilon_3 - y_n \big)$ and set $G \coloneqq \{u < 0\}$.)

Now let $\varphi$ be a smooth and bounded from above plurisubharmonic function on $\Omega$ that is strictly plurisubharmonic on $\Omega \setminus \mathfrak{c}(\Omega)$ (see the definition of minimal functions, which was stated after the Main Theorem, and the related remarks). Let $\widetilde{\varphi} \colon \overline{G} \to \R$ be the maximal plurisubharmonic function such that $\widetilde{\varphi}|_{bG} = \varphi$. Since $M'' \cap \mathfrak{c}(\Omega) = \varnothing$, $\varphi$ is strictly plurisubharmonic near $M''$ and hence $\varphi < \widetilde{\varphi}$ in a one-sided neighbourhood $W \subset G$ of $M'' \cap bG$ (indeed, for $z_0 \in G \setminus \mathfrak{c}(\Omega)$ the function $\psi(z) \coloneqq \varphi(z) + \max(\delta_1 - \delta_2\norm{z-z_0}^2, 0)$ with $0 < \delta_1 << \delta_2 << 1$ is plurisubharmonic on $G$ with $\psi|_{bG} = \varphi$, hence $\varphi(z_0) < \psi(z_0) \le \widetilde{\varphi}(z_0)$ by maximality of $\widetilde{\varphi}$). We want to show that $\varphi < \widetilde{\varphi}$ holds not only on $W$, but in fact on the whole set $G \cap \{y_n > -\frac{\varepsilon_3}{3}\} \ni 0$. If we have done so, then $\psi \coloneqq \delta_{\gamma_2} \ast (\widetilde{\varphi} - \gamma_1) + \gamma_3 \norm{\,\cdot\,}^2$, where $\gamma_1 \coloneqq (\widetilde{\varphi}(0) - \varphi(0))/2$, $\gamma_2$ and $\gamma_3$ are small enough positive constants and $\delta_{\gamma_2}$ is a smooth nonnegative function depending only on $\norm{z}$ such that $\supp \delta_{\gamma_2} = \overline{B^n(0, \gamma_2)}$ and $\int_{\C^n} \delta_{\gamma_2} = 1$, is a smooth strictly plurisubharmonic function on $\overline{G}_{\gamma_2} \coloneqq \{z \in G : \dist(z, bG) \ge \gamma_2 \}$ such that $\psi + \delta < \varphi$ on $bG_{\gamma_2}$ and $\psi(0) > \varphi(0) + \delta$ for some $\delta > 0$. In particular, $\widetilde{\max}_\delta(\varphi, \psi)$ is a smooth and bounded from above plurisubharmonic function on $\Omega$ that is strictly plurisubharmonic in $0$. This contradicts the fact that $0 \in \mathfrak{c}(\Omega)$. 

In order to show that $\varphi < \widetilde{\varphi}$ on $G \cap \{y_n > -\frac{\varepsilon_3}{3}\}$ let $G' \subset \subset G$ be a smoothly bounded strictly convex domain such that $bG' \cap \{y_n \ge -\frac{\varepsilon_3}{3}\} \subset W$ and $G \cap \{y_n \ge -\frac{\varepsilon_3}{3}\} \setminus W \subset G'$. Since $\varphi < \widetilde{\varphi}$ on $W$, the function $h \colon [-\frac{\varepsilon_3}{3}, \varepsilon] \to \R$ defined by $h(t) \coloneqq \min_{bG' \cap \{y_n=t\}} (\widetilde{\varphi}- \varphi)$ is strictly positive, where $\varepsilon \coloneqq \sup_{G'} y_n > 0$. In particular, we can choose a smooth function $\chiup \colon (-\infty, \varepsilon] \to \R$ such that $\chiup|_{(-\varepsilon_3/3, \varepsilon]}$ is strictly convex, $\chiup(t) = 0$ for $-\infty < t \le -\frac{\varepsilon_3}{3}$ and $0 < \chiup(t) < h(t)$ for $-\frac{\varepsilon_3}{3} < t \le \varepsilon$. Let $\rho \colon \overline{G'} \to \R$ be defined as $\rho(z) \coloneqq \chiup(y_n)$ and observe that $\rho$ is plurisubharmonic. Then, by construction of $\rho$, one has $\varphi + \rho \le \widetilde{\varphi}$ on $b\big(G' \cap \{y_n > -\frac{\varepsilon_3}{3}\} \big)$, and hence $\varphi + \rho \le \widetilde{\varphi}$ on $G' \cap \{y_n > -\frac{\varepsilon_3}{3}\}$ by maximality of $\widetilde{\varphi}$. Since $\rho > 0$ on $\{y_n > -\frac{\varepsilon_3}{3}\}$, this proves our claim.
\end{proof}

As the first consequence of Lemma \ref{thm_psctouch} we get the following property of the core.
\begin{proposition} \label{thm_coreStein}
Let $\mathcal{M}$ be a Stein manifold and let $\Omega$ be a domain in $\mathcal{M}$. Then no connected component of $\mathfrak{c}(\Omega)$ can be relatively compact in $\Omega$.
\end{proposition}
\begin{proof}
Assume, to get a contradiction, that $A$ is a connected component of $\mathfrak{c}(\Omega)$ which is relatively compact. Let $\varphi$ be a smooth strictly plurisubharmonic exhaustion function for $\mathcal{M}$ and let $C \in \R$ be the minimal value such that $A \subset \{\varphi \le C\}$. It may happen that $C$ is not a regular value of $\varphi$. In this case choose $p \in A \cap \{\varphi = C\}$  and let $U \subset \mathcal{M}$ be an open coordinate patch around $p$ with corresponding chart $h \colon U \to \C^n$. Choose $\delta > 0$ so small that $h(U)_\delta \coloneqq \{z \in h(U) : \dist(z, b(h(U))) > \delta\}$ still contains $h(p)$, and let $U' \coloneqq h^{-1}(h(U)_\delta)$. For each $v \in \C^n$, let $\tau_v \colon \C^n \to \C^n$ be the translation $\tau_v(z) \coloneqq z - v$ and define a map $g \colon B^n(0,\delta) \to \R$ by $g(v) \coloneqq \max_{z \in A \cap U'} (\varphi \circ h^{-1} \circ \tau_v \circ h)(z)$. Then the image of $g$ contains an open intervall $I \subset \R$ and by Sard's theorem there exists a regular value $C' \in I$ of $\varphi$. Let $v' \in B^n(0,\delta)$ such that $g(v') = C'$. Then $\varphi' \colon U' \to \R$ defined by $\varphi' \coloneqq \varphi \circ h^{-1} \circ \tau_{v'} \circ h$ is a smooth strictly plurisubharmonic function and $\max_{p \in A \cap U'} \varphi' = C'$. In particular, $\{\varphi' = C'\}$ is a smooth strictly pseudoconvex hypersurface that touches $\mathfrak{c}(\Omega)$ as described in Lemma \ref{thm_psctouch}. 
\end{proof}

\noindent \textbf{Remark.} As Example \ref{ex_compactanalytic} from the Introduction shows, the core $\mathfrak{c}(\Omega)$ can be relatively compact in $\Omega$ if $\mathcal{M}$ is not Stein.

Now we want to use Lemma \ref{thm_psctouch} to prove that $\mathfrak{c}(\Omega)$ is always $1$-pseudoconcave. We recall briefly the notion of $q$-pseudoconcavity: Let $\Delta^n \coloneqq \{z \in \C^n : \norm{z}_\infty < 1 \}$, where $\norm{z}_\infty = \max_{1 \le j \le n} \abs{z_j}$. An $(n-q,q)$ Hartogs figure $H$ is a set of the form
\[H = \big\{(z,w) \in \Delta^{n-q} \times \Delta^q :  \norm{z}_\infty > r_1 \text{ or } \norm{w}_\infty < r_2\big\}\]
where $0 < r_1,r_2 < 1$, and we write $\hat{H} \coloneqq \Delta^n$. A domain $\Omega$ in a complex manifold $\mathcal{M}$, $\dim_\C \mathcal{M} \eqqcolon n$, is called $q$-pseudoconvex in $\mathcal{M}$, $q = 0, 1, \ldots, n-1$, if it satisfies the Kontinuit\"atssatz with respect to $(n-q)$ polydiscs in $\mathcal{M}$, i.e., if for every $(n-q,q)$ Hartogs figure and every injective holomorphic mapping $\Phi \colon \hat{H} \to \mathcal{M}$ such that $\Phi(H) \subset \Omega$ we have $\Phi(\hat{H}) \subset \Omega$ (for details see \cite{Rothstein55}; a good presentation of this topic can also be found in \cite{Riemenschneider67}). In particular, every domain $\Omega \subset \mathcal{M}$ is $0$-pseudoconvex, $(n-1)$-pseudoconvexity is usual pseudoconvexity, and every $q$-pseudoconvex domain is $q'$-pseudoconvex for every $q' < q$. A closed set $A \subset \mathcal{M}$ is called $q$-pseudoconcave in $\mathcal{M}$ if $\mathcal{M} \setminus A$ is $q$-pseudoconvex in $\mathcal{M}$. (The above definition of $q$-pseudoconvexity is due to Rothstein, see \cite{Rothstein55}. Observe that the earlier definition of strict $q$-pseudoconvexity that was introduced in the smooth case by Andreotti-Grauert in \cite{AndreottiGrauert62} and that we stated in Section \ref{sec_deffct} is indexed differently with respect to $q$ when compared to the definition of Rothstein. Indeed, a domain $\Omega \subset \mathcal{M}$ that is strictly $q$-pseudoconvex in the sense of Andreotti-Grauert is $(n-q-1)$-pseudoconvex in the sense of Rothstein. Moreover, a domain $\Omega \subset \mathcal{M}$ with $\mathcal{C}^2$-smooth boundary is $q$-pseudoconvex in the sense of Rothstein if and only if for every $z \in b\Omega$ there exists an open neighbourhood $U_z \subset \mathcal{M}$ of $z$ and a $\mathcal{C}^2$-smooth function $\varphi_z \colon U_z \to \R$ such that $\Omega \cap U_z = \{\varphi_z < 0\}$, $d\varphi \neq 0$ on $b\Omega \cap U_z$ and $\Lev(\varphi)(z,\,\cdot\,)|_{H_z(\varphi)}$ has at least $q$ nonnegative eigenvalues. For the rest of this article $q$-pseudoconvexity and $q$-pseudoconcavity will always be understood in the sense of Rothstein.)

Now we can formulate the main result of this section.
\begin{theorem}\label{thm_pseudoconcave}
Let $\mathcal{M}$ be a complex manifold and let $\Omega \subset \mathcal{M}$ be a domain. Then $\mathfrak{c}(\Omega)$ is $1$-pseudoconcave in $\Omega$. In particular, $\mathfrak{c}(\Omega)$ is pseudoconcave in $\Omega$ if $\dim_\C \mathcal{M} = 2$.
\end{theorem}
\begin{proof}
Assume, to get a contradiction, that $\mathfrak{c}(\Omega)$ is not $1$-pseudoconcave in $\Omega$. Then there exists an $(n-1,1)$ Hartogs figure $H = \big\{(z,w) \in \Delta^{n-1} \times \Delta :  \norm{z}_\infty > r_1 \text{ or } \abs{w} < r_2\big\}$ and an injective holomorphic mapping $\Phi \colon \hat{H} \to \Omega$ such that $\Phi(H) \subset \Omega \setminus \mathfrak{c}(\Omega)$ but $\Phi(\hat{H}) \cap \mathfrak{c}(\Omega) \neq \varnothing$. For small $\varepsilon > 0$ let $\varphi \colon \C^{n-1}_z \times \C^\ast_w \to \R$ be the smooth strictly plurisubharmonic function defined by $\varphi(z,w) \coloneqq -\log\abs{w} + \varepsilon \big(\norm{z}^2 + \abs{w}^2\big)$, and for each $C \in \R$ let $G_C$ denote the domain $G_C \coloneqq \big\{p \in \Phi(\hat{H}) : (\varphi \circ \Phi^{-1})(p) < C\big\}$. Since for $C$ large enough the set $\hat{H} \cap \{(z,w) \in \C^n : \varphi < C\}$ contains $\hat{H} \setminus H$, and since $\Phi(\hat{H}) \cap \mathfrak{c}(\Omega) \subset \Phi(\hat{H} \setminus H)$, we know that for $C$ large enough $\Phi(\hat{H}) \cap \mathfrak{c}(\Omega) \subset G_C$. Let $C_0 \coloneqq \inf \{C \in \R : \Phi(\hat{H}) \cap \mathfrak{c}(\Omega) \subset G_C\}$. Then $M \coloneqq bG_{C_0} \cap \Phi(\hat{H})$ is a strictly pseudoconvex hypersurface that touches $\mathfrak{c}(\Omega)$ as described in Lemma \ref{thm_psctouch} (observe that $\{\varphi = C_0\} \cap b\hat{H} \subset b\Delta_z^{n-1} \times \Delta_w$ if $\varepsilon << 1$). Since the lemma states that such hypersurfaces cannot exist, we arrived at a contradiction.
\end{proof}

\noindent \textbf{Remarks.} 1) Observe that it follows from Example \ref{ex_complexsubspace} above that in general $\mathfrak{c}(\Omega)$ is not $q$-pseudoconcave in $\Omega$ for any $q > 1$.

\noindent 2) A different proof of Theorem \ref{thm_pseudoconcave} can also be obtained by modifying the arguments of the proof of Theorem 3.6 in \cite{SlodkowskiTomassini04} and adapting them to our setting.

Recall that, by Theorems 4.2 and 5.1 in \cite{Slodkowski86}, a nonempty relatively closed subset $A$ of an open set $U \subset \C^n$ is $(q+1)$-pseudoconcave in $U$ if and only if $q$-plurisubharmonic functions have the local maximum property on $A$. An analoguous statement is also true in the setting of complex manifolds. Since we were not able to find this statement in the literature, and since in the more general setting the precise formulation of the local maximum property needs a little bit of caution, we state here the following proposition for the convenience of reading.

\begin{proposition} \label{thm_lmp}
Let $\mathcal{M}$ be a complex manifold of dimension $n$, let $A \subset \mathcal{M}$ be a closed set and let $q \in \{0, 1, \ldots, n-2\}$. Then the following assertions are equivalent:
\begin{enumerate}
  \item[$(1)$] For every $p \in A$, there exists an open neighbourhood $U \subset \mathcal{M}$ of $p$ such that $A \cap U$ is $(q+1)$-pseudoconcave in $U$.
  \item[$(1')$] A is $(q+1)$-pseudoconcave in $\mathcal{M}$.
  \item[$(2)$] For every $p \in A$, there exists an open neighbourhood $U \subset \mathcal{M}$ of $p$ such that for every compact set $K \subset U$ and every $q$-plurisubharmonic function $\varphi$ defined in a neighbourhood of $K$ one has $\max_{A \cap K} \varphi \le \max_{A \cap bK} \varphi$.
\end{enumerate}
If $\mathcal{M}$ is Stein, then the above statements are also equivalent to the following one:
\begin{enumerate}
  \item[$(2')$] For every compact set $K \subset \mathcal{M}$ and every $q$-plurisubharmonic function $\varphi$ defined in a neighbourhood of $K$, one has $\max_{A \cap K} \varphi \le \max_{A \cap bK} \varphi$.
\end{enumerate}
Here $\max_{A \cap bK} \varphi$ is meant to be $-\infty$ if $A \cap bK = \varnothing$.
\end{proposition}

\noindent \textbf{Remark.} If $\mathcal{M}$ is not Stein, then in general the assertion $(2')$ does not follow from $(2)$, as it is shown by the following simple examples:
\begin{enumerate}
  \item[i)]  $\mathcal{M} = \CP^1_z \times \C^{n-1}_w$, $A = \CP^1_z \times \{0\}$, $K = \CP^1_z \times \overline{B^{n-1}(0,1)}$ and $\varphi(z,w) = \norm{w}^2$.
  \item[ii)]  $\mathcal{M}_q = \CP^{n-q}_z \times \C^q_w$, $A_q = \CP^{n-q}_z \times \overline{{B^q}(0,1)}$, $K_q = A_q$ and $\varphi_q(z,w) = -\norm{w}^2$, where $q \in \{1,2,\ldots, n-1\}$.
\end{enumerate}

\begin{proof}
The implication $(1') \Rightarrow (1)$ is clear and the implication $(1) \Rightarrow (2)$ follows from Theorems 4.2 and 5.1 of \cite{Slodkowski86}. We will show that also $(2) \Rightarrow (1')$. Indeed, let $A$ have the properties from $(2)$ and assume, to get a contradiction, that $A$ is not $(q+1)$-pseudoconcave in $\mathcal{M}$. Then, by the same kind of arguments as in the proof of Theorem \ref{thm_pseudoconcave}, we can find an open set $V \subset \mathcal{M}$ and a smooth real hypersurface $M \subset V$ such that $V \setminus M$ consists of two connected components $V_1$ and $V_2$, $M \cap A \neq \varnothing$, $V \cap A \subset \overline{V}_1$ and $V_1$ is strictly $q$-pseudoconvex at every point of $M$. After possibly shrinking $V$ and perturbing $M$, we can assume that $M \cap A = \{p\}$ for some $p \in V$. Let $U \subset V$ be an arbitrary neighbourhood of $p$. Let $W \subset\subset U$ be another open neighbourhood of $p$ and let $\varphi$ be a smooth strictly $q$-plurisubharmonic function defined near $\overline{W}$ such that $V_1 \cap W = \{\varphi|_W < 0\}$. Then for $K \coloneqq \overline{W}$ we have $\max_{A \cap K} \varphi > \max_{A \cap bK} \varphi$. This contradicts the assumptions in $(1)$.

It remains to consider statement $(2')$. Clearly, one always has that $(2') \Rightarrow (2)$. Now let $\mathcal{M}$ be Stein and let $A$ satisfy the properties from $(2)$. Assume, to get a contradiction, that there exists a compact set $K \subset \mathcal{M}$ and a $q$-plurisubharmonic function $\varphi$ defined in a neighbourhood of $K$ such that $\max_{A \cap K} \varphi  > \max_{A \cap bK} \varphi$. Let $m \coloneqq \max_{A \cap K} \varphi$ and consider the set $L \coloneqq \{z \in A \cap K : \varphi(z) = m\}$. Since $\mathcal{M}$ is Stein, we can use the same arguments as in Proposition \ref{thm_coreStein} to obtain an open set $V \subset \mathcal{M}$ and a smooth real hypersurface $M \subset V$ such that $V \setminus M$ consists of two connected components $V_1$ and $V_2$, $M \cap L \neq \varnothing$, $V \cap L \subset \overline{V}_1$, and $V_1$ is strictly pseudoconvex at every point of $M$. After possibly shrinking $V$, and after introducing suitable holomorphic coordinates, we can assume that $V \subset \C^n$ and that $V_1$ is strictly convex at every point of $M$. Fix arbitrary $p \in L \cap M$ and let $U \subset\subset V$ be an open neighbourhood of $p$ as described in $(2)$. Without loss of generality we can assume that $p = 0$. By strict convexity of $M$, we can then choose an $\R$-linear functional $\lambda \colon V \to \R$ such that $\lambda \le 0$ on $V \cap L$ and $\{\lambda = 0\} \cap L = \{p\}$. Let $W \subset\subset U$ be another open neighbourhood of $p$. Then one sees easily that for $\widetilde{K} \coloneqq \overline{W}$ and for $\varepsilon > 0$ small enough the $q$-plurisubharmonic function $\widetilde{\varphi} \coloneqq \varphi + \varepsilon\lambda \colon V \to \R$ satisfies $\max_{A \cap \widetilde{K}} \widetilde{\varphi} > \max_{A \cap b\widetilde{K}} \widetilde{\varphi}$. But this contradicts the choice of $U$.
\end{proof}

We conclude this section by a brief discussion on the role of $1$-pseudoconcavity of $\mathfrak{c}(\Omega)$. Namely, in light of the results of the next section, we want to point out that for our purpose it is reasonable to interpret $1$-pseudoconcavity as a generalized notion of analytic structure. This viewpoint is motivated by the following simple lemma, which will be used in the proof of Theorem \ref{thm_analyticcore} below, and which is an easy consequence of the above mentioned results of S{\l}odkowski.

\begin{lemma} \label{thm_pseudoconcaveconstant}
Let $\mathcal{M}$ be a complex manifold and let $A \subset \mathcal{M}$ be closed and $1$-pseudoconcave in $\mathcal{M}$. Then every plurisubharmonic function $\varphi$ which is defined on an open neighbourhood of $A$ and which is constant on $A$ fails to be strictly plurisubharmonic at every point of $A$.
\end{lemma}
\begin{proof}
Let $\varphi$ be a plurisubharmonic function defined on an open neighbourhood of $A$ such that $\varphi$ is constant on $A$. Assume, to get a contradiction, that there exists $z \in A$ such that $\varphi$ is strictly plurisubharmonic on a small open coordinate neighbourhood $U \subset \mathcal{M}$ of $z$. Let $\theta \colon \mathcal{M} \to [0,\infty)$ be a nonconstant smooth function with compact support in $U$, and choose $\varepsilon > 0$ so small that $\psi \coloneqq \varphi + \varepsilon\theta$ is still plurisubharmonic. Then $\psi$ attains a local maximum along the $1$-pseudoconcave set $A$. But this is not possible, since plurisubharmonic functions have the local maximum property on $A$, see \cite{Slodkowski86}. 
\end{proof}

To further support our interpretation of $1$-pseudoconcavity, we also want to formulate the following version of Rossi's local maximum modulus principle. It is easily achieved from Rossi's original result by applying S{\l}odkowski's characterization of $1$-pseudoconcave sets as local maximum sets for absolute values of holomorphic functions (the original theorem of Rossi is contained in [Ros]; a formulation of this result which is better suited for our purpose can be found, for example, in Theorem 2.1.8 of [St]). This version most likely was known to some people before, therefore we do not claim any originality for its proof.

\begin{proposition}
Let $K \subset \C^n$ be a compact set and let $z_0$ in $\C^n$.  Let $\hat{K}$ denote the polynomial hull of $K$. Then the following assertions are equivalent:
\begin{itemize}
  \item[(1)] $z_0 \in \hat{K} \setminus K$.
  \item[(2)] There exists a connected bounded locally closed set $\lambda \subset \C^n \setminus K$ with the following properties:
  \begin{itemize}
    \item[(i)] $\lambda$ is $1$-pseudoconcave in $\C^n \setminus K$.
    \item[(ii)] $\bar{\lambda} \setminus \lambda \neq \varnothing$ and $\bar{\lambda} \setminus \lambda \subset K$.
    \item[(iii)] $z_0 \in \lambda$.
  \end{itemize}
\end{itemize}
\end{proposition}
\begin{proof}
Let first $z_0 \in \hat{K} \setminus K$ be an arbitrary fixed point. Define $\lambda \subset \C^n$ to be the connected component of $\hat{K} \setminus K$ that contains $z_0$. Then, by definition, $\lambda$ is closed in $\C^n \setminus K$, but, by the Shilov idempotent theorem (see, for example, Corollary 6.5 in \cite{Gamelin69}), $\lambda$ is not closed in $\C^n$. Thus $\bar{\lambda} \setminus \lambda \neq \varnothing$ and $\bar{\lambda} \setminus \lambda \subset K$. Moreover, Rossi's local maximum modulus principle states that absolute values of holomorphic polynomials have the local maximum property on $\lambda$, i.e., $\lambda$ is $1$-pseudoconcave in $\C^n \setminus K$, see [Sl3].

For the other direction, fix a set $\lambda \subset \C^n$ such that $\lambda$ satisfies all the properties (i)-(iii) above. Assume, to get a contradiction, that $z_0 \notin \hat{K}$. Then there exists a holomorphic polynomial $p$ on $\C^n$ such that $\abs{p(z_0)} > \max_{z \in K} \abs{p(z)}$. Hence, slightly shrinking $\lambda$, one will find a compact set $L \subset \lambda \subset \hat{K} \setminus K$ such that $z_0 \in L$ and $\abs{p(z_0)} > \max_{z \in b_\lambda L} \abs{p(z)}$, where $b_\lambda L$ denotes the relative boundary of $L$ in $\lambda$. But, in view of the results from [Sl3], this contradicts $1$-pseudoconcavity of $\lambda$. 
\end{proof}

Finally, we want to mention the following result due to Fornaess-Sibony (see Corollary 2.6 in \cite{FornaessSibony95}): Let $T$ be a positive closed current of bidimension $(p,p)$ on $\C^n$, $1 \le p \le n-1$. Then the support of $T$ is $p$-pseudoconcave in $\C^n$ (hence, in particular, it is $1$-pseudoconcave in $\C^n$). 

%Finally, we mention two more results, which further support our interpretation of $1$-pseudoconcavity: First, for each compact set $K \subset \C^n$ the set $\hat{K} \setminus K$ is $1$-pseudoconcave in $\C^n \setminus K$, where $\hat{K}$ denotes the polynomial hull of $K$. (This follows immediately from Rossi's local maximum modulus principle, see, for example, Theorem 2.1.8 in \cite{Stout07}, and from the results in \cite{Slodkowski86}.) And second, the support of every positive closed current of bidimension $(p,p)$ on $\C^n$, $1 \le p \le n-1$, is $1$-pseudoconcave in $\C^n$, see Corollary 2.6 in \cite{FornaessSibony95}.

%
%
%
%
%
\section{A core with no analytic structure} \label{sec_noanalyticcore}
Observe that in each example of a domain $\Omega \subset \C^n$ such that $\mathfrak{c}(\Omega) \neq \varnothing$ that we have seen so far, the core $\mathfrak{c}(\Omega)$ is an analytic subset of $\Omega$. In this section we investigate the question whether this is a general phenomenon, i.e., whether $\mathfrak{c}(\Omega)$ always carries an analytic structure. We will show that this is not the case by proving the following theorem.

\begin{theorem} \label{thm_analyticcore}
For every $n \ge 2$, there exists an unbounded strictly pseudoconvex domain $\Omega \subset \C^n$ with smooth boundary such that $\mathfrak{c}(\Omega)$ is nonempty and contains no analytic variety of positive dimension.
\end{theorem}

In fact, $\Omega$ will be the strictly pseudoconvex domain constructed in Theorem 1.2 of \cite{HarzShcherbinaTomassini12} and $\mathfrak{c}(\Omega)$ will coincide with the Wermer type set $\mathcal{E} \subset \Omega$ of Theorem 1.1 in \cite{HarzShcherbinaTomassini12}. In particular, $\mathfrak{c}(\Omega)$ is connected, there exists a plurisubharmonic function $\varphi \colon \C^n \to [-\infty, \infty)$ such that $\mathfrak{c}(\Omega) = \{z \in \C^n : \varphi = -\infty\}$, $\varphi$ is pluriharmonic on $\C^n \setminus \mathfrak{c}(\Omega)$, $\C^n \setminus \mathfrak{c}(\Omega)$ is pseudoconvex and for every $R>0$ one has $\widehat{bB^n(0,R) \cap \mathfrak{c}(\Omega)} = \overline{B^n(0,R)} \cap \mathfrak{c}(\Omega)$, where $\widehat{b B^n(0,R)\cap\mathfrak{c}(\Omega)}$ denotes the polynomial hull of the set $b B^n(0,R) \cap \mathfrak{c}(\Omega)$. In the proof of Theorem \ref{thm_analyticcore} we will also use a Liouville theorem for Wermer type sets. Since this result is of independent interest, we state it in explicit form. 

\begin{theorem} \label{thm_liouville}
Let $\mathcal{E} \subset \C^n$ be the Wermer type set of Theorem 1.1 in \cite{HarzShcherbinaTomassini12}. If $\varphi^\ast$ is a plurisubharmonic function defined on an open neighbourhood of $\mathcal{E}$ and $\varphi^\ast$ is bounded from above, then $\varphi^\ast$ is constant on $\mathcal{E}$. 
\end{theorem}
\noindent \textbf{Remark.} Observe that the above theorem implies, in particular, that the set $\mathcal{E}$ is connected. A more geometric proof of this fact will be given in Lemma \ref{thm_connectedE} below

\indent We recall briefly  the construction of the Wermer type set $\mathcal{E}$ carried out in \cite{HarzShcherbinaTomassini12}. Let $(z,w) = (z_1, \ldots, z_{n-1},w)$ denote the coordinates in $\C^n$ and for each $\nu \in \N$ let $\N_\nu \coloneqq \{1, 2, \ldots, \nu\}$. For each $p \in \N_{n-1}$, fix an everywhere dense subset $\{a_l^p\}_{l=1}^\infty$ of $\C$ such that $a_l^p \neq a_{l'}^p$ if $l \neq l'$. Further, fix a bijection $\Phi \coloneqq ([\,\cdot\,], \phi) \colon \N \to \N_{n-1} \times \N$ and define a sequence $\{a_l\}_{l=1}^\infty$ in $\C$ by letting $a_l \coloneqq a^{[l]}_{\phi(l)}$. Moreover, let $\{\varepsilon_l\}_{l=1}^\infty$ be a decreasing sequence of positive numbers converging to zero that we consider to be fixed, but that will be further specified later on. For every $\nu \in \N$, we define a set $E_\nu \subset \C^n$ as
\[ E_\nu \coloneqq \big\{ (z,w) \in \C^n : w = \sum_{l=1}^\nu \varepsilon_l \sqrt{z_{[l]} - a_l} \big\}. \]
Note that $\sum_{l=1}^\nu \varepsilon_l \sqrt{z_{[l]} - a_l}$ takes $2^\nu$ values at each point $z \in \C^{n-1}$ (counted with multiplicities). Thus there exist single-valued functions $w^{(\nu)}_1, \ldots, w^{(\nu)}_{2^\nu}$ on $\C^{n-1}$ such that 
\[ \sum_{l=1}^\nu \varepsilon_l \sqrt{z_{[l]} - a_l} = \big\{w_j^{(\nu)}(z) : j = 1, \ldots, 2^\nu \big\} \]
for all $z \in \C^{n-1}$. For every $\nu\in\N$, define a function $P_\nu \colon \C^n \to \C$ as
\[ P_\nu(z,w) \coloneqq \bigr(w-w^{(\nu)}_1(z)\bigl) \cdots \bigr(w - w^{(\nu)}_{2^\nu}(z)\bigl). \]
Then each $P_\nu$ is a well defined holomorphic polynomial. Moreover, provided that $\{\varepsilon_l\}$ is decreasing fast enough, the following additional assertions hold true: The sets $E_\nu = \{P_\nu = 0\}$ converge towards a nonempty unbounded connected closed set $\mathcal{E} \subset \C^n$, where the convergence is understood with respect to the Hausdorff metric on each compact subset of $\C^n$. The limit set $\mathcal{E}$ is pseudoconcave and contains no analytic variety of positive dimension. The map $\underline{\mathcal{E}} \colon (\C^{n-1},d_{\norm{\cdot}}) \to (\mathcal{F}(\C), d_H)$, $\underline{\mathcal{E}}(z) \coloneqq \{w \in \C : (z,w) \in \mathcal{E}\}$, from the metric space $\C^{n-1}$ of all $(n-1)$-tupels of complex numbers with the standard euclidean metric $d_{\norm{\cdot}}$ to the metric space $\mathcal{F}(\C)$ of all nonempty compact subsets of $\C$ with the Hausdorff metric $d_H$ is continuous. For every $R > 0$ one has $\widehat{bB^n(0,R) \cap \mathcal{E}} = \overline{B^n(0,R)} \cap \mathcal{E}.$ The sequence $\{\varphi_\nu\}_{\nu=1}^\infty$ of functions $\varphi_\nu \colon \C^n \to [-\infty, +\infty)$ defined as 
\[ \varphi_\nu(z,w) \coloneqq \frac{1}{2^\nu} \log\abs{P_\nu(z,w)}\]
converges uniformly on compact subsets of $\C^n \setminus \mathcal{E}$ to a pluriharmonic function $\varphi \colon \C^n \setminus \mathcal{E} \to \R$, and $\lim_{(z,w) \to (z_0,w_0)}\varphi(z,w) = -\infty$ for every $(z_0,w_0) \in \mathcal{E}$. In particular, $\varphi$ has a unique extension to a plurisubharmonic function on $\C^n$, and $\mathcal{E} = \{\varphi = -\infty\}$ is complete pluripolar. Finally, for fixed generic $C \in \R$, and after possibly passing to a suitable connected component, the set
\begin{equation} \label{equ_defOmega} \Omega \coloneqq \big\{(z,w) \in \C^n : \varphi(z,w) + \big(\norm{z}^2 + \abs{w}^2\big) < C \big\} \end{equation}
is an unbounded strictly pseudoconvex domain with smooth boundary that contains $\mathcal{E}$, and $\Omega \setminus \mathcal{E}$ is pseudoconvex.

\noindent \textbf{Remark.} The properties of continuity of the map $\underline{\mathcal{E}}$ and of connectedness of the set $\mathcal{E}$ were not stated explicitly in \cite{HarzShcherbinaTomassini12}. The proofs of these facts will be given in Lemma \ref{thm_haushölder} and Lemma \ref{thm_connectedE} below.

We want to use the above constructions in the proof of Theorem \ref{thm_analyticcore}. To do so, we will assume for the moment that the statement of Theorem \ref{thm_liouville} is true. The proof of this fact will be postponed to a later part of this section.

\noindent {\bf Proof of Theorem \ref{thm_analyticcore}.} In what follows, we will assume that the sequence $\{\varepsilon_l\}$ is decreasing so fast that the Wermer type set $\mathcal{E}$ and the associated function $\varphi$ have all the properties described above. We have to show that the domain $\Omega$ defined in $(\ref{equ_defOmega})$ satisfies $\mathfrak{c}(\Omega) = \mathcal{E}$.

First, define $\psi \colon \C^n \to [-\infty, \infty)$ as $\psi(z,w) \coloneqq \varphi(z,w) + (\norm{z}^2 + \abs{w}^2) - C$. As before, if $\{\eta_j\}_{j=1}^\infty$ is a sequence of positive numbers that converges to $0$ fast enough, then $\varphi^\ast \coloneqq \sum_{j=1}^\infty \eta_j\, \widetilde{\max}_1(\psi, -j)$ is a smooth global defining function for $\Omega$ that is strictly plurisubharmonic outside $\mathcal{E}$. In particular, $\mathfrak{c}(\Omega) \subset \mathcal{E}$.

Let now $\varphi^\ast \colon \Omega \to \R$ be a smooth plurisubharmonic function that is bounded from above. We know from Theorem \ref{thm_liouville} that $\varphi^\ast \equiv C^\ast$ on $\mathcal{E}$ for some constant $C^\ast \in \R$. Moreover, since $\Omega \setminus \mathcal{E}$ is pseudoconvex, the set $\mathcal{E}$ is $1$-pseudoconcave. Thus, by Lemma \ref{thm_pseudoconcaveconstant}, $\varphi^\ast$ fails to be strictly plurisubharmonic at every point of $\mathcal{E}$. 
This shows that $\mathcal{E} \subset \mathfrak{c}(\Omega)$. $\hfill \Box$

We now turn to the proof of Theorem \ref{thm_liouville}. We start by recalling some facts from potential theory, which will be needed in the course of the proof, mainly to fix our notation: Let $M$ be a Riemann surface and let $D \subset\subset M$, $D \neq M$, be a relatively compact open subset. For every $f \colon bD \to \R$, the associated Perron function 
\begin{equation*} \begin{split}
H_Df \coloneqq \sup\{&u \colon D \to [-\infty,\infty) : u \text{ subharmonic} \\ &\text{and } \limsup_{z\to\zeta} u(z) \le f(\zeta) \text{ for every } \zeta \in bD \}
\end{split} \end{equation*}
is harmonic in $D$. If $z \in D$ is fixed, then the assignment $z \mapsto (H_Df)(z)$ is a positive linear functional on $\mathcal{C}^0(bD)$. Hence there exists a unique Radon measure $\omega_D(z, \cdot)$ on the Borel $\sigma$-algebra of $bD$, called the harmonic measure with respect to $D$ and $z$, such that
\begin{equation} \label{equ_generalpoisson}
(H_Df)(z) = \int_{bD} f(\zeta) \,d\omega_D(z, \zeta)
\end{equation}
for every $f \in \mathcal{C}^0(bD)$. It turns out that $(\ref{equ_generalpoisson})$ remains true for arbitrary bounded Borel measurable functions on $bD$. In particular, it holds for characteristic functions $\chi_E \colon bD \to \{0,1\}$ of Borel sets $E \subset bD$. Thus
\[ \omega_D(z,E) = (H_D \chiup_E)(z) \]
for every Borel set $E \subset bD$ and $z \in D$, and $\omega_D(\,\cdot\,, E) \colon D \to \R$ is harmonic. If $D$ is regular with respect to the Dirichlet problem (this is always satisfied if $bD$ is smooth), and if $f$ is continuous at $\zeta \in bD$, then
\[ \lim_{z \to \zeta} (H_Df)(z) = f(\zeta). \]
\noindent \textbf{Proof of Theorem \ref{thm_liouville}.} We proceed in two steps.

\noindent \textsc{Step 1.} {\it The theorem holds true in the case $n=2$.}

\noindent \textsc{Proof.} Since $\C^2 \setminus \mathcal{E}$ is pseudoconvex, the assignment $z \mapsto \sup_{w \in \mathcal{E}_z} \varphi^\ast(z,w)$, where for every $z \in \C$ the set $\mathcal{E}_z \coloneqq \{w \in \C : (z,w) \in \mathcal{E}\}$ denotes the fiber of $\mathcal{E}$ over $z$, defines a subharmonic function on $\C$ (see \cite{Slodkowski81}, Theorem II). Moreover, by assumption on $\varphi^\ast$, this function is bounded from above. Hence, by the classical Liouville theorem, it is constant, i.e., there exists $C^\ast \in [-\infty, \infty)$ such that $\sup_{w \in \mathcal{E}_z} \varphi^\ast(z,w) = C^\ast$ for every $z \in \C$. Observe that for $C^\ast = -\infty$ this already proves our claim, hence without loss of generality we can assume that $C^\ast \in \R$. We want to show that $\varphi^\ast \equiv C^\ast$ on $\mathcal{E}$, so, in order to get a contradiction, assume that $\varphi^\ast \not\equiv C^\ast$ on $\mathcal{E}$. Clearly, in this case there exists a point $\tilde{p} = (\tilde{z}, \tilde{w}) \in \mathcal{E}$ such that $\varphi^\ast(\tilde{p}) < C^\ast$. By continuity of $\underline{\mathcal{E}}$, and by upper semicontinuity of $\varphi^\ast$, we can assume that $\tilde{z} \notin \{a_l\}_{l=1}^\infty$ and that for some positive numbers $\delta, \rho > 0$ we have $\varphi^\ast < C^\ast - \delta$ on the ball $B^2(\tilde{p}, \rho)$. Fix $\nu_0 \in \N$ such that $\sum_{l=\nu_0+1}^\infty \varepsilon_l \sqrt{\abs{z-a_l}} < \rho/3$ for $z \in \Delta(\tilde{z},\rho)$. Further, choose $p_0 = (z_0, w_0) \in \mathcal{E}$ such that $\varphi^\ast(p_0) = C^\ast$, and a bounded smoothly bounded domain $U \subset \C_z$ such that $a_1, a_2, \ldots, a_{\nu_0}, z_0 \in U$, $\tilde{z} \in bU$ and $bU \subset \C \setminus \{a_l\}_{l=1}^\infty$. Now the general idea of the proof is to show that the harmonic measure of $\mathcal{E} \cap b(U \times \C_w) \cap B^2(\tilde{p}, \rho)$ with respect to the set $\mathcal{E} \cap (U \times \C_w)$ and the point $p_0 \in \mathcal{E} \cap (U \times \C_w)$ is positive, and hence, since $\varphi^\ast \le C^\ast$ on $\mathcal{E}$ and $\varphi^\ast < C^\ast$ on $B^2(\tilde{p}, \rho)$, that $\varphi^\ast(p_0) < C^\ast$. However, in order to have a decent notion of harmonic measure available, we need to approximate $\mathcal{E}$ by the analytic varieties $E_\nu$ and, by performing desingularizations $\pi_\nu \colon F_\nu \to E_\nu$, translate the situation into a problem on Riemann surfaces $F_\nu$. The setup is as follows:

For every $\nu \in \N$, let $f_\nu \colon \C^{\nu+1} \to \C^\nu$ be the holomorphic mapping
\[ f_\nu(z, w_1', \ldots, w_\nu') \coloneqq \big(w'^2_1 - \varepsilon_1^2(z-a_1), \ldots, w'^2_\nu - \varepsilon_\nu^2(z-a_\nu) \big). \]
Then, using the fact that $a_l \neq a_{l'}$ for $l \neq l'$, one immediately sees that $F_\nu \coloneqq \{f_\nu = 0\}$ is a one-dimensional complex submanifold of $\C^{\nu+1}$. For every $\nu \ge \nu_0$, define holomorphic projections 
\begin{equation*} \begin{split}
\pi_\nu \colon F_\nu \to E_\nu, &\quad \pi_\nu(z,w_1', \ldots, w_\nu') \coloneqq \big(z, \sum\nolimits_{l=1}^\nu w_l'\big) \\
P_\nu \colon F_\nu \to F_{\nu_0}, &\quad P_\nu(z,w_1', \ldots, w_\nu') \coloneqq (z, w_1', \ldots, w_{\nu_0}').
\end{split} \end{equation*}
Since $(z,w) \in E_\nu$ if and only if $w = \sum_{l=1}^\nu \varepsilon_l \sqrt{z-a_l}$, and since $(z,w_1', \ldots, w_\nu') \in F_\nu$ if and only if $w_l' = \varepsilon_l \sqrt{z-a_l}$ for every $l \in \N_\nu$ (with the obvious abuse of notation), these maps are indeed well defined. Let $V_\nu \coloneqq E_\nu \cap (U \times \C_w)$ and $W_\nu \coloneqq F_\nu \cap (U \times \C_{w'}^\nu)$, $\nu \in \N$. Then $V_\nu$ is an analytic subvariety of $E_\nu$ with boundary $bV_\nu = E_\nu \cap b(U \times \C_w)$, and $W_\nu$ is a relatively compact open subset of $F_\nu$ with boundary $bW_\nu = F_\nu \cap b(U \times \C_{w'}^\nu)$. Since $bU \subset \C \setminus \{a_l\}_{l=1}^\infty$, it is easy to see that for every $\nu \in \N$ the Riemann surface $F_\nu$ intersects the corresponding set $b(U \times \C_{w'}^\nu)$ transversally. In particular, the boundary of $W_\nu$ is smooth and hence $W_\nu$ is regular with respect to the Dirichlet problem. Finally, for $\nu \ge \nu_0$ choose points $q_\nu = (z_0, w'(\nu)) \in W_\nu$ such that $P_\nu(q_\nu) = q_{\nu_0}$ for every $\nu \ge \nu_0$ and $\lim_{\nu \to \infty} \pi_\nu(q_\nu) = p_0$.

We consider now the sets
\[  X \coloneqq \pi_{\nu_0}^{-1}\big(bV_{\nu_0} \cap B^2(\tilde{p}, \rho/3)\big) \quad \text{and} \quad X_\nu \coloneqq \pi_\nu^{-1}\big(bV_\nu \cap B^2(\tilde{p}, 2\rho/3)\big), \;\nu > \nu_0. \] 
Since the complex curve $F_{\nu}$ intersects $b(U \times \C_{w'}^\nu)$ transversally, and since $W_\nu$ is relatively compact in $F_\nu$, it follows that $bW_\nu$ is a compact smooth manifold of real dimension $1$ for every $\nu \ge \nu_0$. The mappings $\pi_\nu \colon F_\nu \to E_\nu$ are continuous and satisfy $\pi_\nu(bW_\nu) = bV_\nu$, thus $X$ is open in $bW_{\nu_0}$ and $X_\nu$ is open in $bW_\nu$ for every $\nu > \nu_0$. Moreover, applying Sard's theorem simultaneously to the smooth functions $r_\nu \colon bW_\nu \to \R$, $r_\nu \coloneqq \norm{\pi_\nu(\,\cdot\,) - \tilde{p}}^2$, $\nu \ge \nu_0$, we see that, after a slight perturbation of $\rho$, we can assume that the relative boundary $b_{bW_{\nu_0}}X$ of $X$ in $bW_{\nu_0}$, and the relative boundaries $b_{bW_\nu}X_\nu$ of $X_\nu$ in $bW_\nu$, $\nu > \nu_0$, consist of at most finitely many points. Since one sees easily that the map $P_\nu$ is finite, it then follows that $P_\nu^{-1}(b_{bW_{\nu_0}}X) \cup b_{bW_\nu}X_\nu \subset bW_\nu$ is a finite set for every $\nu > \nu_0$.

We come to the main point of the proof. We claim that $P_\nu^{-1} (X) \subset X_\nu$ for every $\nu > \nu_0$. Indeed, let $q = (z, w_1', \ldots, w_\nu') \in P_\nu^{-1}(X)$. Then $(\pi_{\nu_0} \circ P_\nu)(q) \in bV_{\nu_0} \subset b(U \times \C_w)$. In particular, $z \in bU$, i.e., $\pi_\nu(q) \in E_\nu \cap b(U \times \C_w) = bV_\nu$. Moreover, since $\big(z, \sum_{l=1}^{\nu_0} w_l'\big) = (\pi_{\nu_0} \circ P_\nu)(q) \in B^2(\tilde{p}, \rho/3)$, and since $\abs{w_l'} = \varepsilon_l \sqrt{\abs{z-a_l}}$ for every $l \in \N$ and $\sum_{l = \nu_0+1}^\infty \varepsilon_l \sqrt{\abs{z-a_l}} < \rho/3$, we also have $\pi_\nu(q) = \big(z, \sum_{l=1}^{\nu} w_l'\big) \in B^2(\tilde{p}, 2\rho/3)$. This shows that $P_\nu^{-1} (X) \subset X_\nu$ and, therefore, also that $\chiup_{X_\nu} \ge \chiup_X \circ P_\nu$. It follows now from regularity of $W_\nu$ and $W_{\nu_0}$ with respect to the Dirichlet problem, and from continuity of the functions $\chiup_{X_\nu} \colon bW_\nu \to \{0,1\}$ and $\chiup_X \colon bW_{\nu_0} \to \{0,1\}$ outside the finite sets $b_{bW_\nu}X_\nu$ and $b_{bW_{\nu_0}}X$, respectively, that $H_{W_\nu} \chiup_{X_\nu} \ge (H_{W_{\nu_0}} \chiup_X) \circ P_\nu$. Evaluating this inequality at the point $q_\nu \in W_\nu$ and using $P_\nu(q_\nu) = q_{\nu_0}$, we see that 
\begin{equation} \label{equ_harmonicmeasure}
\omega_{W_\nu}(q_\nu, X_\nu) \ge \omega_{W_{\nu_0}}(q_{\nu_0}, X) \quad \text{for} \quad \nu > \nu_0.
\end{equation}
Since the condition $a_1, a_2, \ldots, a_{\nu_0} \in U$ implies  that $W_{\nu_0}$ is connected, we obviously have that
\[ \omega_{W_{\nu_0}}(q_{\nu_0}, X) > 0. \]
Moreover, since for $\nu \in \N$ large enough the function $\varphi^\ast \circ \pi_\nu$ is well defined and subharmonic near $W_\nu \subset F_\nu$, we also have that 
\begin{equation*} \begin{split}
\varphi^\ast(\pi_\nu(q_\nu)) &\le H_{W_\nu} (\varphi^\ast \circ \pi_\nu|_{bW_\nu}) (q_\nu) = \int_{bW_\nu} (\varphi^\ast \circ \pi_\nu)(\zeta) \,d\omega_{W_\nu}(q_\nu, \zeta) \\
& = \int_{X_\nu} (\varphi^\ast \circ \pi_\nu)(\zeta) \,d\omega_{W_\nu}(q_\nu, \zeta) + \int_{bW_\nu \setminus X_\nu} (\varphi^\ast \circ \pi_\nu)(\zeta) \,d\omega_{W_\nu}(q_\nu, \zeta).
\end{split} \end{equation*}
Observe now that, since $\varphi^\ast$ is upper semicontinuous, $\varphi^\ast \le C^\ast$ on $\mathcal{E}$ and $\lim_{\nu \to \infty} V_\nu = \mathcal{E} \cap (U \times \C_w)$ in the Hausdorff metric, there exists a sequence $\{\delta_\nu\}$ of positive numbers such that $\varphi^\ast < C^\ast + \delta_\nu$ on $V_\nu$ and $\lim_{\nu \to \infty} \delta_\nu = 0$. Moreover, $\varphi^\ast < C^\ast-\delta$ on $B^2(\tilde{p}, \rho)$. Hence we get from the above estimate that
\[ \varphi^\ast(\pi_\nu(q_\nu)) \le (C^\ast-\delta)\omega_{W_\nu}(q_\nu, X_\nu) + (C^\ast+\delta_\nu)\big(1 - \omega_{W_\nu}(q_\nu, X_\nu)\big), \]
and then, in view of $(\ref{equ_harmonicmeasure})$, we conclude that for $\nu \ge \nu_0$
\[ \varphi^\ast(\pi_\nu(q_\nu)) \le (C^\ast-\delta)\omega_{W_{\nu_0}}(q_{\nu_0}, X) + (C^\ast+\delta_\nu)\big(1 - \omega_{W_{\nu_0}}(q_{\nu_0}, X_{\nu_0})\big). \]
Since $\omega_{W_{\nu_0}}(q_{\nu_0}, X) > 0$, and since $\lim_{\nu \to \infty} \delta_\nu = 0$, this implies that $\varphi^\ast(\pi_\nu(q_\nu)) < C^\ast$ for every $\nu \ge \nu_0'$ if $\nu_0' \ge \nu_0$ is large enough. Finally, there exists $\mu \ge \nu_0'$ such that $\norm{p_0 - \pi_\mu(q_\mu)} < \rho/3$, hence applying the same reasoning as above to the translated variety $V_\mu' \coloneqq V_\mu + (p_0 - \pi_{\mu}(q_\mu))$ for large enough $\mu$ we see that also $\varphi^\ast(p_0) < C^\ast$ (here we use that $X_\mu + (p_0 - \pi_{\mu}(q_\mu)) \in B^2(\tilde{p},\rho)$). This contradicts the fact that $\varphi^\ast(p_0) = C^\ast$ and completes the proof of Step 1. 

\noindent \textsc{Step 2.} {\it The theorem holds true in the case $n > 2$.}

\noindent \textsc{Proof.} We first show that $\varphi^\ast(z_0,w) = \varphi^\ast(z_0,w')$ for every $z_0 \in \C^{n-1}$ and every $w,w' \in \mathcal{E}_{z_0}$. To do so, fix $z_0 \in \C^{n-1}$ and consider for every $p \in \N_{n-1}$ the set
\[ \tilde{\mathcal{E}}_p \coloneqq \big\{ (z,w) \in \C^{n-1} \times \C : z_j = z_{0,j} \text{ for } j \in \N_{n-1} \setminus \{p\}, w =\!\!\! \sum_{l=1, [l] = p}^\infty \!\!\! \varepsilon_l\sqrt{z_p-a_l} \,\big\}. \]
Observe that for $\tilde{\mathcal{E}}_p(z_0) \coloneqq \{ w \in \C : (z_0,w) \in \tilde{\mathcal{E}_p} \}$ we have
\[\mathcal{E}_{z_0} = \tilde{\mathcal{E}}_1(z_0) + \cdots + \tilde{\mathcal{E}}_{n-1}(z_0). \]
Thus for arbitrary fixed $w,w' \in \mathcal{E}_{z_0}$ we can write  
\[ \begin{array}{r@{\,=\,}l} w & w[1] + \cdots + w[n-1] \\[1ex] w' & w'[1] + \cdots + w'[n-1] \end{array} \!,\; \text{where } w[p], w'[p] \in \tilde{\mathcal{E}}_p(z_0). \]
Define 
\begin{equation*} \begin{split}
w_p &\coloneqq w'[1] + \cdots + w'[p-1] + w[p] + \cdots + w[n-1], \; p \in \N_n, \\ \tilde{w}_p &\coloneqq w'[1] + \cdots + w'[p-1] + w[p+1] + \cdots + w[n-1], \; p \in \N_{n-1},  
\end{split} \end{equation*}
and observe that 
\[ (z_0,w_p), (z_0,w_{p+1}) \in \tilde{\mathcal{E}}_p + (0,\tilde{w}_p) \subset \mathcal{E} \text{ for every } p \in \N_{n-1}. \] 
Since, up to a suitable embedding $i_{z_0,\tilde{w}_p} \colon \C^2_{z_p,w} \hookrightarrow \C^n$, the set $\tilde{\mathcal{E}}_p + (0,\tilde{w}_p)$ is a Wermer type set in $\C^2$, it follows from Step 1 that $\varphi^\ast$ is constant on $\tilde{\mathcal{E}}_p + (0,\tilde{w}_p)$, and hence $\varphi^\ast(z_0,w_p) = \varphi^\ast(z_0,w_{p+1})$ for each $p \in \N_{n-1}$. But $w_1 = w$ and $w_n = w'$. Thus $\varphi^\ast(z_0,w) = \varphi^\ast(z_0,w')$, as claimed.

Now for every choice of $p \in \N_{n-1}$ and  $\xi = (\xi', \xi'') \in \C^{p-1} \times \C^{n-p-1}$ consider the set
\[ \mathcal{E}_{p, \xi} \coloneqq \mathcal{E} \cap \big[(\{\xi'\} \times \C_{z_p} \times \{\xi''\})  \times \C_w \big].\]
Observe that, up to embedding of $\C^2_{z_p,w}$ into $\C^n$, each set $\mathcal{E}_{p,\xi}$ is of the form $\mathcal{E}_{p, \xi} = \bigcup_{w \in W} [\mathcal{E}_p + (0,w)]$ for a Wermer type set $\mathcal{E}_p \subset \C^2_{z_p,w}$ and a suitable set $W = W(p,\xi) \subset \C_w$. Since, by Step 1, the function $\varphi^\ast$ is constant on $\mathcal{E}_p + (0,w)$ for every $w \in W$, and since we have already shown that the values of $\varphi^\ast$ on $\mathcal{E}$ depend only on the $z$-coordinate, it follows that $\varphi^\ast \equiv C_{p,\xi}$ on $\mathcal{E}_{p,\xi}$ for some constant $C_{p,\xi} \in [-\infty, \infty)$. Now if $(z,w)$ and $(z',w')$ are arbitrary points of $\mathcal{E}$, let $\xi_p \coloneqq (z_1', \ldots, z_{p-1}',z_{p+1}, \ldots, z_{n-1})$ for every $p \in \N_{n-1}$ and observe that in the sequence 
\[ (z,w) \in \mathcal{E}_{1, \xi_1}, \mathcal{E}_{2, \xi_2}, \ldots, \mathcal{E}_{n-2, \xi_{n-2}}, \mathcal{E}_{n-1, \xi_{n-1}} \ni (z',w') \]
we have $\mathcal{E}_{j, \xi_j} \cap \mathcal{E}_{j+1, \xi_{j+1}} \neq \varnothing$ for every $j \in \N_{n-2}$. It follows that $C_{p,\xi_p} = C_{p'\!,\xi_{p'}}$ for every $p, p' \in \N_{n-1}$, and thus $\varphi^\ast(z,w) = \varphi^\ast(z',w')$. Since $(z,w), (z',w') \in \mathcal{E}$ were arbitrary, the proof is complete. $\hfill\Box$

\section{Liouville type properties of the core} \label{sec_liouville}

In all examples of strictly pseudoconvex domains $\Omega \subset \C^n$ such that $\mathfrak{c}(\Omega) \neq \varnothing$, which we have constructed so far, the core has the following Liouville type property: if $\varphi$ is a smooth and bounded from above plurisubharmonic function on $\Omega$, then $\varphi$ is constant on every connected component of $\mathfrak{c}(\Omega)$. Thus it is natural to ask whether this is a general property of the core, i.e., we want to investigate whether every connected component of the core of a strictly pseudoconvex domain $\Omega \subset \C^n$ satisfies a Liouville type theorem. 

Further interest to this question is derived from $1$-pseudoconcavity of $\mathfrak{c}(\Omega)$, see Theorem \ref{thm_pseudoconcave} above, and the fact that the following easy lemma holds true.

\begin{lemma} \label{thm_liouvillepseudoconcave}
Let $\mathcal{M}$ be a complex manifold and let $L \subset \mathcal{M}$ be a closed set consisting of more than one point such that every smooth and bounded from above plurisubharmonic function defined near $L$ is constant on $L$. Then $L$ is $1$-pseudoconcave in $\mathcal{M}$.
\end{lemma}
\begin{proof}
Assume, to get a contradiction, that $L$ is not $1$-pseudoconcave in $\mathcal{M}$. Then, using the same argument as in the proof of Theorem \ref{thm_pseudoconcave}, we can find an open set $U \subset \mathcal{M}$ and a smooth real hypersurface $M \subset U$ such that $U \setminus M$ consists of two connected components $U_1$ and $U_2$, $M \cap L \neq \varnothing$, $U \cap L \subset \overline{U}_1$ and $U_1$ is strictly pseudoconvex at every point of $M$. After possibly shrinking $U$ and perturbing $M$, we can assume that $M \cap L = \{p\}$ for some $p \in U$ and that there exists a smooth strictly plurisubharmonic function $\widetilde{\varphi} \colon U \to (-\infty,1]$ defined on an open neighbourhood of $\overline{U}$ such that $U_1 = \{\widetilde{\varphi}|_U < 0\}$. Let $\tilde{c} \coloneqq \max_{L \cap bU} \widetilde{\varphi}$ and set $c \coloneqq \max(\tilde{c}, -1)$. Then for $\delta > 0$ small enough the trivial extension of $\widetilde{\max}_\delta(\widetilde{\varphi}, c/2) \colon U \to \R$ by $c/2$ defines a smooth and bounded from above plurisubharmonic function $\varphi$ on a suitable neighbourhood of $L$ such that $\varphi$ is not constant on $L$. This is a contradiction.
\end{proof}

We will prove in this section that a Liouville type theorem holds true for the core of highest order, i.e., for the set $\mathfrak{c}_n(\Omega) \subset \mathfrak{c}(\Omega)$ of all points $z \in \Omega$ where every smooth and bounded from above plurisubharmonic function $\varphi \colon \Omega \to \R$ satisfies $\Lev(\varphi)(z, \,\cdot\,) \equiv 0$ (for the general definition of cores of higher order see Section \ref{sec_higherordercore}). More precisely, we will prove the following theorem.

\begin{theorem} \label{thm_highliouville}
Let $\mathcal{M}$ be a complex manifold of complex dimension $n$ and let $\Omega \subset \mathcal{M}$ be a domain. Then every smooth and bounded from above plurisubharmonic function on $\Omega$ is constant on each connected component of $\mathfrak{c}_n(\Omega)$.
\end{theorem}

However, we will show that in general no analogue of Theorem \ref{thm_highliouville} holds true for $\mathfrak{c}(\Omega)$, even if $\Omega$ is strictly pseudoconvex. In particular, we will construct strictly pseudoconvex domains $\Omega \subset \C^2_z \times \C^{n-2}_w$ that are bounded in the $z$-directions such that the core has the form $\mathfrak{c}(\Omega) = A \times L$ for some connected set $A \subset \C^2$ that consists of more than one point and some connected set $L \subset \C^{n-2}$ with the property that every smooth and bounded from above plurisubharmonic function defined near $L$ is constant on $L$. Then for $z_0 \in \C^2$ the function $\varphi(z,w) \coloneqq \norm{z-z_0}^2$ is a smooth and bounded from above plurisubharmonic function on $\Omega$, but for almost every choice of $z_0$ the function $\varphi$ is not constant on $\mathfrak{c}(\Omega)$.

These examples show that the connected components of $\mathfrak{c}(\Omega)$ in general do not satisfy a Liouville type theorem. However, observe that here $\mathfrak{c}(\Omega) = \bigcupdot_{\alpha \in A} \{\alpha\} \times L$, where each set $L_\alpha \coloneqq \{\alpha\} \times L$ has the property that smooth and bounded from above plurisubharmonic functions are constant on $L_\alpha$ (here the symbol $\bigcupdot$ is used in order to indicate that the union is disjoint). Thus the question arises whether one can still formulate a Liouville type theorem for suitably defined ``irreducible components'' of $\mathfrak{c}(\Omega)$ instead of connected components. At the moment we do not know whether this is possible or not. However, at least in the $2$-dimensional case we are able to give a partial answer. It is contained in the Theorem \ref{thm_foliatedcore} and the subsequent remarks below. (Note that a local version of Theorem \ref{thm_foliatedcore} in the different setting of exhaustion functions was given earlier in Lemma 4.1 of \cite{SlodkowskiTomassini04}.)

Before stating the result, we recall some definitions: Let $M$ be a smooth manifold. An immersed submanifold of $M$ is a subset $S \subset M$ endowed with a topology with respect to which it is a topological manifold, and a smooth structure with respect to which the inclusion map $i \colon S \hookrightarrow M$ is a smooth immersion. An immersed submanifold $S \subset M$ is called weakly embedded in $M$ if every smooth map $f \colon N \to M$ from a smooth manifold $N$ to $M$ that satisfies $f(N) \subset S$ is smooth as a map from $N$ to $S$. An immersed submanifold $S \subset M$ is called complete if for every complete Riemannian metric $g$ on $M$ the induced metric $i^\ast g$ on $S$ is complete (a Riemannian metric $g$ on $M$ is called complete if the metric on $M$ that is induced by $g$ turns $M$ into a complete metric space). Now let $\mathcal{M}$ be a complex manifold. An immersed complex submanifold of $\mathcal{M}$ is a subset $\mathcal{S} \subset \mathcal{M}$ endowed with a topology with respect to which it is a topological manifold, and a complex structure with respect to which the inclusion map $i \colon \mathcal{S} \hookrightarrow \mathcal{M}$ is a holomorphic immersion. An immersed complex submanifold $\mathcal{S} \subset \mathcal{M}$ will be called weakly embedded or complete if the underlying smooth manifold is weakly embedded or complete, respectively. By a complex curve $\gamma \subset \mathcal{M}$ we will mean a $1$-dimensional immersed complex submanifold of $\mathcal{M}$.

\begin{theorem} \label{thm_foliatedcore}
Let $\mathcal{M}$ be a $2$-dimensional complex manifold and let $\Omega \subset \mathcal{M}$ be a domain. Then the following assertions hold true:
\begin{enumerate}
  \item[(1)] Let $\varphi \colon \Omega \to \R$ be a minimal function for $\Omega$. Then for every regular value $t \in \R$ of $\varphi$ there exists a family $\{\gamma_\alpha\}_{\alpha \in \mathcal{A}}$ (possibly empty) of weakly embedded complete connected complex curves $\gamma_ \alpha \subset \mathcal{M}$ such that 
\[ \mathfrak{c}_t \coloneqq \mathfrak{c}(\Omega) \cap \{\varphi = t\} = \bigcupdot_{\alpha \in \mathcal{A}} \gamma_\alpha.\]
Moreover, there is a decomposition $\mathcal{A} = \mathcal{A}' \cupdot \mathcal{A}''$ such that
\[ \mathring{\mathfrak{c}}_t = \bigcupdot_{\alpha \in \mathcal{A}'} \gamma_\alpha \quad \text{and} \quad b\mathfrak{c}_t = \bigcupdot_{\alpha \in \mathcal{A}''} \gamma_\alpha, \]
where $\mathring{\mathfrak{c}}_t$ and $b\mathfrak{c}_t$ denote the interior and the boundary of $\mathfrak{c}_t$ in the relative topology of $\{\varphi = t\}$, respectively.
  \item[(2)] Let $\varphi \colon \Omega \to \R$ be a smooth and bounded from above plurisubharmonic function, let $t \in \R$ be a regular value of $\varphi$ and let $\gamma \subset \mathfrak{c}(\Omega) \cap \{\varphi = t\}$ be a connected complex curve. Then every smooth and bounded from above plurisubharmonic function $\widetilde{\varphi} \colon \Omega \to \R$ is constant on $\gamma$.
\end{enumerate}
\end{theorem}

\noindent \textbf{Remark.} We do not have an example of a domain $\Omega$ in $2$-dimensional complex manifold such that for some minimal function $\varphi \colon \Omega \to \R$ and some regular value $t$ of $\varphi$ the set $\mathfrak{c}(\Omega) \cap \{\varphi = t\}$ has nonempty interior in $\{\varphi = t\}$. Moreover, we do not know if it can happen that the sets $\gamma_\alpha$ from part $(1)$ of the theorem are only immersed but not embedded submanifolds of $\Omega$.

We discuss briefly some consequences of the results in Theorem \ref{thm_foliatedcore}. Let $\Omega$ be a domain in a $2$-dimensional complex manifold and consider the set
\begin{equation*} \begin{split}
\mathfrak{c}_{\text{reg}}(\Omega) \coloneqq \{z \in  \mathfrak{c}(\Omega) : & \text{ there exists a minimal function } \varphi \colon \Omega \to \R \\ & \text{ such that } \varphi(z) \text{ is a regular value of } \varphi \}.
\end{split} \end{equation*}
It follows from part $(1)$ of Theorem \ref{thm_foliatedcore} that for every $z \in \mathfrak{c}_{\text{reg}}(\Omega)$ there exists a minimal function $\varphi_z \colon \Omega \to \R$ and a weakly embedded complete complex curve $\gamma_z \subset \mathfrak{c}_{\text{reg}}(\Omega) \cap H_z$, where $H_z \coloneqq \{\zeta \in \Omega : \varphi_z(\zeta) = \varphi_z(z)\}$, such that $z \in \gamma_z$. Fix arbitrary $z \in \mathfrak{c}_{\text{reg}}(\Omega)$ and assume that $\gamma_z \cap \gamma_{z'} \neq \varnothing$ for some $z' \in \mathfrak{c}_{\text{reg}}(\Omega)$. Then, by part $(2)$ of Theorem \ref{thm_foliatedcore}, we conclude that $\gamma_z' \subset H_z$. Since for every $p \in H_z$ there exists at most one germ of a complex curve $\gamma \subset H_z$ through $p$, it follows from Step 2 in the proof of part $(1)$ and from maximality of the curves $\gamma_z, \gamma_{z'}$ that $\gamma_z = \gamma_{z'}$. This shows that $\mathfrak{c}_{\text{reg}}(\Omega) = \bigcupdot_{\alpha \in A} \gamma_\alpha$ for a suitable set $A \subset \mathfrak{c}_{\text{reg}}(\Omega)$. In view of part $(2)$ of the theorem, every smooth and bounded from above plurisubharmonic function on $\Omega$ is constant on each curve $\gamma_\alpha$, $\alpha \in A$. Now consider the set
\[ \mathfrak{c}_{\text{sing}}(\Omega) \coloneqq \mathfrak{c}(\Omega) \setminus \mathfrak{c}_{\text{reg}}(\Omega) \]
and let $\mathfrak{c}_{\text{sing}}(\Omega) = \bigcupdot_{\beta \in B} \sigma_\beta$ be the decomposition of $\mathfrak{c}_{\text{sing}}(\Omega)$ into its connected components $\sigma_\beta$, $\beta \in B$. We claim that every smooth and bounded from above plurisubharmonic function $\varphi \colon \Omega \to \R$ is constant on each set $\sigma_\beta$, $\beta \in B$. Indeed, assume, to get a contradiction, that there exist $\varphi$ and $\sigma_\beta$ as above such that $\varphi$ is not constant on $\sigma_\beta$. Then let $\psi$ be a minimal function for $\Omega$ and observe that for $\varepsilon > 0$ small enough $\widetilde{\varphi} \coloneqq \varphi + \varepsilon\psi$ is a minimal function for $\Omega$ which is not constant on $\sigma_\beta$. Hence there exist points $p,q \in \sigma_\beta$ such that $\widetilde{\varphi}(p) < \widetilde{\varphi}(q)$. By connectedness of $\sigma_\beta$, every hypersurface $\{\widetilde{\varphi} = t\}$ for $\widetilde{\varphi}(p) < t < \widetilde{\varphi}(q)$ has nonempty intersection with $\sigma_\beta$, which, in view of Sard's theorem, contradicts the definition of $\mathfrak{c}_{\text{sing}}(\Omega)$. Thus we have shown that in the $2$-dimensional case there exists a decomposition
\[ \mathfrak{c}(\Omega) = \mathfrak{c}_{\text{reg}}(\Omega) \cupdot \mathfrak{c}_{\text{sing}}(\Omega) = \Big( \bigcupdot_{\alpha \in A} \gamma_\alpha \Big) \cupdot \Big( \bigcupdot_{\beta \in B} \sigma_\beta\Big) \]
such that each of the sets $\gamma_\alpha$, $\alpha \in A$, and $\sigma_\beta$, $\beta \in B$, is connected and satisfies a Liouville type theorem for smooth plurisubharmonic functions on $\Omega$.

Observe, however, that the described above decomposition of $\mathfrak{c}(\Omega)$ is not completely satisfactory, since we do not have much information on the sets $\sigma_\beta$ so far. In particular, we do not know if it can happen that some of the sets $\sigma_\beta$ consist of only one point. A bit more information on the set $\mathfrak{c}_{\text{sing}}(\Omega)$ can be obtained as follows: Let
\[ \mathfrak{c}_{\text{sing}}'(\Omega) \coloneqq \overline{\mathfrak{c}_{\text{reg}}(\Omega)} \setminus \mathfrak{c}_{\text{reg}}(\Omega) \quad \text{and} \quad \mathfrak{c}_{\text{sing}}''(\Omega) \coloneqq \mathfrak{c}(\Omega) \setminus \overline{\mathfrak{c}_{\text{reg}}(\Omega)}. \]
Moreover, let $\mathfrak{c}_{\text{sing}}'(\Omega) = \bigcupdot_{\beta \in B'} \sigma'_\beta$ and $\mathfrak{c}_{\text{sing}}''(\Omega) = \bigcupdot_{\beta \in B''} \sigma''_\beta$ be the decompositions into connected components of $\mathfrak{c}_{\text{sing}}'(\Omega)$ and $\mathfrak{c}_{\text{sing}}''(\Omega)$, respectively. Then
\[ \mathfrak{c}_{\text{sing}}(\Omega) = \mathfrak{c}_{\text{sing}}'(\Omega) \cupdot \mathfrak{c}_{\text{sing}}''(\Omega) = \Big(\bigcupdot_{\beta \in B'} \sigma_\beta'\Big) \cupdot \Big(\bigcupdot_{\beta \in B''} \sigma_\beta''\Big).\]
By using the same argument as before, we see that smooth plurisubharmonic functions on $\Omega$ satisfy a Liouville type theorem with respect to each of the sets $\sigma_\beta'$, $\beta \in B'$, and $\sigma_\beta''$, $\beta \in B''$. Moreover, we claim that $\sigma_\beta''$ is $1$-pseudoconcave in $\Omega \setminus \overline{\mathfrak{c}_{\text{reg}}(\Omega)}$ for every $\beta \in B''$ (in particular, each set $\sigma_\beta''$ consists of more than one point). Indeed, fix arbitrary $\beta \in B''$. Let $\varphi \colon \Omega \to \R$ be a minimal function and choose $t \in \R$ such that $\sigma_\beta'' \subset \{\varphi = t\}$. Then, using the same argument as in Theorem 3.6 of \cite{SlodkowskiTomassini04}, one can show that the set $\mathfrak{c}_t \coloneqq \mathfrak{c}(\Omega) \cap \{\varphi =t\}$ has the local maximum property with respect to plurisubharmonic functions, i.e., $\mathfrak{c}_t$ is $1$-pseudoconcave in $\Omega$. Since, by construction, $\sigma_\beta''$ is relatively open in $\mathfrak{c}_t$, and since $1$-pseudoconcavity is a local property, see, for example, Proposition \ref{thm_lmp}, it follows that $\sigma_\beta''$ is $1$-pseudoconcave in $\Omega \setminus \overline{\mathfrak{c}_{\text{reg}}(\Omega)}$, as claimed. However, at the moment we are not able to prove any statement on the set $\mathfrak{c}_{\text{sing}}'(\Omega)$. In particular, the following two questions are open to us: Can it happen that $\mathfrak{c}_{\text{sing}}'(\Omega) \neq \varnothing$? And, if $\mathfrak{c}_{\text{sing}}'(\Omega) \neq \varnothing$ (i.e., there exist $z \in \mathfrak{c}_{\text{sing}}(\Omega)$, a sequence $\{\gamma_j\}_{j=1}^\infty$ of weakly embedded complete complex curves $\gamma_j \subset \mathfrak{c}_{\text{reg}}(\Omega)$ and points $z_j \in \gamma_j$, $j \in \N$, such that $\lim_{j \to \infty} z_j = z)$, is it then possible to prove that the sets $\sigma_\beta' \subset \mathfrak{c}_{\text{sing}}'(\Omega)$, $\beta \in B'$, possess some kind of additional structure (for example, that of an (immersed) complex curve or a $1$-pseudoconcave set)?

We conclude the discussion of Theorem \ref{thm_foliatedcore} by introducing the following two definitions.

\begin{definition} \label{def_maximalcomponent}
Let $\Omega$ be a domain in a complex manifold $\mathcal{M}$. Let $L$ be the family of all subsets $\lambda \subset \mathfrak{c}(\Omega)$ with the property that every smooth and bounded from above plurisubharmonic function on $\Omega$ is constant on $\lambda$, and define a partial ordering on $L$ by setting $\lambda \le \lambda'$ if and only if $\lambda \subset \lambda '$. Then $\lambda \subset \mathfrak{c}(\Omega)$ is called a {\it maximal component} of $\mathfrak{c}(\Omega)$ if $\lambda$ is a maximal element in $L$.
\end{definition}

\noindent One might expect that the curves $\gamma_\alpha$ from Theorem \ref{thm_foliatedcore} or the sets $\sigma_\beta, \sigma_\beta', \sigma_\beta''$ from the above remarks are maximal components of $\mathfrak{c}(\Omega)$ according to the previous definition. However, at the moment we do not know whether this is true or not. We also do not know if it can happen that a maximal component of $\mathfrak{c}(\Omega)$ consists of only one point.

\begin{definition} \label{def_coretype}
Let $\mathcal{M}$ be a complex manifold and let $E \subset \mathcal{M}$ be a closed set. We say that $E$ is of {\it core type} in $\mathcal{M}$ if for every open set $\Omega \subset \mathcal{M}$ such that $E \subset \Omega$ one has $E \subset \mathfrak{c}(\Omega)$.
\end{definition}

\noindent Observe that in all examples of domains $\Omega \subset \mathcal{M}$ such that $\mathfrak{c}(\Omega) \neq \varnothing$, which we have constructed so far, the set $\mathfrak{c}(\Omega)$ was always of core type in $\mathcal{M}$. We do not know whether this holds true in general (see also the remarks at the beginning of Section \ref{sec_core}). In particular, we do not know if the assertions of Theorem \ref{thm_foliatedcore} hold true for functions which are not defined on $\Omega$ but only in a neighbourhood of $\mathfrak{c}(\Omega)$. Moreover, it is not clear if the curves $\gamma_\alpha$ from Theorem \ref{thm_foliatedcore} or the sets $\sigma_\beta, \sigma_\beta', \sigma_\beta''$ from the above remarks are of core type in $\Omega$. Note, however, that in the special case of an irreducible closed subvariety $E \subset \C^n$ of pure dimension, some sufficient conditions for $E$ to be of core type in $\C^n$ are given, for example, by Corollary 1 in \cite{Takegoshi93} and Theorem 4 in \cite{Kaneko96}. Note also that if $E \subset \mathcal{M}$ is of core type in $\mathcal{M}$, then $E$ is $1$-pseudoconcave in $\mathcal{M}$. (Indeed, assume, to get a contradiction, that $E$ is not $1$-pseudoconcave in $\mathcal{M}$. Then, by the same kind of arguments as in Proposition \ref{thm_lmp}, we can find $p \in E$, an open neighbourhood $W \subset\subset \mathcal{M}$ of $p$ and a smooth strictly plurisubharmonic function $\varphi$ defined near $\overline{W}$ such that $\varphi(p) = 0$ and $\varphi < 0$ on $(E \cap \overline{W}) \setminus \{p\}$. Let $C \coloneqq \max_{bW \cap E} \varphi < 0$ if $bW \cap E \neq \varnothing$, and let $C \coloneqq -1$ otherwise. Then for $\delta > 0$ small enough the trivial extension of $\widetilde{\max}_\delta(\varphi, C/2) \colon W \to \R$ to a suitable open neighbourhood $\Omega \subset \mathcal{M}$ of $E$ defines a smooth and bounded from above plurisubharmonic function on $\Omega$ which is strictly plurisubharmonic near $p$. This contradicts the fact that $E$ is of core type in $\mathcal{M}$.)
% Moreover, by using similar arguments as in the proof of Theorem 1 in [Ron], one can show easily that algebraic hypersurfaces in $\C^n$ (i.e., the zero sets of nonconstant holomorphic polynomials) are always of core type in $\C^n$.

We now begin to prove the statements of this section. We start by showing that for every $n \ge 3$ there exists an unbounded strictly pseudoconvex domain $\Omega \subset \C^n$ with smooth boundary and a smooth plurisubharmonic function $\varphi \colon \C^n \to \R$ which is bounded from above on $\Omega$ such that $\mathfrak{c}(\Omega)$ is nonempty and connected but $\varphi$ is not constant on $\mathfrak{c}(\Omega)$. For this we first mention two results on complete pluripolar subsets of $\C^2$. 

\begingroup \leftskip3.5ex 
\textsc{Result 1}. {\it There exist compact connected complete pluripolar subsets $A \subset \C^2$ that consist of more than one point.} 

The first construction of a bounded connected complete pluripolar set $A \subset \C^2$ that consists of more than one point is contained in Example 2.4 in \cite{Sadullaev79}; the fact that this set is complete pluripolar follows from Proposition 2.4 in \cite{LevenbergMartinPoletsky92}. Here $A$ is the graph of a certain holomorphic function $f \in \mathcal{O}(\Delta)$ on the unit disc $\Delta \subset \C$, which is not analytically continuable across any point of $b\Delta$. By slightly improving the construction from \cite{Sadullaev79}, it is possible to choose $A$ as the graph of a function $f \in \mathcal{O}(\Delta) \cap \mathcal{C}^\infty(\overline{\Delta})$, see Example 2.17 and Proposition 2.15 in \cite{LevenbergMartinPoletsky92}, or as the graph of a function $f \in \mathcal{C}^\infty(b\Delta)$, see Theorem 1 in \cite{Edlund04}. In particular, the last two examples show that $A$ can be assumed to be compact. More examples of compact complete pluripolar sets in $\C^2$ can be found, for example, in \cite{Edlund04} and \cite{ElMir84}.
\par \endgroup

\begingroup \leftskip4ex 
\textsc{Result 2}. {\it If $A \subset \C^n$ is a closed complete pluripolar set, then there exists a strictly plurisubharmonic function $\psi \colon \C^n \to [-\infty, \infty)$ such that $A = \{\psi = -\infty\}$ and $\psi$ is smooth on $\C^n \setminus A$.} 

The existence of such a function follows from Corollary 1 in \cite{Coltoiu90}. It is also a consequence of the earlier Proposition II.2 in \cite{ElMir84} and the smoothing procedure of Richberg as formulated, for example, in Theorem I.5.21 from \cite{DemaillyXX}.
\par \endgroup

\begin{example} \label{ex_noliouville}
Let $n \ge 3$. Let $A$ be a compact connected complete pluripolar subset of $\C^2$ that consists of more than one point, and let $L$ be a connected complete pluripolar subset of $\C^{n-2}$ such that every smooth and bounded from above plurisubharmonic function defined near $L$ is constant on $L$. (Possible choices for $A$ are described in Result 1 above. If $n = 3$, then for $L$ we can take, for example, $L = \C$; if $n \ge 3$ is large enough, then for $L$ we can also take, for example, unions of positive-dimensional complex subspaces of $\C^{n-2}$, the Wermer type set $\mathcal{E}$ from Section \ref{sec_noanalyticcore} or any of the sets from Examples \ref{ex_Riemannpolar} and \ref{ex_CPhypersurface}.) Moreover, let $\psi_1 \colon \C^2 \to [-\infty, \infty)$ and $\psi_2 \colon \C^{n-2} \to [-\infty, \infty)$ be strictly plurisubharmonic functions such that $\psi_1$ is smooth on $\C^2 \setminus A$ and $A = \{\psi_1 = -\infty\}$ and such that $\psi_2$ is smooth on $\C^{n-2} \setminus L$ and $L = \{\psi_2 = -\infty\}$. Then the function $\widetilde{\psi}(z,w) \coloneqq \max(\psi_1(z), \psi_2(w))$ is strictly plurisubharmonic and continuous outside $A \times L = \{\widetilde{\psi} = -\infty\}$. Hence, by Richberg, we can smooth it up to get a strictly plurisubharmonic function $\psi \colon \C^2 \times \C^{n-2} \to [-\infty, \infty)$ such that $\psi$ is smooth on $\C^2 \times \C^{n-2} \setminus A \times L$, $\abs{\psi - \widetilde{\psi}} < 1$ on $\C^2 \times \C^{n-2} \setminus A \times L$ and $A \times L = \{\psi = -\infty\}$. Choose a strictly increasing and strictly convex smooth function $\chiup \colon \R \to \R$ such that $\lim_{z \to \infty} \big[\psi_1(z) + \chiup({\norm{z}^2})\big] = \infty$ (if for each $N \in \N$ such that $A \subset B^{\begin{minipage}{1ex} $\scriptstyle \vspace{-0.3ex} 2$ \end{minipage}}(0,N)$ we set $C_N \coloneqq \max\{\abs{\psi_1(z)} : z \in \overline{B^2(0,N+1)} \setminus B^{\begin{minipage}{1ex} $\scriptstyle \vspace{-0.3ex} 2$ \end{minipage}}(0,N)\}$, then every strictly increasing and strictly convex smooth function $\chiup \colon \R \to \R$ satisfying $\chiup(N^2) > C_N + N$ has the required property). Then for generic $C \in \R$, define an unbounded strictly pseudoconvex open set with smooth boundary $\widetilde{\Omega} \subset \C^n$ as
\[ \widetilde{\Omega} \coloneqq \big\{(z,w) \in \C^2 \times \C^{n-2} : \psi(z,w) + \chiup(\norm{z}^2) + \norm{w}^2 < C \big\} \]
and denote by $\Omega$ the connected component of $\widetilde{\Omega}$ that contains the set $A \times L$. Observe that, by the choice of $L$, and by Lemmas \ref{thm_pseudoconcaveconstant} and \ref{thm_liouvillepseudoconcave} above, every smooth and bounded from above plurisubharmonic function on $\Omega$ fails to be strictly plurisubharmonic on $A \times L$. On the other hand, for small enough constants $\eta_j > 0$, $j \in \N$, the function $\Psi(z,w) \coloneqq \sum_{j = 1}^\infty \eta_j\widetilde{\max}_1(\psi(z,w) + \chiup(\norm{z}^2) + \norm{w}^2 - C, -j)$ is a smooth global defining function for $\Omega$ that is strictly plurisubharmonic outside $A \times L$. This shows that $\mathfrak{c}(\Omega) = A \times L$ and, in particular, that $\mathfrak{c}(\Omega)$ is connected. Observe now that, by the choice of the functions $\psi$ and $\chiup$, the domain $\Omega$ is bounded in the $z$-directions. Hence for every $z_0 \in \C^2$ the smooth plurisubharmonic function $\varphi(z,w) \coloneqq \norm{z-z_0}^2$ is bounded from above on $\Omega$. But for almost every choice of $z_0$ it is not constant on $\mathfrak{c}(\Omega)$.
\end{example}

This gives us, for every $n \ge 3$, an example of a strictly pseudoconvex domain with smooth boundary $\Omega \subset \C^n$ and a connected component of $\mathfrak{c}(\Omega)$ (actually here it is the whole of $\mathfrak{c}(\Omega)$) without Liouville type property for plurisubharmonic functions. By slightly changing the above constructions, we can also show that for every $n \ge 4$, we can additionally assume that $\mathfrak{c}(\Omega)$ contains no analytic variety of positive dimension.

\begin{example}
Let $n \ge 4$. Let $A$ be a compact and connected complete pluripolar subset of $\C^2$ that consists of more than one point such that the projections $\pi_{z_1}(A)$ and $\pi_{z_2}(A)$ of $A$ onto the coordinate axes $\C \times \{0\}$ and $\{0\} \times \C$, respectively, have no interior points (for example, take $A$ to be the graph of a suitable smooth function $f \colon b\Delta \to \C$, see \cite{Edlund04}). Moreover, let $L \coloneqq \mathcal{E} \subset \C^{n-2}$ be the Wermer type set as in Theorem 1.1 of \cite{HarzShcherbinaTomassini12}. Then repeat the construction of the previous example to obtain an unbounded strictly pseudoconvex domain with smooth boundary $\Omega \subset \C^n$ such that $\mathfrak{c}(\Omega) = A \times \mathcal{E}$. In particular, $\mathfrak{c}(\Omega)$ is connected and contains no analytic variety of positive dimension. For the last assertion observe that every holomorphic function $f = (f_z, f_w) \colon \Delta \to \C^2_z \times \C^{n-2}_w$ has to have constant $f_z$ component, since, by choice of $A$, the holomorphic images $(\pi_{z_1} \circ f_z)(A), (\pi_{z_2} \circ f_z)(A) \subset \C$ have no interior points, and also constant $f_w$ component, since $\mathcal{E}$ contains no analytic variety of positive dimension, which implies that $f$ is constant. Finally, observe that, as before, for almost every choice of $z_0 \in \C^2$ the function $\varphi \colon \C^2 \times \C^{n-2} \to \R$ defined as $\varphi(z,w) \coloneqq \norm{z-z_0}^2$ is a smooth and bounded from above plurisubharmonic function on $\Omega$ which is not constant on $\mathfrak{c}(\Omega)$.
\end{example}

We now begin to prove the theorems of this section. First we prove the Liouville type property of the highest order core, as formulated in Theorem \ref{thm_highliouville}.

\noindent \textbf{Proof of Theorem \ref{thm_highliouville}.} Let $Z$ be a connected component of $\mathfrak{c}_n(\Omega)$ and let $\varphi \colon \Omega \to \R$ be a smooth plurisubharmonic function which is bounded from above. Assume, to get a contradiction, that there exist points $z_1, z_2 \in Z$ such that $\varphi(z_1) \neq \varphi(z_2)$. Then, by Sard's theorem and by connectedness of $Z$, there exists a regular value $t \in \R$ of $\varphi$ such that $\{\varphi = t\} \cap \mathfrak{c}_n(\Omega) \neq \varnothing$. Choose a strictly increasing and strictly convex smooth function $\chiup \colon \R \to \R$ such that $\chiup \circ \varphi$ is still bounded from above on $\mathcal{M}$. Then $\chiup \circ \varphi$ is a smooth plurisubharmonic function on $\Omega$ that is bounded from above, but $\Lev(\chiup \circ \varphi)(z,\,\cdot\,) = \chiup''(\varphi(z))\abs{(\partial \varphi)_z(\,\cdot\,)}^2 + \chiup'(\varphi(z))\Lev(\varphi)(z,\,\cdot\,) \not\equiv 0$ for every $z \in \{\varphi = t\}$, which contradicts the fact that $\{\varphi = t\} \cap \mathfrak{c}_n(\Omega) \neq \varnothing$. $\hfill \Box$

We now turn to the proof of Theorem \ref{thm_foliatedcore}. For a domain $\Omega \subset \C^n$ denote by $\mathcal{A}(\Omega) \coloneqq \mathcal{O}(\Omega) \cap \mathcal{C}(\overline{\Omega})$ the algebra of functions holomorphic on $\Omega$ and continuous on $\overline{\Omega}$ and write $\mathcal{P}(\Omega) \coloneqq \mathcal{PSH}(\Omega) \cap \mathcal{USC}(\overline{\Omega})$ for the functions plurisubharmonic on $\Omega$ and upper semicontinuous on $\overline{\Omega}$. Recall that a plurisubharmonic function $\varphi \colon \Omega \to [-\infty, \infty)$ is called maximal if for every relatively compact open set $G \subset \Omega$ and for each $\psi \in \mathcal{P}(G)$ such that $\psi \le \varphi$ on $bG$ we have $\psi \le \varphi$ on $G$. If $\Omega$ is a bounded strictly pseudoconvex domain in $\C^n$ and $f \colon b\Omega \to \R$ is a continuous function, then by Theorem 4.1 in \cite{Bremermann59} and Theorem 1 in \cite{Walsh68} there exists a unique continuous function $F \colon \overline{\Omega} \to \R$ such that $F$ is maximal plurisubharmonic on $\Omega$ and $F|_{b\Omega} = f$ (note that both mentioned above theorems are stated for $\mathcal{C}^2$-smooth strictly pseudoconvex domains, but actually no assumption on smoothness of $b\Omega$ is needed in the proof of Theorem 4.1 in \cite{Bremermann59} and, hence, also in the proof of Theorem 1 in \cite{Walsh68}). Moreover,
\begin{equation} \label{equ_MADirichlet}
F(z) = \sup \big\{\varphi(z) : \varphi \in \mathcal{U}(\Omega,f)\big\},
\end{equation}
where $\mathcal{U}(\Omega,f)$ denotes the family of all $\varphi \in \mathcal{P}(\Omega)$ such that $\varphi \le f$ on $b\Omega$. If $f \ge 0$, then one can easily see that $\{F=0\} = \widehat{\{f=0\}}\vphantom{\{}^{\mathcal{P}(\Omega)}$, where $\widehat{\{f=0\}}\vphantom{\{}^{\mathcal{P}(\Omega)} \coloneqq \{z \in \overline{\Omega} : \varphi(z) \le \sup_{w \in \{f=0\}} \varphi(w)\text{ for every } \varphi \in \mathcal{P}(\Omega)\}$. In fact, more is true as it is shown in the following lemma (the statement of the lemma seems to be well known, but since we were not able to find a reference in the literature, we include here its proof for the convenience of reading).

\begin{lemma} \label{thm_MAhull}
Let $\Omega \subset \C^n$ be a bounded strictly pseudoconvex domain (not necessarily with smooth boundary). Let $f \colon b\Omega \to [0, \infty)$ be a continuous function and let $F \colon \overline{\Omega} \to [0,\infty)$ be the maximal plurisubharmonic function on $\Omega$ such that $F|_{b\Omega} = f$. Then $\{F= 0\} = \widehat{\{f=0\}}\vphantom{\{}^{\mathcal{A}(\Omega)}$. If $\,\overline{\Omega}$ is polynomially convex, then $\{F=0\} = \widehat{\{f=0\}}$, where $\widehat{\{f=0\}}$ denotes the polynomially hull of $\{f=0\}$.
\end{lemma}
\begin{proof} Set $K \coloneqq \{f=0\}$. Let first $z_0 \in \overline{\Omega} \setminus \hat{K}^{\mathcal{A}(\Omega)}$, i.e., there exists $h \in \mathcal{A}(\Omega)$ such that $\abs{h(z_0)} > \max_{z \in K} \abs{h(z)}$. Then for suitably chosen constants $C, \varepsilon > 0$ the function $\varphi \coloneqq \varepsilon(\abs{h} - C) \in \mathcal{P}(\Omega)$ satisfies $\varphi|_{b\Omega} \le f$ and $\varphi(z_0) > 0$. By $(\ref{equ_MADirichlet})$ this implies that $F(z_0) > 0$, i.e., $z_0 \in \overline{\Omega} \setminus \{F=0\}$. On the other hand, let now $z_0 \in \overline{\Omega} \setminus \{F=0\}$, i.e., there exists $\varphi \in \mathcal{P}(\Omega)$ such that $\varphi|_{b\Omega} \le f$ but $\varphi(z_0) > 0$. Let $g \colon \C^n \to \R$ be a smooth function such that $g|_{b\Omega} > \varphi|_{b\Omega}$ and $g|_K < \varphi(z_0) < g(z_0)$, and let $\psi \colon U \to \R$ be a strictly plurisubharmonic function on an open neighbourhood $U \subset \C^n$ of $\overline{\Omega}$ such that $\Omega = \{\psi < 0\}$. Moreover, choose $C>0$ so large that, after possibly shrinking $U$, the function $C\psi + g$ is plurisubharmonic on $U$, and, in the case when $z_0 \in \Omega$, one also has $(C\psi + g)(z_0) < \varphi(z_0)$. Then $\phi \coloneqq \max(\varphi, C\psi + g) \colon U \to [-\infty,\infty)$ is a plurisubharmonic function such that $\phi(z_0) > \max_{z \in K} \phi(z)$. Since $\Omega$ has a Stein neighbourhood basis, we can assume that $U$ is pseudoconvex. By the equality of holomorphic and plurisubharmonic convex hulls of compact sets in pseudoconvex domains (see, for example, Theorem 4.3.4 in \cite{Hoermander90}), we then can find a holomorphic function $h \in \mathcal{O}(U) \subset \mathcal{A}(\Omega)$ such that $\abs{h(z_0)} > \max_{z \in K} \abs{h(z)}$, i.e., $z_0 \in \overline{\Omega} \setminus \hat{K}^{\mathcal{A}(\Omega)}$. If $\overline{\Omega}$ is polynomially convex, then, by the Oka-Weil theorem, $h$ can be chosen to be a holomorphic polynomial.
\end{proof}

\noindent \textbf{Proof of Theorem \ref{thm_foliatedcore}.} (1) We start with proving the first part of the theorem. In order to do so, we proceed in three steps.

\noindent \textsc{Step 1.} {\it For every $p \in \mathfrak{c}_t$, there exist local holomorphic coordinates on an open neighbourhood $U \subset \mathcal{M}$ of $p$ and numbers $\varepsilon, \delta, c > 0$ such that $p$ is the origin in $\C^2_{z,w}$ and the following assertions hold true:

\begin{itemize}
  \item[(i)] $U = U' \times (-c, c) \subset \C^2$ for some domain $U' \subset \C_z \times \R_u$, where $w = u+iv$,
  \item[(ii)] there exist a smooth function $\Phi \colon U' \to (-c,c)$ and a continuous function $\Psi \colon U' \to (-c,c)$, $\Psi \le \Phi$, such that $\Gamma_\Phi = U \cap \{\varphi = t\}$, $\Gamma_\Psi$ is a Levi-flat hypersurface and $\mathfrak{c}_t \cap U = \Gamma_\Phi \cap \Gamma_\Psi$, where $\Gamma_\Phi$ and $\Gamma_\Psi$ denote the graphs of $\Phi$ and $\Psi$, respectively,
  \item[(iii)] there exists a continuous one-parameter family $\{f_u\}_{-\delta < u < \delta}$ of holomorphic functions $f_u \colon \Delta_{\varepsilon} \to \C_w$ satisfying $\Re f_u(0) = u$ for every $u \in (-\delta, \delta)$ such that $U' = \bigcupdot_{-\delta < u < \delta} \{(z, \Re f_u(z)) \in \C \times \R : z \in \Delta_{\varepsilon}\} \eqqcolon \bigcupdot_{-\delta < u < \delta}D_u'$ and $\Gamma_\Psi = \bigcupdot_{-\delta < u < \delta} \{(z, f_u(z)) \in \C^2 : z \in \Delta_{\varepsilon}\} \eqqcolon \bigcupdot_{-\delta < u < \delta}D_u$, where $\Delta_{\varepsilon} \coloneqq \{z \in \C : \abs{z} < \varepsilon\}$, 
  \item[(iv)] there exists a subset $d \subset (-\delta, \delta)$ such that $\mathfrak{c}_t \cap U = \bigcupdot_{u \in d} D_u$. 
\end{itemize} }

\noindent \textsc{Proof.} Without loss of generality we can assume that $t = 0$. Moreover, by introducing suitable local coordinates around $p$, we can also assume that $p$ is the origin in $\C^2_{z,w}$ and that $T_p(\{\varphi=0\}) = \C_z \times \R_u$, where $w = u+iv$.  Choose constants $r, R > 0$ such that for the convex domain $G \coloneqq \{\abs{z}^2 + u^2 < r\} \cap \{\abs{z}^2 + \abs{w}^2 < R\} \subset \C^2$ there exists a smooth function 
\[ \Phi \colon \{(z,u) \in \C \times \R : \abs{z}^2 + u^2 \le r\} \to \R_v \quad \text{such that} \quad \Gamma_\Phi = \{\varphi = 0\} \cap \overline{G}, \]
and such that $\interior \Gamma_\Phi \coloneqq \Gamma_\Phi \cap \{(z,w) \in \C^2 : \abs{z}^2 + u^2 < r\}$ satisfies $\interior \Gamma_\Phi = \Gamma_\Phi \cap G$. After smoothing the wedges of $G$, we can assume without loss of generality that $bG$ is smooth. Since $\{\varphi = 0\} \cap bG$ is the graph of a smooth function over $\{(z,u) \in \C \times \R : \abs{z}^2 + u^2 = r\}$, it follows from Theorem 3 in \cite{BedfordKlingenberg91} that there exists a continuous function
\[ \Psi \colon \{(z,u) \in \C \times \R : \abs{z}^2 + u^2 \le r\} \to \R_v \quad \text{such that} \quad \Gamma_\Psi = \widehat{\{\varphi = 0\} \cap bG}, \] 
and $\interior \Gamma_\Psi$ is a Levi-flat hypersurface. Let $f \coloneqq \max(\varphi,0)|_{bG}$ and let $F \colon \overline{G} \to \R$ be the continuous function such that $F$ is maximal plurisubharmonic on $G$ and $F|_{bG} = f$. Finally, set $K \coloneqq \{f=0\} \subset bG$ and observe that $\hat{K} = \{F=0\} \subset \overline{G}$ by Lemma \ref{thm_MAhull}. 

We claim that $\Phi \ge \Psi$ and that $\hat{K} = \{(z,w) \in \overline{G} : v \le \Psi(z,u)\} \eqqcolon \Gamma^-_{\Psi}$. Indeed, by the choice of $f$, we have $\Gamma_\Psi \subset \hat{K}$ and hence $\hat{K} = \widehat{K \cup \Gamma_\Psi} \supset \Gamma^-_{\Psi}$, since $K \cup \Gamma_\Psi = b\Gamma^-_{\Psi}$. For the other direction, note first that, by strict pseudoconvexity of $bG$, for every $a \in \hat{K} \cap bG$, there exists a holomorphic polynomial $P$ such that $\abs{P}$ attains a strict local maximum at $a$ along $\hat{K}$. (Indeed, we can choose for $P$ a finite part of the Taylor expansion of $1/L_G(a,z-\varepsilon N_G(a))$, where $L_G(a, \,\cdot\,)$ denotes the Levi polynomial at $a$ of a strictly plurisubharmonic defining function for $G$, $N_G(a)$ is the outward unit normal vector to $G$ at $a$ and $\varepsilon > 0$ is small enough.) Thus, by Rossi's local maximum modulus principle, it follows that $\hat{K} \cap bG = K$. In particular, $\hat{K} \cap b\Gamma_\Psi^+ = \hat{K} \cap \Gamma_\Psi$, where $\Gamma_\Psi^+ \coloneqq \{(z,w) \in G : v \ge \Psi(z,u)\}$. Another application of Rossi's local maximum modulus principle now shows that $\hat{K} \cap \Gamma_\Psi^+ = \widehat{\hbox{\it{\^{K}}} \cap b\Gamma_\Psi^+}$, and, in view of polynomial convexity of $\Gamma_\Psi$, we get that $\widehat{\hbox{\it{\^{K}}} \cap \Gamma_\Psi} \subset \Gamma_\Psi$. Hence $\hat{K} \cap \Gamma_\Psi^+ \subset \Gamma_\Psi$, i.e., $\hat{K} \subset \Gamma_\Psi^-$. The proof of the second claim is now complete. For the first claim observe that $\varphi \in \mathcal{U}(G,f)$, i.e., $\varphi \le F$ in view of $(\ref{equ_MADirichlet})$ and hence $\{F=0\} \subset \{\varphi \le 0\}$. In particular, $\Gamma_\Psi \subset \{\varphi \le 0\}$ and thus $\Phi \ge \Psi$.

Next we want to show that $\mathfrak{c}(\Omega) \cap \Gamma_\Phi \subset \Gamma_\Psi$. Indeed, in view of the assertions that we have just proven, it suffices to show that $\mathfrak{c}(\Omega) \cap \Gamma_\Phi \subset \hat{K}$. Thus let $q \in \mathfrak{c}(\Omega) \cap \overline{G}$ such that $\varphi(q) = 0$ and assume, to get a contradiction, that $q \notin \hat{K}$ and hence, in view of Lemma \ref{thm_MAhull}, that $F(q) > 0$. Since $\Gamma_\Phi \cap bG \subset K$, it follows that $q \in G$. Then for $\gamma_1 \coloneqq F(q)/2 > 0$ define $\varphi^\ast \colon \Omega \to \R$ as $\varphi^\ast \coloneqq \widetilde{\max}_{\gamma_1}(\varphi,0)$ and observe that $\max(\varphi,0) \le \varphi^\ast \le \max(\varphi,0) + \gamma_1/2$ by definition of the smooth maximum. Hence $\varphi^\ast(q) \le \max(\varphi(q),0) + \gamma_1/2 = \gamma_1/2 < F(q) - \gamma_1$ while on $bG$ we have $\varphi^\ast \ge \max(\varphi,0) = F > F - \gamma_1$. Then $F^\ast \coloneqq \delta_{\gamma_2} \ast (F - \gamma_1) + \gamma_3\norm{\,\cdot\,}^2$, where $\gamma_2$ and $\gamma_3$ are small enough positive constants and $\delta_{\gamma_2}$ is a smooth nonnegative function depending only on $\norm{(z,w)}$ such that $\supp \delta_{\gamma_2} \subset \overline{B^2(0,\gamma_2)}$ and $\int_{\C^2} \delta_{\gamma_2} = 1$, is a smooth strictly plurisubharmonic function on $\overline{G}_{\gamma_2} \coloneqq \{(z,w) \in G : \dist((z,w), bG) \ge \gamma_2\}$ such that $\varphi^\ast(q) < F^\ast(q)$ and $\varphi^\ast > F^\ast$ on $bG_{\gamma_2}$. In particular, for $\delta > 0$ small enough $\widetilde{\max}_{\delta}(\varphi^\ast, F^\ast)$ is a smooth and bounded from above plurisubharmonic function on $\Omega$ that is strictly plurisubharmonic in $q$. This contradicts the fact that $q \in \mathfrak{c}(\Omega)$. 

From the Main Theorem in \cite{Shcherbina93} we know that $\Gamma_\Psi$ is the disjoint union of a family of complex discs. Moreover, it follows from the Main Lemma in \cite{Shcherbina93} that there exist positive constants $\varepsilon, \delta > 0$ and a continuous one-parameter family $\{f_u\}_{-\delta < u < \delta}$ of holomorphic functions $f_u \colon \Delta_{\varepsilon} \to \C_w$ satisfying $\Re f_u(0) = u$ for every $u \in (-\delta, \delta)$ such that $U' \coloneqq \bigcupdot_{-\delta < u < \delta} \{(z, \Re f_u(z)) \in \C \times \R : z \in \Delta_{\varepsilon}\}$ is an open neighbourhood of $p$ in $\C \times \R$ contained in $\{(z,u) \in \C \times \R : \abs{z} ^2 + u^2 < r\}$ and $\Gamma_\Psi \cap (U' \times \R_v) = \bigcupdot_{-\delta < u < \delta} \{(z, f_u(z)) \in \C^2 : z \in \Delta_{\varepsilon}\}$. Chooce $c > 0$ such that $\Psi(U') \subset (-c,c)$.

Now fix some $u \in (-\delta, \delta)$ and assume that $D_u \cap \Gamma_\Phi \neq \varnothing$, i.e, there exists $p \in D_u$ such that $\varphi(p) = 0$. Since $\Gamma_\Psi \subset \hat{K} \subset \{\varphi \le 0\}$, it follows that $\varphi|_{D_u}$ is a subharmonic function that attains a maximum at $p$. Hence $D_u \subset \{\varphi=0\} \cap U = \Gamma_\Phi$. This shows that there exists $d \subset (-\delta, \delta)$ such that $\Gamma_\Phi \cap \Gamma_\Psi = \bigcupdot_{u \in d} D_u$. In particular, for every $p \in \Gamma_\Phi \cap \Gamma_\Psi$ there exists $u \in d$ such that $p \in D_u \subset \Gamma_\Phi \cap \Gamma_\Psi$. But then $\varphi$ is constant on $D_u$ and hence not strictly plurisubharmonic at $p$, which implies $p \in \mathfrak{c}(\Omega)$ by minimality of $\varphi$. This shows that $\Gamma_\Phi \cap \Gamma_\Psi \subset \mathfrak{c}_0 \cap U$. On the other hand, we already know that $\mathfrak{c}_0 \cap U = \mathfrak{c}(\Omega) \cap \Gamma_\Phi \subset \Gamma_\Phi \cap \Gamma_\Psi$. Hence $\Gamma_\Phi \cap \Gamma_\Psi = \mathfrak{c}_0 \cap U$. The proof of Step 1 is now complete.

\noindent \textsc{Step 2.} {\it Let $H \subset \mathcal{M}$ be a smooth real hypersurface in the $2$-dimensional complex manifold $\mathcal{M}$. For every $p \in H$, let $\{\gamma_{p,j}\}_{j \in J_p}$ be the family of all connected complex curves $\gamma_{p,j} \subset H$ such that $p \in \gamma_{p,j}$. Then $\gamma_p \coloneqq \bigcup_{j \in J_p} \gamma_{p,j}$ is a connected complex curve in $H$ and each $\gamma_{p,j}$ is an open complex submanifold of $\gamma_p$. }

\noindent \textsc{Proof.} This statement surely is well known, but for the convenience of reading we sketch its proof (observe that $J_p$ here might be empty).

First we note that for every $p \in H$ there exists at most one germ $\delta_p$ of an embedded $1$-dimensional complex submanifold $\delta \subset H$ of $\mathcal{M}$. Indeed, assume, to get a contradiction, that there exist two submanifolds $\delta_1, \delta_2 \subset H$ such that $\delta_{1,p} \neq \delta_{2,p}$. After possibly shrinking $\delta_1$ and $\delta_2$ we can assume that $\delta_1 \cap \delta_2 = \{p\}$. Choose an open coordinate neighbourhood $U \subset \mathcal{M}$ of $p$ and local holomorphic coordinates on $U$ such that there exist $U' = \Delta_\varepsilon \times (-a,a) \subset \C_z \times \R_u$ and a smooth function $h \colon U' \to \R_v$ satisfying $U \cap H = \Gamma_h$, and, moreover, holomorphic functions $f_1,f_2 \colon \Delta_\varepsilon \to \C_w$ such that $\Gamma_{f_1} = \delta_1 \cap U$ and $\Gamma_{f_2} = \delta_2 \cap U$. It follows then from Rouch\'e's Theorem that for $\delta > 0$ small enough the two functions $g_0, g_\delta \colon \Delta_\varepsilon \to \C$, $g_0 = f_1 - f_2$ and $g_\delta \coloneqq f_1 - (f_2 + i\delta)$ have the same number of zeros, which contradicts the facts that $\Gamma_{f_1} \cap \Gamma_{f_2} \neq \varnothing$ but $\Gamma_{f_1} \cap \Gamma_{f_2 + i\delta} \subset \Gamma_h \cap \Gamma_{h+\delta}  = \varnothing$.

%%%%%%%%%%%%%%%%%%%%%%%%%%%%%%%%%%%%
%%% proof of the above statement %%%
%%%%%%%%%%%%%%%%%%%%%%%%%%%%%%%%%%%%
%(Indeed, assume, to get a contradiction, that there exist two germs given by curves $\gamma_1, \gamma_2 \subset M$. After a local biholomorphic change of variables we can assume that locally, i.e., in $U(B^3_\varepsilon) \times (-R,R)$, we have $p = 0$, there exists a smooth function $g \colon U(B^3_\varepsilon) \to \R_v$ such that $\Gamma_g = M$, holomorphic functions $f_1, f_2 \colon U(B^2_\varepsilon) \subset \C_z$ such that $\Gamma_{f_1} = \gamma_1$, $\Gamma_{f_2} = \gamma_2$ and $\Gamma_{f_1} \cap \Gamma_{f_2} = \{0\}$, and finally $\gamma_2 \subset \{0\} \times \C_w$. Since $\Gamma_{f_1} \cap \Gamma_{f_2} \supset \{0\} \neq \varnothing$, it then follows from Rouch\'e theorem that also $\Gamma_{f_1} \cap \Gamma_{f_2 +i\delta} \neq \varnothing$ for small enough $\delta > 0$ which contradicts the fact that $\Gamma_{f_1} \subset M$ but $\Gamma_{f_2 + i\delta} \subset \C^2 \setminus M$.)

Define $V \subset \gamma_p$ to be open in $\gamma_p$ if $V \cap \gamma_{p,j}$ is open in $\gamma_{p,j}$ for every $j \in J_p$. By the unicity of germs of complex manifolds in $H$ described above we conclude that $\gamma_{p, j_1} \cap \gamma_{p, j_2}$ is open in $\gamma_{p, j_1}$ and $\gamma_{p, j_2}$. Thus the open sets in $\gamma_p$ define a topology on $\gamma_p$ and each $\gamma_{p,j}$ is open in $\gamma_p$. Since each $\gamma_{p,j}$ is locally an embedded submanifold of $\mathcal{M}$, and since each $\gamma_{p,j}$ is open in $\gamma_p$, it follows that there exists a unique complex structure on $\gamma_p$ such that the inclusion $\gamma_p \hookrightarrow \mathcal{M}$ is a holomorphic immersion (and with respect to this complex structure the inclusions $\gamma_{p,j} \hookrightarrow \gamma_p$ are holomorphic for every $j \in J_p$). From the continuity of the inclusion we conclude that the topology on $\gamma_p$ is Hausdorff. Then Rad\'o's Theorem on second countability of Riemann surfaces shows that the topology of $\gamma_p$ has a countable basis.

\noindent \textsc{Step 3.} {\it The set $\mathfrak{c}_t$ is the disjoint union of a family of weakly embedded complete connected complex curves. Moreover, the sets $\mathring{\mathfrak{c}}_t$ and $b\mathfrak{c}_t$ have the structure as described above.}

\noindent \textsc{Proof.} It follows immediately from Step 1 that for every $p \in \mathfrak{c}_t$ there exists a complex disc $\delta_p \subset \mathcal{M}$ such that $p \in \delta_p \subset \mathfrak{c}_t$. Applying Step 2 in the case $H \coloneqq \{\varphi = t\}$, we conclude that each complex disc $\delta_p$ extends to a maximal connected complex curve $\gamma_p$ in $H$, and that there exists $\mathcal{A} \subset \mathfrak{c}_t$ such that $\mathfrak{c}_t = \bigcupdot_{\alpha \in \mathcal{A}} \gamma_\alpha$. 

We claim that each $\gamma_\alpha$ is weakly embedded in $\mathcal{M}$. Indeed, let $N$ be a smooth manifold and let $f \colon N \to \mathcal{M}$ be a smooth map such that $f(N) \subset \gamma_\alpha$, where $\alpha \in \mathcal{A}$ is fixed. Choose arbitrary $q \in N$ and local holomorphic coordinates on $U \subset \mathcal{M}$ around $f(q) \in \gamma_\alpha \cap U \subset \mathfrak{c}_t \cap U = \bigcupdot_{u \in d} D_u$ as described in Step 1. Since $\gamma_\alpha \cap U$ is open in $\gamma_\alpha$, it has at most countably many connected components $\{\gamma_\alpha^k\}_{k \in K_\alpha}$. Moreover, by the unicity of germs of complex curves in $H$ (see the proof of Step 2), and by the identity theorem applied to the function $g_u \coloneqq w - f_u(z)$, we see that $\gamma_\alpha^k \subset D_u$ whenever $\gamma_\alpha^k \cap D_u \neq \varnothing$. In fact, we even get that $\gamma_\alpha^k = D_u$, since $\gamma_\alpha$ is maximal. This shows that there exists an at most countable set $d_\alpha \subset d$ such that $\gamma_\alpha \cap U = \bigcupdot_{u \in d_\alpha} D_u$. Now let $W \subset N$ be a connected neighbourhood of $q$ such that $f(W) \subset U$. Observe that the function $e \colon \bigcupdot_{u \in d} D_u \to \R$ defined as $e(z,w) = u$ if and only if $f_u(z) = w$ is continuous. Hence $e \circ f \colon W \to \R$ is a continuous function that takes at most countably many values. By connectedness of $W$, we conclude that $f(W) \subset D_{u_W}$ for some $u_W \in d_\alpha$. Since $D_{u_W}$ is an embedded submanifold of $\mathcal{M}$, it follows that $f \colon W \to D_{u_W}$ is smooth. Moreover, $D_{u_W}$ is open in $\gamma_\alpha$, hence the inclusion $D_{u_W} \hookrightarrow \gamma_\alpha$ is smooth too. This shows that $f \colon W \to \gamma_\alpha$ is a smooth map, and thus that $\gamma_\alpha$ is weakly embedded. % Since $q \in W \subset \mathcal{N}$ was arbitrary, it follows that $f \colon \mathcal{N} \to \gamma_\alpha$ is holomorphic, and hence $\gamma_\alpha$ is weakly embedded in $\mathcal{M}$.

Next we want to show that each $\gamma_\alpha \subset \mathcal{M}$ is complete. Indeed, fix $\alpha \in \mathcal{A}$ and let $g$ be a complete Riemannian metric on $\mathcal{M}$. Let $\{p_j\}_{j=1}^\infty \subset \gamma_\alpha$ be a Cauchy sequence with respect to $i^\ast g$ and let $p \coloneqq \lim_{j\to\infty} p_j \in \mathcal{M}$. Since $\gamma_\alpha \subset \mathfrak{c}_t$, and since $\mathfrak{c}_t$ is closed in $\mathcal{M}$, it follows that $p \in \mathfrak{c}_t$. Choose local holomorphic coordinates on $U \subset \mathcal{M}$ around $p \in \mathfrak{c}_t \cap U = \bigcupdot_{u \in d} D_u$ as described in Step 1. Then $p \in D_{u_p}$ for some $u_p \in d$. Observe that, after possibly shrinking $U$, there exists a constant $C > 0$ such that for every $u_1, u_2 \in d$, $u_1 \neq u_2$, and every $q_1 \in \gamma_\alpha \cap D_{u_1}$, $q_2 \in \gamma_\alpha \cap D_{u_2}$ one has $\dist_{i^\ast g}(q_1, q_2) > C$, where $\dist_{i^\ast g}$ denotes the metric on $\gamma_\alpha$ induced by $i^\ast g$. Since $\{p_j\}_{j=1}^\infty$ is a Cauchy sequence with respect to $\dist_{i^\ast g}$, it follows that there exists $j_0 \in \N$ such that $p_j \in D_{u_p}$ for every $j \ge j_0$. Hence $D_{u_p} \cap \gamma_\alpha \neq \varnothing$. By Step 2 and by maximality of the set $\gamma_\alpha$, we conclude that $D_{u_p} \subset \gamma_\alpha$, i.e., $p \in \gamma_\alpha$.

Finally, observe that from property $(iv)$ of the local holomorphic coordinates in Step 1 it follows immediately that $\mathcal{A}$ has a decomposition $\mathcal{A} = \mathcal{A}' \cupdot \mathcal{A}''$ such that $\mathring{\mathfrak{c}}_t = \bigcupdot_{\alpha \in \mathcal{A}'} \gamma_\alpha$ and $b\mathfrak{c}_t = \bigcupdot_{\alpha \in \mathcal{A}''} \gamma_\alpha$. This concludes the proof of Step 3 and hence also of part $(1)$ of the theorem.

(2) We now prove the second part of the theorem. Assume, to get a contradiction, that there exists a function $\widetilde{\varphi}$ as above that is not constant on $\gamma$. After possibly replacing $\widetilde{\varphi}$ by $\widetilde{\varphi} + \varepsilon\psi$, where $\psi \colon \Omega \to \R$ is a minimal function for $\Omega$ and $\varepsilon > 0$ is small enough, we can assume without loss of generality that $\widetilde{\varphi}$ is minimal. Applying Sard's theorem to the functions $\widetilde{\varphi}$ and $\widetilde{\varphi}|_\gamma$ simultaneously, we see that we can choose a regular value $\widetilde{t} \in \R$ for $\widetilde{\varphi}$ such that $\gamma$ and $\{\widetilde{\varphi} = \widetilde{t}\}$ intersect transversally. Let $p \in \gamma \cap \{\widetilde{\varphi} = \widetilde{t}\}$. From part $(1)$ we know that there exists a complex curve $\widetilde{\gamma} \subset \{\widetilde{\varphi} = \widetilde{t}\}$ such that $p \in \widetilde{\gamma}$. Observe that, by transversality of $\gamma$ and $\{\widetilde{\varphi} = \widetilde{t}\}$ at $p$, we have $T_p\gamma \cap T_p\widetilde{\gamma} = \{0\}$. Now let $\chiup \colon \R \to \R$ be a smooth strictly increasing and strictly convex function such that the smooth plurisubharmonic function $\Phi \coloneqq \chiup \circ \varphi + \chiup \circ \widetilde{\varphi}$ is still bounded from above on $\Omega$. We claim that $\Phi$ is strictly plurisubharmonic in $p$ which contradicts the fact that $p \in \mathfrak{c}(\Omega)$. Indeed, observe that $\Lev(\chiup \circ \varphi)(p,\xi) = \chiup''(\varphi(p))\abs{(\partial\varphi)_p(\xi)}^2 + \chiup'(\varphi(p))\Lev(\varphi)(p,\xi) > 0$ for every $\xi \in T_p\mathcal{M} \setminus \Ker[(\partial\varphi)_p] = T_p\mathcal{M} \setminus T_p\gamma$. In the same way we conclude that $\Lev(\chiup \circ \widetilde{\varphi})(p,\xi) > 0$ for every $\xi \in T_p\mathcal{M} \setminus T_p\widetilde{\gamma}$. Since $T_p\gamma \cap T_p\widetilde{\gamma} = \{0\}$, this proves our claim. $\hfill \Box$

\section{Some geometric properties of Wermer type sets} \label{sec_geometricproperties}

The following section consists of two parts: In the first part we prove some geometric properties of the Wermer type set $\mathcal{E} = \{\varphi = -\infty\}$ (see Section \ref{sec_noanalyticcore} for the definitions of $\mathcal{E}$ and $\varphi$). Namely, we prove that $\mathcal{E}$ is $(M,1/2)$-H\"older continuous as a set-valued map for some constant $M > 0$, and we show that $\mathcal{E}$ is connected. In the second part of this section we use H\"older continuity of $\mathcal{E}$ to give an explicit form for the smoothing of the function $\phi \coloneqq \varphi + \norm{\,\cdot\,}^2$ along $\mathcal{E}$.

We begin with the first part: Recall that a map $f: (X_1, d_1) \to (X_2, d_2)$ between metric spaces is called $(M, \alpha)$-H\"older continuous if $d_2(f(x), f(y)) \le M d_1(x,y)^\alpha$ for every $x,y \in X_1$. Here $M, \alpha > 0$ are positive constants. Moreover, observe that the Wermer type set $\mathcal{E}$ defines a map $\underline{\mathcal{E}}$ from the metric space $\C^{n-1}$ of all $(n-1)$-tupels of complex numbers with the standard euclidean metric $d_{\norm{\cdot}}$ to the metric space $\mathcal{F}(\C)$ of all nonempty compact subsets of $\C$ with the Hausdorff metric $d_H$, namely $\underline{\mathcal{E}} \colon (\C^{n-1},d_{\norm{\cdot}}) \to (\mathcal{F}(\C), d_H)$, $\underline{\mathcal{E}}(z) \coloneqq \mathcal{E}_z \coloneqq \{w \in \C : (z,w) \in \mathcal{E}\}$.
\begin{lemma} \label{thm_haushölder}
There exists a constant $M>0$ such that the map $\underline{\mathcal{E}}$ is $(M, 1/2)$-H\"older continuous.
\end{lemma}
\begin{proof} We have to show that there exists $M > 0$ such that
\begin{equation} \label{equ_haushölder}
d_H\big(\mathcal{E}_{z_1}, \mathcal{E}_{z_2}\big) \le M \sqrt{\norm{z_1 - z_2}} \quad \text{for all } z_1, z_2 \in \C^{n-1}.
\end{equation}
To prove $(\ref{equ_haushölder})$, consider the set-valued functions $e_l(z) \coloneqq \varepsilon_l \sqrt{z_{[l]}-a_l}$, $l \in \N$, $e_l \colon \C^{n-1} \to \mathcal{F}(\C)$. Observe that $\underline{\mathcal{E}} = \sum_{l=1}^\infty e_l$, by definition of $\mathcal{E}$, where the sum of the functions $e_l$ is taken pointwise and the sum of two elements $K_1,K_2 \in \mathcal{F}(\C)$ is defined as $K_1 + K_2 \coloneqq \{w \in \C : w = k_1 + k_2 \text{ for some } k_1 \in K_1, k_2 \in K_2 \}$. For each $l \in \N$, choose $e_l^\ast \colon \C^{n-1} \to \C$ such that $e_l(z) = \{e_l^\ast(z), - e_l^\ast(z)\}$ for every $z \in \C^{n-1}$. Then for every $z_1, z_2 \in \C^{n-1}$ we have
\begin{equation*} \begin{split} \varepsilon_l \sqrt{\abs{z_1 - z_2}} &\ge \varepsilon_l \sqrt{\abs{z_{1,[l]} - z_{2,[l]}}} = \sqrt{\gabs{\varepsilon_l^2(z_{1,[l]} - a_l) - \varepsilon_l^2(z_{2,[l]} - a_l)}} \\ 
& = \sqrt{\gabs{\big(\pm e_l^\ast(z_1) - e_l^\ast(z_2)\big) \big(\pm e_l^\ast(z_1) + e_l^\ast(z_2) \big)}} \\ 
& \ge \sup_{\zeta_1 \in e_l(z_1)} \inf_{\zeta_2 \in e_l(z_2)} \abs{\zeta_1 - \zeta_2}
\end{split} \end{equation*}
and similarly
\begin{equation*} \begin{split} \varepsilon_l \sqrt{\abs{z_1 - z_2}} &\ge \varepsilon_l \sqrt{\abs{z_{1,[l]} - z_{2,[l]}}} = \sqrt{\gabs{\varepsilon_l^2(z_{1,[l]} - a_l) - \varepsilon_l^2(z_{2,[l]} - a_l)}} \\ 
& = \sqrt{\gabs{\big(e_l^\ast(z_1) \pm e_l^\ast(z_2)\big) \big(-e_l^\ast(z_1) \pm e_l^\ast(z_2) \big)}} \\ 
& \ge \sup_{\zeta_2 \in e_l(z_2)} \inf_{\zeta_1 \in e_l(z_1)} \abs{\zeta_1 - \zeta_2}.
\end{split} \end{equation*}
This shows that $d_H\big(e_l(z_1), e_l(z_2)\big) \le \varepsilon_l \sqrt{\abs{z_1-z_2}}$, i.e., $e_l$ is $(\varepsilon_l, 1/2)$-H\"older continuous. Observe now that for any two functions $f,g \colon \C^{n-1} \to F(\C)$ we have $d_H(f(z_1) + g(z_1), f(z_2) + g(z_2)) \le d_H(f(z_1), f(z_2)) + d_H(g(z_1), g(z_2))$, hence if $f$ is $(M_1,1/2)$-H\"older continuous and $g$ is $(M_2,1/2)$-H\"older continuous, then $f+g$ is $(M_1 + M_2,1/2)$-H\"older continuous. Applying this to the sequence $\{e_l\}$, we conclude that
\[ d_H\Big(\sum_{l=1}^\nu e_l(z_1), \sum_{l=1}^\nu e_l(z_2)\Big) \le \sum_{l=1}^\nu \varepsilon_l \sqrt{\abs{z_1-z_2}}  \]
for every $\nu \in \N$, and for $\nu \to \infty$ this yields $(\ref{equ_haushölder})$ with $M \coloneqq \sum_{l=1}^\infty \varepsilon_l$.
\end{proof}
\begin{lemma} \label{thm_connectedE}
If $\{\varepsilon_l\}$ is decreasing fast enough, then $\mathcal{E}$ is connected.
\end{lemma}
\begin{proof}
Let $\{\varepsilon_l\}$ be decreasing so fast that $\varepsilon_l\sqrt{\abs{z_{[l]}-a_l}} < 1/2^l$ on $B^{n-1}(0,l)$ for every $l \in \N$. Assume, to get a contradiction, that there exist two open sets $U_1, U_2 \subset \C^n$ such that $\mathcal{E} \cap U_1 \neq \varnothing$, $\mathcal{E} \cap U_2 \neq \varnothing$, $\mathcal{E} \subset U_1 \cup U_2$ and $U_1 \cap U_2 = \varnothing$. Then we conclude from continuity of $\underline{\mathcal{E}}$, see Lemma \ref{thm_haushölder} above, that $\pi_z(\mathcal{E} \cap U_1)$ and $\pi_z(\mathcal{E} \cap U_2)$ are open in $\C^{n-1}$, where $\pi_z \colon \C^n \to \C^{n-1}_z$ denotes the canonical projection. Since $\pi_z(\mathcal{E} \cap U_1) \cup \pi_z(\mathcal{E} \cap U_2) = \pi_z(\mathcal{E}) = \C^{n-1}$, it follows that $D \coloneqq \pi_z(\mathcal{E} \cap U_1) \cap \pi_z(\mathcal{E} \cap U_2)$ is open and nonempty. Thus we can choose $z_0 \in D$ such that $z_{0,p} \notin \{a_l\}_{l=1}^\infty$ for every $p \in \N_{n-1}$ and $\arg(a_l - z_{0,[l]}) \neq \arg(a_{l'} - z_{0,[l']})$ for every $l,l' \in \N$, $[l] = [l']$, $l \neq l'$. Set $U_j(z_0) \coloneqq \{w \in \C : (z_0,w) \in U_j\}$, $j =1,2$, and choose $\delta > 0$ so small that $\dist(\mathcal{E}_{z_0}, b(U_1(z_0) \cup U_2(z_0))) > \delta$. Fix $\nu_0 \in \N$ such that $\sum_{l=\nu_0+1}^\infty \varepsilon_l\sqrt{\abs{z_{0,[l]}-a_l}} < \delta/2$. Then $E_{\nu,z_0} \cap U_1(z_0) \neq \varnothing$, $E_{\nu,z_0} \cap U_2(z_0) \neq \varnothing$ and $E_{\nu,z_0} \subset U_1(z_0) \cup U_2(z_0)$, where $E_{\nu,z_0} \coloneqq \{w \in \C : (z_0,w) \in E_\nu\}$. For every $l \in \N$, let $\sigma_l \coloneqq \{z \in \C^{n-1} : \arg(z_{[l]}-z_{0,[l]}) = \arg(a_l-z_{0,[l]}), \abs{z_{[l]}-z_{0,[l]}} > \abs{a_l-z_{0,[l]}}\}$, and let $h_l \colon \C^{n-1} \setminus \sigma_l \to \C$ be a continuous branch of $\varepsilon_l\sqrt{z_{[l]}-a_l}$. Fix $p_1 = (z_0, w_1) \in E_{\nu_0} \cap U_1$ and $p_2 = (z_0, w_2) \in E_{\nu_0} \cap U_2$. Then there exist functions $\tau_1,\tau_2 \colon \N_{\nu_0} \to \{0,1\}$ such that $w_j = \sum_{l=1}^{\nu_0} (-1)^{\tau_j(l)}h_l(z_0)$, $j=1,2$. Set $\hat{p}_j \coloneqq (z_0, w_j + \sum_{l=\nu_0+1}^\infty h_l(z_0))$ and observe that, by the choice of $\delta$ and $\nu_0$, one has $\hat{p}_j \in \mathcal{E} \cap U_j$, $j =1,2$. Now define a continuous curve $\gamma_z \colon [0,\nu_0] \to \C^{n-1}_z \setminus \bigcup_{l=1}^\infty \sigma_l$ as
\[ \gamma_z(t) \coloneqq \left\{ \begin{array}{l@{\,}l} \big(z_{0,1}, \ldots, z_{0,[\nu]-1}, &z_{0,[\nu]} + 2(t-\nu+1)(a_\nu-z_{0,[\nu]}), \\ & z_{0,[\nu]+1}, \ldots, z_{n-1} \big) \,,\quad t \in [\nu-1, \nu-1/2] \\ \vphantom{0.25ex} \\ \big(z_{0,1}, \ldots, z_{0,[\nu]-1}, &a_\nu + 2(t-\nu+1/2)(z_{0,[\nu]}-a_\nu), \\ & z_{0,[\nu]+1}, \ldots, z_{n-1} \big) \,,\quad t \in [\nu-1/2, \nu] \end{array} \right. \quad (\nu \in \N_{\nu_0}).\] 
and let $\gamma \colon [0,\nu_0] \to \mathcal{E}$ be given as
\[ \gamma(t) \coloneqq \left\{ \begin{array}{l@{\,}l} \big(\gamma_z(t), \sum_{l=1}^{\nu-1} &(-1)^{\tau_2(l)} h_l(\gamma_z(t)) + \sum_{l=\nu}^{\nu_0} (-1)^{\tau_1(l)} h_l(\gamma_z(t)) \\ & + \sum_{l=\nu_0+1}^\infty h_l(\gamma_z(t))\big) \,,\quad t \in [\nu-1, \nu-1/2] \\ \vphantom{0.25ex} \\ \big(\gamma_z(t), \sum_{l=1}^{\nu} &(-1)^{\tau_2(l)} h_l(\gamma_z(t)) + \sum_{l=\nu+1}^{\nu_0} (-1)^{\tau_1(l)} h_l(\gamma_z(t)) \\ & + \sum_{l=\nu_0+1}^\infty h_l(\gamma_z(t))\big) \,,\quad t \in [\nu-1/2, \nu] \end{array} \right. \quad (\nu \in \N_{\nu_0}).\] 
Then it is easy to see that $\gamma$ is a continuous curve in $\mathcal{E}$ such that $\gamma(0) = \hat{p}_1 \in U_1$ and $\gamma(1) = \hat{p}_2 \in U_2$. This is a contradiction.
\end{proof}

Now we turn to the second part of this section. Note that in the examples of domains with nonempty core, which we have constructed so far, the following smoothing procedure has been used several times: Let $\mathcal{M}$ be a complex manifold, let $\Omega \subset \mathcal{M}$ be an open set and let $\phi \colon \Omega \to [-\infty,\infty)$ be a plurisubharmonic function such that $\phi$ is smooth and strictly plurisubharmonic outside $Z \coloneqq \{\phi = -\infty\}$. Then there exists a sequence $\{\eta_j\}_{j=1}^\infty$ of positive numbers converging to zero such that $\Phi \coloneqq \sum_{j=1}^\infty \eta_j\widetilde{\max}_1(\phi,-j)$ is a smooth plurisubharmonic function on $\Omega$ that is strictly plurisubharmonic outside $Z$. Observe that if $\{\eta_j\}_{j=1}^\infty$ is converging to zero fast enough, then the smoothing $\Phi$ can be represented in the form $\Phi = \Lambda \circ \phi$ for some smooth strictly increasing convex function $\Lambda \colon [-\infty,\infty) \to [-1,\infty)$. Indeed, $\Lambda = \sum_{j=1}^\infty \eta_j\lambda_j$, where for every $j\in\N$ the function $\lambda_j$ is given by $\lambda_j(t) \coloneqq (1/2)(t-j+\chiup_1(t+j))$, see page \pageref{def_smoothmax} for the definition of the smooth maximum and the function $\chiup_1$.

We want to give a precise form of this smoothing procedure in the case where $Z$ is the Wermer type set $\mathcal{E}$. Namely, let $\varphi$ be the plurisubharmonic function defined in Section \ref{sec_noanalyticcore} such that $\mathcal{E} = \{\varphi = -\infty\}$, and let $\phi \coloneqq \varphi + \norm{\,\cdot\,}^2$. We then want to make an explicit choice of a function $\Lambda \colon [-\infty,\infty) \to [0,\infty)$ such that $\Phi \coloneqq \Lambda \circ \phi$ is smooth and plurisubharmonic on $\Omega$ and strictly plurisubharmonic outside $\mathcal{E}$. To do so, consider first the function $e^\phi \colon \C^n \to [0,\infty)$. Observe that $e^\phi$ is a continuous plurisubharmonic function on $\C^n$ that is smooth and strictly plurisubharmonic in the complement of $\mathcal{E}$. Thus this function has all the properties we seek except, possibly, for smoothness in points of $\mathcal{E}$. Now the general idea to obtain $\Lambda$ and $\Phi$ as desired is to compose $e^\phi$ with a smooth, strictly increasing and strictly convex function that vanishes at $0$ of infinite order. In fact, we will take $\Lambda \colon [-\infty, \infty) \to [0, \infty)$ such that $\Lambda(x) = e^{-1/e^x}$ for small values of $x$. To actually prove smoothness of $\Phi$, we proceed as follows: We show that for each point $(z,w) \in \C^n \setminus \mathcal{E}$ there exists a polycylinder around $(z,w)$ that does not intersect $\mathcal{E}$, the size of which depends uniformly on the vertical distance $d(w,\mathcal{E}_z) \coloneqq \inf_{w' \in \mathcal{E}_z} \abs{w-w'}$ of $(z,w)$ to $\mathcal{E}$. Moreover, we estimate the value of $e^\varphi$ by means of the vertical distance $d(w, \mathcal{E}_z)$ to the set $\mathcal{E}$. We then use the Poisson integral formula and pluriharmonicity of $\varphi$ outside $\mathcal{E}$ to derive Cauchy type estimates for the derivatives of $\phi$, and apply the above results to conclude that each $D^\alpha \Phi(z,w)$ tends to zero when $(z,w)$ approaches $\mathcal{E}$.

We first prove the existence of uniformly large polycylinders in the complement of $\mathcal{E}$, which follows easily from H\"older continuity of the map $\underline{\mathcal{E}}$.

\begin{lemma} \label{thm_boxE} There exists a constant $C > 0$ such that
\begin{equation} \label{equ_box} \overline{\Delta^n_{\mathcal{E}}}(z,w) \coloneqq \overline{\Delta}^{\begin{minipage}{1ex} $\scriptstyle \vspace{-1.2ex} \! {n-1}$ \end{minipage}}\;\;\;\Big(z, C \Big(\frac{d(w,\mathcal{E}_z)}{2}\Big)^2\Big) \times \overline{\Delta}^{\begin{minipage}{1ex} $\scriptstyle \vspace{-1.2ex} \! 1$ \end{minipage}}\Big(w, \frac{d(w, \mathcal{E}_z)}{2}\Big) \subset \C^n \setminus \mathcal{E} \end{equation}
for every $(z,w) \in \C^n \setminus \mathcal{E}$.
\end{lemma}
\begin{proof}
Fix $(z,w) \in \C^n \setminus \mathcal{E}$. Then, in view of Lemma \ref{thm_haushölder}, for any given $(z',w') \in \mathcal{E} \cap \big[ \overline{\Delta}^{\begin{minipage}{1ex} $\scriptstyle \vspace{-1.2ex} {n-1}$ \end{minipage}}\hspace{2ex}\big(z, (1/(\sqrt{n}M^2))(d(w, \mathcal{E}_z)/2)^2 \big) \times \C_w \big]$ we can find $(z, \tilde{w}) \in \mathcal{E}_z$ such that \linebreak $\norm{(z, \tilde{w}) - (z',w')} \le M\sqrt{\norm{z-z'}} < d(w, \mathcal{E}_z)/2$. Hence $\abs{w-w'} \ge \abs{w-\tilde{w}} - \abs{\tilde{w}-w'} > d(w, \mathcal{E}_z) - d(w,\mathcal{E}_z)/2 > d(w,\mathcal{E}_z)/2$, which proves $(\ref{equ_box})$ for $C \coloneqq 1/(\sqrt{n}M^2)$.
\end{proof}
We now want to estimate the growth of $e^{\varphi(z,w)}$ in terms of the vertical distance $d(w, \mathcal{E}_z)$ to $\mathcal{E}$. For every $\nu, \mu \in \N$, $j \in \N_{2^\nu}$ and $z \in \C^{n-1}$, we denote the $2^\mu$ values of $w_j^{(\nu)}(z) + \sum_{l = \nu + 1}^{\nu + \mu} \varepsilon_l \sqrt{z_{[l]} - a_l}$ by $w_1^{(\mu)}(\nu, j; z), \ldots, w_{2^\mu}^{(\mu)}(\nu, j; z)$. Moreover, for every set $K \subset \C^n$ and every positive number $\delta > 0$ we let $K^{(\delta)} \coloneqq \bigcup_{\zeta \in K} B^n(\zeta, \delta)$.
\begin{lemma} \label{thm_Ml}
If $\{\varepsilon_l\}$ is decreasing fast enough, then there exists an increasing sequence $\{L_N\}_{N=1}^\infty$ of positive numbers such that for every $N \in \N$, $\nu \ge N$ and $j \in \N_{2^\nu}$ 
\begin{equation} \label{equ_Ml}
\abs{P_{\nu+\mu}(z,w)}^{1/2^{\nu+\mu}} \le L_N \cdot \prod_{k=1}^{2^\mu} \abs{w - w_k^{(\mu)}(j,\nu;z)}^{1/2^{\nu+\mu}} \text{ on } \Delta_N \cap \mathcal{E}^{(1)}
\end{equation}
for every $\mu \in \N$, where $\Delta_N \coloneqq B^{n-1}(0,N) \times \C$.
\end{lemma}
\begin{proof}
Fix $N \in \N$, $\nu \ge N$ and $j \in \N_{2^\nu}$. Let $\{\varepsilon_l\}$ be decreasing so fast that $\varepsilon_l \sqrt{\abs{z_{[l]}-a_l}} < 1/2^l$ on $B^{n-1}(0,l)$ for every $l \in \N$. Then $\abs{w_{j'}^{(\nu)}(z) - w_k^{(\mu)}(j',\nu;z)} \le \sum_{l=\nu+1}^{\nu+\mu} \varepsilon_l\sqrt{\abs{z_{[l]}-a_l}} < 1/2^\nu$ on $B^{n-1}(0,N)$ for every $j' \in \N_{2^\nu}$ and $\mu \in \N$, $k \in \N_{2^\mu}$. Hence $d_H(E_{\nu,z}, \mathcal{E}_z) \le 1/2^\nu$ for every $z \in B^{n-1}(0,N)$, where $E_{\nu,z} \coloneqq \{w \in \C : (z,w) \in E_\nu\}$, and
\begin{equation} \label{equ_Ml1} \prod_{k=1}^{2^\mu} \abs{w - w_k^{(\mu)}(j',\nu;z)}^{1/2^\mu} < \abs{w - w_{j'}^{(\nu)}(z)} + 1/2^\nu \end{equation}
for every $(z,w) \in \Delta_N$, $j' \in \N_{2^\nu}$ and $\mu \in \N$. Further, let $\{\varepsilon_l\}$ be decreasing so fast that $\sum_{l=1}^\infty \varepsilon_l \sqrt{\abs{a_l}} < 1/2$ and $\varepsilon_l < 1/2^{l+1}$ for every $l \in \N$. Since $\gabs{\sqrt{\abs{z_{[l]} - a_l}} - \sqrt{\abs{z_{[l]}}}} \le \sqrt{\gabs{\abs{z_{[l]} - a_l} - \abs{z_{[l]}}}} \le \sqrt{\abs{a_l}}$, we have $\sqrt{\abs{z_{[l]} - a_l}} \le \sqrt{\abs{z_{[l]}}} + \sqrt{\abs{a_l}} \le \sqrt{\norm{z}} + \sqrt{\abs{a_l}}$. Thus, by definition of the sets $E_\nu$, it follows that $\abs{w} \le \sum_{l=1}^\nu \varepsilon_l\sqrt{\abs{z_{[l]} - a_l}} \le \sum_{l=1}^\nu \varepsilon_l(\sqrt{\norm{z}} + \sqrt{\abs{a_l}}) < (1/2)(\sqrt{\norm{z}} + 1)$ for every $(z,w) \in E_\nu$. Taking the limit $\nu \to \infty$ we get that $\mathcal{E} \subset \{(z,w) \in \C^n : \abs{w} \le (1/2)(\sqrt{\norm{z}} + 1)\}$. For fixed $(z,w) \in \Delta_N \cap \mathcal{E}^{(1)}$ and $j' \in \N_{2^\nu}$ choose $\tilde{w}_1, \tilde{w}_2 \in \mathcal{E}_z$ such that $\abs{w-\tilde{w}_1} = d(w, \mathcal{E}_z)$ and $\abs{\tilde{w}_2-w_{j'}^{(\nu)}(z)} = d_H(E_{\nu,z}, \mathcal{E}_z)$. Then $\abs{w - w_{j'}^{(\nu)}(z)} \le \abs{w - \tilde{w}_1} + \abs{\tilde{w}_1 - \tilde{w_2}} + \linebreak \abs{\tilde{w}_2 - w_{j'}^{(\nu)}(z)} \le 1 + (1+\sqrt{N}) + 1/2^N$, hence there exists a constant $L_N > 1$ such that 
\begin{equation} \label{equ_Ml2} \abs{w - w_{j'}^{(\nu)}(z)} + 1/2^\nu \le L_N \;\text{ on }\; \Delta_N \cap \mathcal{E}^{(1)}. \end{equation}
We conclude that for every $(z,w) \in \Delta_N \cap \mathcal{E}^{(1)}$ and $\mu \in \N$ we have
\begin{equation*} \begin{split}
&\abs{P_{\nu+\mu}(z,w)}^{1/2^{\nu+\mu}} = \prod_{l=1}^{2^{\nu+\mu}} \abs{w - w_l^{(\nu+\mu)}(z)}^{1/2^{\nu+\mu}} = \prod_{j' = 1}^{2^\nu} \prod_{k=1}^{2^\mu} \abs{w - w_k^{(\mu)}(j',\nu;z)}^{1/2^{\nu+\mu}} \\ &= \prod_{1 \le j' \le 2^\nu \atop j' \neq j} \Big(\prod_{k=1}^{2^\mu} \abs{w - w_k^{(\mu)}(j',\nu;z)}^{1/2^\mu}\Big)^{1/2^\nu} \cdot \prod_{k=1}^{2^\mu} \abs{w - w_k^{(\mu)}(j,\nu;z)}^{1/2^{\nu+\mu}} \\ & \le L_N \cdot \prod_{k=1}^{2^\mu} \abs{w - w_k^{(\mu)}(j,\nu;z)}^{1/2^{\nu+\mu}},
\end{split} \end{equation*}
where the last inequality follows from $(\ref{equ_Ml1})$ and $(\ref{equ_Ml2})$.
\end{proof}
\begin{lemma} \label{thm_propE}
The following assertions hold true:
\begin{enumerate}
  \item[(1)] If $\{\varepsilon_l\}$ is decreasing fast enough, then
    \[ d(w, \mathcal{E}_z) \le e^{\varphi(z,w)} \le d(w, 
  \mathcal{E}_z) + (1 + \sqrt{\norm{z}}) \text{ on } \C^n. \]
  \item[(2)] Let $\{\delta_l\}_{l=1}^\infty$ be a decreasing sequence of positive numbers converging to 0. If $\{\varepsilon_l\}$ is decreasing fast enough, then for every $N \in \N$ and $\nu \ge N$ we have
    \[ d(w, \mathcal{E}_z) \le e^\varphi(z,w) \le (L_N + 1) d(w,\mathcal{E}_z)^{1/2^\nu} \text{ on } B^n(0,N) \cap \big(\mathcal{E}^{(1)} \setminus \mathcal{E}^{(\delta_\nu)}\big), \]
  where $L_N$ are the constants from Lemma \ref{thm_Ml}.
\end{enumerate}
\end{lemma}
\begin{proof}
1) Observe that $e^{\varphi} = \lim_{\nu \to \infty} e^{\varphi_\nu} = \lim_{\nu \to \infty} \abs{P_\nu}^{1/2^\nu}$. Hence the first inequaltiy follows from $\abs{P_\nu(z,w)}^{1/2^\nu} = \prod_{j=1}^{2^\nu} \abs{w - w_j^{(\nu)}(z)}^{1/2^\nu} \ge d(w, E_{\nu,z})$ and the fact that $E_{\nu,z} \to \mathcal{E}_z$ in the Hausdorff metric for $\nu \to \infty$.

Let $\{\varepsilon_l\}$ be decreasing so fast that $\sum_{l=1}^\infty \varepsilon_l \sqrt{\abs{a_l}} < 1/2$ and $\varepsilon_l < 1/2^{l+1}$ for every $l \in \N$. As in the proof of Lemma \ref{thm_Ml} we conclude that $E_\nu \subset \{(z,w) \in \C^n : \abs{w} < (1/2)(\sqrt{\norm{z}} + 1)\}$ for every $\nu \in \N$. Now for arbitrary fixed $\nu \in \N$ and $(z,w) \in \C^n$ choose $\tilde{w} \in E_{\nu,z}$ such that $\abs{w-\tilde{w}} = d(w, E_{\nu,z}) $. Then $\abs{w - w_j^{(\nu)}(z)} \le \abs{w-\tilde{w}} + \abs{\tilde{w} - w_j^{(\nu)}(z)} < d(w, E_{\nu,z}) + (\sqrt{\norm{z}} + 1)$ for every $j \in \N_{2^\nu}$, hence $\abs{P_\nu(z,w)}^{1/2^\nu} = \prod_{j=1}^{2^\nu} \abs{w - w_j^{(\nu)}(z)}^{1/2^\nu} < d(w, E_{\nu,z}) + (\sqrt{\norm{z}} + 1)$ and the second inequality follows for $\nu \to \infty$.

2) We only need to show the second inequality. Let $\{\varepsilon_l\}$ be decreasing so fast that the assertion of Lemma \ref{thm_Ml} holds, and observe that $(\ref{equ_Ml})$ remains true with the same constants $\{L_N\}$ if later on we choose $\{\varepsilon_l\}$ to be converging to zero even faster. For every $N \in \N$, let $\gamma_N$ be a positive constant such that
\begin{equation} \label{equ_propE1} \gamma_N L_N \le d(w,\mathcal{E}_z)^{1/2^N} \text{ on } B^n(0,N) \setminus \mathcal{E}^{(\delta_N)}. \end{equation}
Let $\{\varepsilon_l\}$ be decreasing so fast that $\sum_{l=\nu+1}^\infty \varepsilon_l\sqrt{\abs{z_{[l]}-a_l}} \le r_\nu \coloneqq 2^{\nu-1}\gamma_\nu\delta_\nu^{(2^\nu-1)/2^\nu}$ on $B^{n-1}(0,\nu)$ for every $\nu \in \N$. Fix $N \in \N$ and let $\nu \ge N$ be arbitrary. Let $(z,w) \in B^n(0,N) \cap (\mathcal{E}^{(1)} \setminus \mathcal{E}^{(\delta_\nu)})$. Choose $\tilde{w} \in \mathcal{E}_z$ such that $\abs{w-\tilde{w}} = d(w, \mathcal{E}_z)$ and choose $j \in \N_{2^\nu}$ such that $\abs{\tilde{w}-w_j^{(\nu)}(z)} \le d_H(\mathcal{E}_z, E_{\nu,z})$. Then for every $\mu \in \N$ and $k \in \N_{2^\mu}$ we get $\abs{w-w_k^{(\mu)}(j,\nu;z)} \le \abs{w-\tilde{w}} + \abs{\tilde{w}-w_j^{(\nu)}(z)} + \abs{w_j^{(\nu)} - w_k^{(\mu)}(j,\nu;z)} \le d(w,\mathcal{E}_z) + r_\nu + r_\nu$. Hence, by the choice of $r_\nu$, and since $\delta_\nu \le d(w,\mathcal{E}_z)$, it follows that $\prod_{k=1}^{2^\mu} \abs{w-w_k^{(\mu)}(j,\nu;z)}^{1/2^{\nu+\mu}} \le (d(w,\mathcal{E}_z) + 2r_\nu)^{1/2^\nu} \le (d(w,\mathcal{E}_z)^{2^\nu/2^\nu} + 2^\nu d(w,\mathcal{E}_z)^{(2^\nu-1)/2^\nu}\gamma_\nu)^{1/2^\nu} \le d(w,\mathcal{E}_z)^{1/2^\nu} + \gamma_\nu$. Applying Lemma \ref{thm_Ml}, monotonicity of $\{L_N\}$ and $(\ref{equ_propE1})$, we conclude that
\[ \abs{P_{\nu+\mu}(z,w)}^{1/2^{\nu+\mu}} \le L_N\big(d(w,\mathcal{E}_z)^{1/2^\nu} + \gamma_\nu\big) \le (L_N + 1)d(w,\mathcal{E}_z)^{1/2^\nu}. \]
Since here $\mu \in \N$ and $(z,w) \in B^n(0,N) \cap (\mathcal{E}^{(1)} \setminus \mathcal{E}^{(\delta_\nu)})$ are arbitrary, the claim follows from the fact that $e^{\varphi(z,w)} = \lim_{\mu \to \infty} \abs{P_{\nu+\mu}(z,w)}^{1/2^{\nu+\mu}}$.
\end{proof}
Finally, we prove Cauchy type estimates for the derivatives of $e^\varphi$. Let $\N_0 \coloneqq \N \cup \{0\}$. For multiindices $\alpha = (\alpha_1, \alpha_2, \ldots, \alpha_{2n}) \in \N_0^{2n}$ and $\beta = (\beta_1, \beta_2, \ldots, \beta_{2n}) \in \N_0^{2n}$ we write $\alpha \le \beta$ if and only if $\alpha_\nu \le \beta_\nu$ for every $\nu = 1, 2, \ldots, 2n$. Moreover, for every $\alpha \in \N_0^{2n}$ and $r = (r_1, r_2, \ldots, r_{2n}) \in [0, \infty)^{2n}$ we let $r^\alpha \coloneqq r_1^{\alpha_1}r_2^{\alpha_2} \cdots r_{2n}^{\alpha_{2n}}$.
\begin{lemma} \label{thm_poisson}
Let $\Delta^n(a,r) \subset\subset \C^n$ be a polycylinder with polyradius $r \in [0, \infty)^n$, let $u \colon \overline{\Delta}^{\begin{minipage}{1ex} $\scriptstyle \vspace{-1.2ex} \! n$ \end{minipage}}(a,r) \to \R$ be a continuous function such that $u$ is pluriharmonic on $\Delta^n(a,r)$ and let $\alpha \in \N_0^{2n}$. Then
\begin{equation} \label{equ_poisson} \gabs{D^\alpha u(\zeta)} \le C_{\abs{\alpha}} \cdot \frac{\sup_{\xi \in b\Delta^n(a,r)} \abs{u(\xi)}}{\hat{r}^\alpha}\end{equation}
for $\zeta \in \Delta^n(a, r/2)$, where $\hat{r} \coloneqq (r_1, r_1, r_2, r_2, \ldots, r_n, r_n)$ and $C_{\abs{\alpha}} > 0$ is a constant that depends only on $\abs{\alpha}$.
\end{lemma}
\begin{proof} The proof of Lemma \ref{thm_poisson} will be divided in two steps.

\noindent \textsc{Step 1}. {\it Let $v \colon \overline{B^2(a,\rho)} \to \R$ be a continuous function such that $v$ is harmonic on $B^2(a,\rho)$. Then for arbitrary fixed $\theta \in (0,1)$ the inequality
\begin{equation} \label{equ_poisson2}  \Gabs{\frac{\partial v}{\partial x_j}(x)} \le C\, \frac{\sup_{y \in bB^2(a,\rho)} \abs{v(y)}}{\rho} \end{equation}
holds true for every $x \in \overline{B^2(a,\theta\rho)}$ and $j = 1,2$. Here $C = C_\theta$ is a positive constant. }

\noindent \textsc{Proof}. Without loss of generality we can assume that $a=0$. Applying the Poisson integral formula to the function $v$, we see that
\[ v(x) = \frac{1}{2 \pi \rho} \int_{y \in bB^2(0,\rho)} \frac{\rho^2 - \norm{x}^2}{\norm{y-x}^2} v(y) d\sigma(y).\]
Since for $x \in \overline{B^2(0,\theta\rho)}$ and $y \in bB^2(0,\rho)$ we have
\begin{equation*} \Gabs{\frac{\partial}{\partial x_j}\Big(\frac{\rho^2 - \norm{x}^2}{\norm{y-x}^2} \Big)} = \Gabs{\frac{-2x_j\norm{y-x}^2 + 2(\rho^2 - \norm{x}^2)(y_j - x_j)}{\norm{y-x}^4}} \le \frac{8\rho^3 + 4\rho^3}{((1-\theta)\rho)^4} \eqqcolon  \frac{C_\theta}{\rho}, \end{equation*}
it follows that
\[ \Gabs{\frac{\partial v}{\partial x_j}(x)} \le \frac{C_\theta}{\rho} \cdot \frac{1}{2 \pi \rho} \int_{y \in bB^2(0,\rho)} \abs{v(y)} d\sigma(y) \le C_\theta\, \frac{\sup_{y \in bB^2(0,\rho)} \abs{v(y)}}{\rho}. \]
\textsc{Step 2}. {\it Let $\Delta_1 \supset \supset \Delta_2 \supset \supset \cdots$ be defined as $\Delta_m \coloneqq \Delta^n(a, (1 - \sum_{j=1}^m 1/2^{j+1})r)$, $m \in \N$. We show that $(\ref{equ_poisson})$ holds true for every $\zeta \in \overline{\Delta}_{\abs{\alpha}}$. Since $\Delta^n(a, r/2) \subset \Delta_m$ for every $m \in \N$, this proves the claim of the lemma.}

\noindent \textsc{Proof}. We proceed by induction on $k \coloneqq \abs{\alpha}$. Since $u$ is pluriharmonic, the case $k = 1$ is an immediate consequence of $(\ref{equ_poisson2})$ with $\theta = 3/4$. For the step $k \to k+1$, write $\alpha = \tilde{\alpha} + e_\nu$ for some $\nu \in \N$, where $e_\nu = (0, \ldots, 1, \ldots, 0)$ is the $\nu$-th canonical unit vector and $\tilde{\alpha} \in \N_0^{2n}$ satisfies $\abs{\tilde{\alpha}} = k$. Without loss of generality we can assume that $\nu = 1$. Applying $(\ref{equ_poisson2})$ to $v \coloneqq D^{\tilde{\alpha}}u$, $\rho \coloneqq \big(1 - \sum_{j=1}^k 1/2^{j+1}\big) r_1 > r_1/2$ and $\theta = \big(1 - \sum_{j=1}^{k+1} 1/2^{j+1}\big)/\big(1 - \sum_{j=1}^k 1/2^{j+1}\big)$ yields 
\[ \gabs{D^\alpha u(\zeta)} = \Gabs{\frac{\partial v}{\partial \zeta_1}(\zeta)} \le C\, \frac{\sup_{y \in bB^2((a_1,a_2),\rho) \times \{\zeta_3, \ldots, \zeta_{2n}\}} \abs{v(y)}}{\rho} < 2C\, \frac{\sup_{y \in b\Delta_k} \abs{v(y)}}{r_1} \]
for $\zeta \in \Delta_{k+1}$. But, by induction hypothesis, $\abs{v(\zeta)} \le C_k \sup_{\xi \in b\Delta^n(a,r)} \abs{u(\xi)} / \hat{r}^{\tilde{\alpha}}$ for $\zeta \in \overline{\Delta}_k$. Thus we get
\[ \gabs{D^\alpha u(\zeta)} \le 2CC_k \frac{\sup_{\xi \in b\Delta^n(a,r)} \abs{u(\xi)}}{\hat{r}^{\tilde{\alpha}}r_1} = C_{\abs{\alpha}} \frac{\sup_{\xi \in b\Delta^n(a,r)} \abs{u(\xi)}}{\hat{r}^\alpha}  \]
with $C_{\abs{\alpha}} = C_{k+1} \coloneqq 2CC_k$ as desired. This proves Lemma \ref{thm_poisson}.
\end{proof}

Now fix a smooth strictly increasing and strictly convex function $\Lambda \colon [-\infty,\infty) \to [0,\infty)$ such that $\Lambda(x) = e^{-1/e^x}$ for small values of $x$. Observe that the function $\phi = \varphi + \norm{\,\cdot\,}^2 \colon \C^n \to [-\infty, \infty)$ is plurisubharmonic on $\C^n$, smooth and strictly plurisubharmonic on $\C^n \setminus \{\varphi = -\infty\}$. Further, the function $\Lambda$ is smooth strictly increasing and strictly convex. Hence the function $\Phi \coloneqq \Lambda \circ \phi$ is plurisubharmonic on $\C^n$, smooth and strictly plurisubharmonic on $\C^n \setminus \{\varphi = -\infty\}$. In the remaining part of this section we will show that $\Phi$ is also smooth at the points of $\mathcal{E} = \{\varphi = - \infty\}$. 

\noindent \textsc{Step 1}. {\it For every $\alpha \in \N_0^{2n}$ and $N \in \N$, there exist constants $\raisebox{0.2ex}{$\rho$}_N, C_{\alpha, N} > 0$ and $m_\alpha \in \N$ such that
\[ \label{equ_estderphi} \gabs{D^\alpha \phi(z,w)} \le C_{\alpha, N} \frac{\log\big(1/d(w, \mathcal{E}_z)\big)}{d(w, \mathcal{E}_z)^{m_\alpha}} \quad \text{for} \quad (z,w) \in B^n(0,N) \cap \big(\mathcal{E}^{(\raisebox{0.2ex}{\scriptsize$\rho$}_N)} \setminus \mathcal{E}\big). \] }

\noindent \textsc{Proof}. We know from Lemma \ref{thm_propE} that $d(w, \mathcal{E}_z) \le e^\varphi \le d(w, \mathcal{E}_z) + (1 + \sqrt{\norm{z}})$, hence $\abs{\varphi(z,w)} \le \max\{\abs{\log d(w, \mathcal{E}_z)}, \abs{\log (d(w, \mathcal{E}_z) + (1 + \sqrt{\norm{z}}))}\}$. Choose $\raisebox{0.2ex}{$\rho$}_N > 0$ so small that $-\log d(w, \mathcal{E}_z) > 1 + \abs{\log (d(w, \mathcal{E}_z) + (1 + \sqrt{\norm{z}}))}$ on $\overline{\Delta^n_{\mathcal{E}}}(z,w)$ for every $(z,w) \in B^n(0,N) \cap (\mathcal{E}^{(\rho_N)} \setminus \mathcal{E})$, see Lemma \ref{thm_boxE}. Then
\begin{equation} \label{equ_estphi} \abs{\varphi(z,w)} \le \log\big(1/d(w, \mathcal{E}_z)\big) \quad \text{on} \quad \overline{\Delta^n_{\mathcal{E}}}(z,w) \subset \C^n \setminus \mathcal{E} \end{equation}
for $(z,w) \in B^n(0,N) \cap (\mathcal{E}^{(\rho_N)} \setminus \mathcal{E})$. Since $\varphi$ is pluriharmonic in $\C^n \setminus \mathcal{E}$, and in view of Lemma \ref{thm_boxE}, we conclude from Lemma \ref{thm_poisson} that 
\[ \gabs{D^\alpha\varphi(z,w)} \le C_{\abs{\alpha}} \frac{\sup_{\xi \in b\Delta^n_{\mathcal{E}}(z,w)} \abs{\varphi(\xi)}}{\hat{r}_{\mathcal{E}}(z,w)^\alpha} \le C_{\abs{\alpha}}' \frac{\sup_{\xi \in b\Delta^n_{\mathcal{E}}(z,w)} \abs{\varphi(\xi)}}{d(w, \mathcal{E}_z)^{m_\alpha}} \]
on $\C^n \setminus \mathcal{E}$ for suitable constants $C_{\abs{\alpha}}' > 0$ and $m_\alpha \in \N$, where $r_{\mathcal{E}}(z,w) \in (0,\infty)^n$ is defined as $r_{\mathcal{E}}(z,w) \coloneqq \big(C(d(w, \mathcal{E}_z)/2)^2, \ldots, C(d(w, \mathcal{E}_z)/2)^2, d(w, \mathcal{E}_z)/2 \big)$. Using $(\ref{equ_estphi})$, we get
\begin{equation} \label{equ_estphi2} \gabs{D^\alpha\varphi(z,w)} \le C_{\abs{\alpha}}' \frac{\log\big(1/d(w, \mathcal{E}_z)\big)}{d(w, \mathcal{E}_z)^{m_\alpha}} \end{equation}
for $(z,w) \in B^n(0,N) \cap (\mathcal{E}^{(\raisebox{0.2ex}{\scriptsize$\rho$}_N)} \setminus \mathcal{E})$. Moreover, for every $N \in \N$, there exists a constant $C_N'' > 0$ such that $\abs{D^\alpha \norm{\,\cdot\,}^2} \le C_N''$ on $B^n(0,N)$ for every $\alpha \in \N_0^{2n}$. Since, by the choice of $\raisebox{0.2ex}{$\rho$}_N$, we have $\log(1/d(w, \mathcal{E}_z)) > 1$ on $B^n(0,N) \cap (\mathcal{E}^{(\raisebox{0.2ex}{\scriptsize$\rho$}_N)} \setminus \mathcal{E})$, it follows together with $(\ref{equ_estphi2})$ that
\[ \gabs{D^\alpha \phi(z,w)} \le \gabs{D^\alpha \varphi(z,w)} + \gabs{D^\alpha \norm{\,\cdot\,}^2(z,w)} \le C_{\alpha, N} \frac{\log\big(1/d(w, \mathcal{E}_z)\big)}{d(w, \mathcal{E}_z)^{m_\alpha}} \]
for $(z,w) \in B^n(0,N) \cap (\mathcal{E}^{(\raisebox{0.2ex}{\scriptsize$\rho$}_N)} \setminus \mathcal{E})$ and $C_{\alpha,N} \coloneqq C_{\abs{\alpha}}' + C_N''$.

\noindent \textsc{Step 2}. {\it For every $\alpha \in \N_0^{2n}$ and $N \in \N$, there exists a polynomial $P_{\alpha,N} \in \R[x]$ with nonnegative coefficients such that 
\[ \gabs{D^\alpha\Phi(z,w)} \le P_{\alpha,N}\big(1/d(w,\mathcal{E}_z)\big) e^{-1/e^{\phi(z,w)}} \quad \text{for} \quad (z,w) \in B^n(0,N) \cap \big(\mathcal{E}^{(\raisebox{0.2ex}{\scriptsize$\rho$}_N)} \setminus \mathcal{E}\big).\] }

\noindent \textsc{Proof}. Recall that $\Phi = e^{-1/e^\phi}$ (since the smoothness of $\Phi$ in $\mathcal{E}$ depends only on the values $\Lambda(x)$ for $0 < x << 1$, we can assume here for simplicity that $\Lambda(x) \equiv e^{-1/e^x}$). Let $\beta_2, \beta_3, \ldots, \beta_{\langle\alpha\rangle} \in \N_0^{2n}$ be pairwise distinct multiindices such that $\{\beta \in \N_0^{2n} : 0 \le \beta \le \alpha, \beta \neq 0\} = \{\beta_2, \beta_3, \ldots, \beta_{\langle\alpha\rangle}\}$. Then an easy induction on $\abs{\alpha}$ shows that there exists a polynomial $Q_\alpha \in \R[x_1, \ldots, x_{\langle\alpha\rangle}]$, $Q_\alpha(x) = \sum_{\gamma \in \N_0^{\langle\alpha\rangle}} a_\gamma x^\gamma$, such that $D^\alpha \Phi = Q_\alpha(1/e^\phi, D^{\beta_2}\phi, \ldots, D^{\beta_{\langle\alpha\rangle}}\phi) e^{-1/e^\phi}$. Define $\tilde{Q}_\alpha \in \R[x_1, \ldots, x_{\langle\alpha\rangle}]$ as $\tilde{Q}_\alpha(x) = \sum_{\gamma \in \N_0^{\langle\alpha\rangle}} \abs{a_\gamma} x^\gamma$. Then
\[\gabs{D^\alpha \Phi(z,w)} \le \tilde{Q}_\alpha \big(1/e^\phi, \gabs{D^{\beta_2}\phi}, \ldots, \gabs{D^{\beta_{\langle\alpha\rangle}} \phi} \big)(z,w) \cdot e^{-1/e^{\phi(z,w)}}. \]
From Lemma \ref{thm_propE} we know that $1/e^{\phi(z,w)} =  1/(e^{\norm{(z,w)}^2} e^{\varphi(z,w)}) \le 1/d(w,\mathcal{E}_z)$ for every $(z,w) \in \C^n$. Applying this and Step 1 to the above formula, we get
\begin{equation*} \begin{split}  \gabs{D^\alpha & \varphi^\ast(z,w)} \\ & \le \tilde{Q}_\alpha\bigg(\frac{1}{d(w, \mathcal{E}_z)}, C_{\beta_2, N}\frac{\log\big(1/d(w, \mathcal{E}_z)\big)}{d(w, \mathcal{E}_z)^{m_{\beta_2}}}, \ldots, C_{\beta_{\langle\alpha\rangle}, N}\frac{\log\big(1/d(w, \mathcal{E}_z)\big)}{d(w, \mathcal{E}_z)^{m_{\beta_{\langle\alpha\rangle}}}} \bigg) e^{-1/e^{\phi(z,w)}} \\ & = \tilde{P}_{\alpha,N} \bigg(\frac{1}{d(w, \mathcal{E}_z)}, \log \Big(\frac{1}{d(w, \mathcal{E}_z)}\Big) \bigg) e^{-1/e^{\phi(z,w)}} \le P_{\alpha,N}\big(1/d(w,\mathcal{E}_z)\big)e^{-1/e^{\phi(z,w)}} \end{split} \end{equation*}
on $B^n(0,N) \cap \big(\mathcal{E}^{(\rho_N)} \setminus \mathcal{E}\big)$ for suitable polynomials $\tilde{P}_{\alpha, N} \in \R[x_1, x_2]$ and $P_{\alpha,N} \in \R[x]$ with nonnegative coefficients.

\noindent \textsc{Step 3}. {\it For every $(z_0, w_0) \in \mathcal{E}$ and $\alpha \in \N_0^{2n}$\!, one has $\lim_{(z,w) \to (z_0,w_0)} D^\alpha \Phi(z,w) \!=\! 0$. }

\noindent \textsc{Proof}. By a standard application of l'Hospital's rule, $\lim_{x \to \infty} P(x)e^{-cx^{1/m}} = 0$ for every polynomial $P \in \R[x]$, $c >0$ and $m \in \N$. Hence for every $\nu \in \N$ there exists a constant $\delta_\nu > 0$ such that
\begin{equation} \label{equ_deffct1} P_{\alpha,N}\big(1/d(w,\mathcal{E}_z)\big) e^{-1/[e^{N^2}(L_N+1)d(w,\mathcal{E}_z)^{1/2^\nu}]} < 1/\nu \;\;\; \text{for} \;\;\; (z,w) \in \overline{\mathcal{E}^{(\delta_\nu)}} \end{equation}
for every $N \in \N$ and $\alpha \in \N_0^{2n}$ such that $N, \abs{\alpha} \le \nu$, where $\{L_N\}$ are the constants from Lemma \ref{thm_Ml}. Clearly, we can assume that $\delta_\nu < \min\{\rho_\nu, 1\}$ and $\delta_{\nu+1} < \delta_{\nu}$ for every $\nu \in \N$ and that $\lim_{\nu \to \infty} \delta_\nu = 0$. Let $\{\varepsilon_l\}$ be decreasing so fast that 
\begin{equation} \label{equ_deffct2} e^{\varphi(z,w)} \le (L_N + 1)d(w,\mathcal{E}_z)^{1/2^\nu} \quad \text{for} \quad (z,w) \in B^n(0,N) \cap \big(\mathcal{E}^{(1)} \setminus \mathcal{E}^{(\delta_{\nu+1})}\big) \end{equation}
for every $N \in \N$ and $\nu \ge \N$. This is always possible as is shown in the second part of Lemma \ref{thm_propE}. Now fix $N \in \N$ and $\alpha \in \N_0^{2n}$. Then, since $\phi = \varphi + \norm{\,\cdot\,}^2$, we conclude from $(\ref{equ_deffct1})$, $(\ref{equ_deffct2})$ and Step 2 that
\begin{equation*} \begin{split} \gabs{D^\alpha\Phi(z,w)} \le P_{\alpha, N}\big(1/&d(w,\mathcal{E}_z)\big)e^{-1/e^{\phi(z,w)}} \\ &\le  P_{\alpha, N}\big(1/d(w,\mathcal{E}_z)\big) e^{-1/[e^{N^2}(L_N+1)d(w,\mathcal{E}_z)^{1/2^\nu}]} < 1/\nu \end{split} \end{equation*}
for every $(z,w) \in B^n(0,N) \cap (\overline{\mathcal{E}^{(\delta_\nu)}} \setminus \mathcal{E}^{(\delta_{\nu+1})})$. Thus it follows from  $\lim_{\nu \to \infty} \delta_\nu = 0$ that $\lim_{(z,w) \to (z_0, w_0)} D^\alpha \Phi(z,w) = 0$ for every $(z_0,w_0) \in B^n(0,N) \cap \mathcal{E}$. Since this holds true for every $N \in \N$ and $\alpha \in \N_0^{2n}$, the proof is complete. $\hfill \Box$

\section{Pseudoconcavity of higher order cores} \label{sec_higherordercore}
Analogously to the notion of the core of a domain one can for every $q = 1, 2, \ldots, n$ define the notion of the core of order $q$. Let $\mathcal{M}$ be a complex manifold of complex dimension $n$ and let $\Omega \subset \mathcal{M}$ be a domain. Denote by $\mathcal{B}^\infty_{psh}(\Omega)$ the set of all smooth plurisubharmonic functions on $\Omega$ that are bounded from above.

\begin{definition} \label{def_corehigherorder}
For every $q = 1, \ldots, n$, we call the set
\[ \mathfrak{c}_q(\Omega) \coloneqq \{z \in \Omega : \rank \Lev(\varphi)(z, \,\cdot\,) \le n - q \text{ for every } \varphi \in \mathcal{B}^\infty_{psh}(\Omega) \} \]
the \textit{core of order $q$} of $\Omega$.
\end{definition}

\noindent \textbf{Remark.} Analoguous definitions are possible for $\mathcal{C}^s$-smooth functions on complex manifolds for every $s \ge 2$, and for $\mathcal{C}^\infty$-smooth functions on locally irreducible complex spaces. 

It follows immediately from the definition that $\mathfrak{c}_1(\Omega) \supset \mathfrak{c}_2(\Omega) \supset \cdots \supset \mathfrak{c}_n(\Omega)$ and that $\mathfrak{c}(\Omega) = \mathfrak{c}_1(\Omega)$. Let us now illustrate this notion with the following example.

\begin{example}
For generic $C \in \R$, let
\[ \Omega \coloneqq \big\{(z_1, z_2, z_3) \in \C^3 : \log\abs{z_1} + \log(\abs{z_2} + \abs{z_3}) + \big(\abs{z_1}^2 + \abs{z_2}^2 + \abs{z_3}^2\big) < C \big\}. \]
Then $\Omega$ is strictly pseudoconvex with smooth boundary and, in view of Liouville's theorem, $\mathfrak{c}(\Omega) = l \cup \Pi$, where $l = \{(z_1,0,0) \in \C^3 : z_1 \in \C\}$ and $\Pi = \{(0, z_2, z_3) \in \C^3 : z_2, z_3 \in \C\}$.  
\end{example}

In the above example one can easily see that $\mathfrak{c}_1(\Omega) = l \cup \Pi$ is $1$-pseudoconcave and $\mathfrak{c}_2(\Omega) = \Pi$ is $2$-pseudoconcave. We know from Theorem \ref{thm_pseudoconcave} that $\mathfrak{c}_1(\Omega)$ is always $1$-pseudoconcave in $\Omega$ for every domain $\Omega \subset \mathcal{M}$. Moreover, in view of the discussion on Liouville type properties of the core in Section \ref{sec_liouville}, observe that the following generalization of Lemma \ref{thm_pseudoconcaveconstant} holds true for every $q = 1, \ldots, n$: every smooth plurisubharmonic function $\varphi$ which is defined on an open neighbourhood of a closed $q$-pseudoconcave set $A \subset \mathcal{M}$ and which is constant on $A$ satisfies $\rank \Lev(\varphi)(z, \,\cdot\,) \le n-q$ for every $z \in A$ (by the results from \cite{Slodkowski86}, $(q-1)$-plurisubharmonic functions have the local maximum property on $q$-pseudoconcave sets for every $q= 1, \ldots, n$; thus the statement follows by the same argument as in the proof of Lemma \ref{thm_pseudoconcaveconstant}). These observations lead us to the following question: Is it always true that $\mathfrak{c}_q(\Omega)$ is $q$-pseudoconcave in $\Omega$ for $q > 1$? In general the answer is negative and we will construct here explicit counterexamples. The main results of this section are summarized in the following theorem.

\begin{theorem} \label{thm_coreconcavity}
For every $n \ge 2$ and every $q = 1, \ldots, n$, $q' = 0, \ldots, n-1$ such that $(q,q') \neq (1,0)$, there exists a strictly pseudoconvex domain $\Omega \subset \C^n$ with smooth boundary such that $\mathfrak{c}_q(\Omega)$ is $q'$-pseudoconcave but not $(q'+1)$-pseudoconcave.
\end{theorem} 

The case $q' = 0$ is rather easy (recall that a set $A \subset \C^n$ is $0$-pseudoconcave if and only if it is closed). Indeed, fix arbitrary $q \in \{2, \ldots, n\}$ and for generic $C \in \R$ consider the set
\[ \Omega \coloneqq \big\{ z \in \C^n : \norm{z}^2 + \sum_{j=1}^q \log(\norm{z}^2 - \abs{z_j}^2) < C \big\}.\]
After possibly passing to a suitable connected component, $\Omega$ is a strictly pseudoconvex domain with smooth boundary such that $L \coloneqq \bigcup_{j=1}^q \{z \in \C^n : z_k = 0 \text{ for every } k \neq j\} \subset \Omega$. By Liouville's theorem, every smooth and bounded from above plurisubharmonic function on $\Omega$ has to be constant on $L$. In particular, $L \subset \mathfrak{c}(\Omega)$ and $0 \in \mathfrak{c}_q(\Omega)$. Moreover, a straightforward computation shows that $\varphi(z) \coloneqq \exp(\norm{z}^2 + \sum_{j=1}^q \log(\norm{z}^2 - \abs{z_j}^2))$ is a smooth and bounded from above plurisubharmonic function on $\Omega$ such that $\varphi$ is strictly plurisubharmonic on $\Omega \setminus L$ and such that $\rank \Lev(\varphi)(z,\,\cdot\,) = n-1$ for every $z \in L \setminus \{0\}$. Hence $\mathfrak{c}(\Omega) \subset L$ and $\mathfrak{c}_q(\Omega) \subset \{0\}$. It follows that $\mathfrak{c}(\Omega) = L$ and $\mathfrak{c}_q(\Omega) = \{0\}$. In particular, $\mathfrak{c}_q(\Omega)$ is not $1$-pseudoconcave.

The case $q' > 0$ is more complicated. In fact, in this situation the domains $\Omega$ of the theorem will be chosen in such a way that $\mathfrak{c}_1(\Omega) = \mathfrak{c}_2(\Omega) = \cdots = \mathfrak{c}_n(\Omega) = \mathcal{E}$, where $\mathcal{E}$ is a Wermer type set similar to the one constructed in \cite{HarzShcherbinaTomassini12}. In particular, we can guarantee that the core $\mathfrak{c}_q(\Omega)$ in the Theorem \ref{thm_coreconcavity} has the following additional properties: $\mathfrak{c}_q(\Omega)$ is complete pluripolar, $\mathfrak{c}_q(\Omega)$ contains no analytic variety of positive dimension and $\widehat{bB^n(0,R) \cap \mathfrak{c}_q(\Omega)} = \overline{B^n(0,R)} \cap \mathfrak{c}_q(\Omega)$ for every $R > 0$, where $\widehat{bB^n(0,R) \cap \mathfrak{c}_q(\Omega)}$ denotes the polynomial hull of $bB^n(0,R) \cap \mathfrak{c}_q(\Omega)$.

Before we begin with the construction of our examples for the case $q' > 0$, we state the following generalization of the Main Theorem.

\begin{theorem} \label{thm_extendedmain}
Let $\Omega$ be a strictly pseudoconvex domain with smooth boundary in a complex manifold $\mathcal{M}$. Then there exists a global defining function $\varphi$ for $\Omega$ such that $\varphi$ is strictly plurisubharmonic in the complement of $\mathfrak{c}_1(\Omega)$ and $\rank \Lev(\varphi)(z,\,\cdot\,) = n-q$ for every $z \in \mathfrak{c}_q(\Omega) \setminus \mathfrak{c}_{q+1}(\Omega)$ and $q = 1, 2, \ldots, n$. 
\end{theorem}
\begin{proof}
We know from the Main Theorem that there exists a smooth global defining function $\varphi_1$ for $\Omega$ such that $\rank \Lev(\varphi_1)(z,\,\cdot\,) = n$ for every $z \notin \mathfrak{c}_1(\Omega)$. Observe that by repeating the same arguments as in the proof of the Main Theorem we can also construct for each $q = 2, 3, \ldots, n$ a smooth global defining function $\varphi_q$ for $\Omega$ such that $\rank \Lev(\varphi_q)(z,\,\cdot\,) \ge n-q+1$ for every $z \notin \mathfrak{c}_{q}(\Omega)$. Then $\varphi \coloneqq \sum_{q=1}^n \varphi_q$ is a function as desired.
\end{proof}

We would like to point out here that the most essential achievement of the Main Theorem and Theorem \ref{thm_extendedmain} is the proof of existence of global defining functions (the construction of these functions is carried out in Theorem \ref{thm_qdeffct_manifolds}). The proof of the additional properties of these functions, namely, of being strictly plurisubharmonic outside the core $\mathfrak{c}(\Omega)$ or having the corresponding rank of the Levi form outside the core $\mathfrak{c}_q(\Omega)$ of order $q$ for every $q = 1,2,\ldots,n$, is simple and rather standard. Note also that a version of the last argument as well as the definition of a notion similar to our notion of the core $\mathfrak{c}_q(\Omega)$, $q=1,2,\ldots,n$, in the different setting of exhaustion functions was given earlier in Lemma 3.1 of \cite{SlodkowskiTomassini04}.

We now turn to the proof of Theorem \ref{thm_coreconcavity}. For this we need to generalize the construction of the Wermer type set from \cite{HarzShcherbinaTomassini12}: Let $(z,w) = (z_1, \ldots, z_k,w_1, \ldots, w_{n-k})$ denote the coordinates in $\C^n$ and for each $\nu \in \N$ let $\N_\nu \coloneqq \{1, 2, \ldots, \nu\}$. For each $(p,q) \in \N_k \times \N_{n-k}$, fix an everywhere dense subset $\{a_l^{p,q}\}_{l=1}^\infty$ of $\C$ such that $a_l^{p,q} \neq a_{l'}^{p,q'}$ if $(q,l) \neq (q',l')$. Further, fix a bijection $\Phi \coloneqq ([\,\cdot\,], \langle\,\cdot\,\rangle, \phi) \colon \N \to \N_k \times \N_{n-k} \times \N$ and define a sequence $\{a_l\}_{l=1}^\infty$ in $\C$ by letting $a_l \coloneqq a^{[l],\langle l \rangle}_{\phi(l)}$. Moreover, let $\{\varepsilon_l\}_{l=1}^\infty$ be a decreasing sequence of positive numbers converging to zero that we consider to be fixed, but that will be further specified later on. For every $\nu \in \N$ and $q \in \N_{n-k}$, we define sets $E_{\nu,q}, E_\nu \subset \C^n$ as
\begin{gather*}
  E_{\nu,q} \coloneqq \big\{ (z,w) \in \C^n : w_q = \sum_{l \in L_\nu^{\ast, q}} \varepsilon_l \sqrt{z_{[l]}-a_l} \big\}, \\
  E_\nu \coloneqq \big\{ (z,w) \in \C^n : w = \big(\!\sum_{l \in L_\nu^{\ast, 1}} \varepsilon_l \sqrt{z_{[l]}-a_l}\,, \ldots, \!\!\!\sum_{l \in L_\nu^{\ast, n-k}} \varepsilon_l \sqrt{z_{[l]}-a_l}\,\big) \big\}, 
\end{gather*}
where $L_\nu^{\ast, q} \coloneqq \{l \in \N_\nu : \langle l \rangle = q\}$. Observe that $E_\nu = \bigcap_{q=1}^{n-k} E_{\nu,q} = \{(z,w) \in \C^n : w = \sum_{l=1}^\nu \varepsilon_l \textbf{e}_{\langle l \rangle} \sqrt{z_{[l]}-a_l}\}$, where for every $q \in \N_{n-k}$ we denote by $\textbf{e}_q \coloneqq (0, \ldots, 1, \ldots, 0)$ the $q$-th unit vector in $\C^{n-k}$. Note further that $\sum_{l=1}^\nu \varepsilon_l \textbf{e}_{\langle l \rangle}\sqrt{z_{[l]} - a_l}$ takes $2^\nu$ values at each point $z \in \C^k$ (counted with multiplicities). Thus there exist single-valued maps $w^{(\nu)}_1, \ldots, w^{(\nu)}_{2^\nu} \colon \C_z^k \to \C_w^{n-k}$ such that $\sum_{l=1}^\nu \varepsilon_l \textbf{e}_{\langle l \rangle}\sqrt{z_{[l]} - a_l} = \big\{w_j^{(\nu)}(z) : j = 1, \ldots, 2^\nu \big\}$ for every $z \in \C^k$. For every $\nu\in\N$ and $q \in \N_{n-k}$, define maps $P_{\nu,q} \colon \C^n \to \C$, $P_\nu \colon \C^n \to \C^{n-k}$ as
\begin{gather*}
P_{\nu,q}(z,w) \coloneqq \big(w_q - w_1^{(\nu)}(z)_q\big) \cdots \big(w_q - w_{2^\nu}^{(\nu)}(z)_q\big), \\ P_\nu(z,w) \coloneqq \big(P_{\nu,1}(z,w), \ldots, P_{\nu, n-k}(z,w)\big),
\end{gather*}
where for every $j \in \N_{2^\nu}$ we denote by $w_j^{(\nu)}(z)_q$ the $q$-th coordinate of $w_j^{(\nu)}(z) \in \C^{n-k}$. Observe that $E_{\nu,q} = \{P_{\nu,q} = 0\}$ and $E_\nu = \{P_\nu = 0\}$. As in Lemma 2.1 of \cite{HarzShcherbinaTomassini12}, we see that each $P_{\nu,q}$ is a holomorphic polynomial. Moreover, one easily proves the following three lemmas.

\begin{lemma}
If $\{\varepsilon_l\}$ is decreasing fast enough, then for every $R>0$ the sequences $\{E_{\nu,q} \cap \overline{B^n(0,R)}\}_{\nu=1}^\infty$ and $\{E_\nu \cap \overline{B^n(0,R)}\}_{\nu=1}^\infty$ converge in the Hausdorff metric to closed sets $\mathcal{E}_{(R),q}$ and $\mathcal{E}_{(R)}$, $q \in \N_{n-k}$, respectively. The sets $\mathcal{E}_q \coloneqq \bigcup_{R>0} \mathcal{E}_{(R),q}$ and $\mathcal{E} \coloneqq \bigcup_{R>0} \mathcal{E}_{(R)}$ are unbounded closed connected subsets of $\C^n$ and $\mathcal{E} = \bigcap_{q=1}^{n-k} \mathcal{E}_q$. Moreover, $\mathcal{E}_z \coloneqq \{w \in \C^{n-k} : (z,w) \in \mathcal{E}\}$ is compact for every $z \in \C^k$.
\end{lemma}
\begin{proof}
All claims, except for the assertions on connectedness, follow immediately from Lemma 2.2 in \cite{HarzShcherbinaTomassini12} and the equality $E_\nu = \bigcap_{q=1}^{n-k} E_{\nu,q}$. The fact that the sets $\mathcal{E}_q$, $q \in \N_{n-k}$, and $\mathcal{E}$ are connected can be shown in the same way as in Lemma \ref{thm_connectedE}.
\end{proof}

\begin{lemma}
If $\{\varepsilon_l\}$ is decreasing fast enough, then for every $q \in \N_{n-k}$ the sequence $\{\frac{1}{2^\nu} \log\abs{P_{\nu,q}}\}_{\nu=1}^\infty$ converges uniformly on compact subsets of $\C^n \setminus \mathcal{E}_q$ to a pluriharmonic function $\varphi_q \colon \C^n \setminus \mathcal{E}_q \to \R$ and $\lim_{(z,w) \to (z_0,w_0)} \varphi_q(z,w) = -\infty$ for every $(z_0,w_0) \in \mathcal{E}_q$. In particular, $\varphi_q$ has a unique extension to a plurisubharmonic function on $\C^n$.
\end{lemma}
\begin{proof}
See Lemma 5.1 in \cite{HarzShcherbinaTomassini12}.
\end{proof}

\begin{lemma} \label{thm_haushölder2}
Let $\underline{\mathcal{E}}\colon (\C^k,d_{\norm{\cdot}}) \to (\mathcal{F}(\C^{n-k}), d_H)$ be the map defined by $\underline{\mathcal{E}}(z) \coloneqq \{w \in \C^{n-k} : (z,w) \in \mathcal{E}\}$, where $\C^k$ is the metric space of all $k$-tupels of complex numbers with the standard euclidean metric $d_{\norm{\cdot}}$ and $\mathcal{F}(\C^{n-k})$ is the metric space of all nonempty compact subsets of $\C^{n-k}$ with the Hausdorff metric $d_H$. Then there exists a constant $M>0$ such that the map $\underline{\mathcal{E}}$ is $(M, 1/2)$-H\"older continuous.
\end{lemma}
\begin{proof}
The proof is essentialy the same as in Lemma \ref{thm_haushölder}.
\end{proof}

Define a function $\Psi^\ast \colon \C^n \to [-\infty,\infty)$ as 
\[ \Psi^\ast(z,w) \coloneqq \log\big(e^{\varphi_1(z,w)} + \cdots + e^{\varphi_{n-k}(z,w)}\big) + (\norm{z}^2 + \norm{w}^2). \]
Then, in view of plurisubharmonicity of the functions $\varphi_1, \varphi_2, \ldots, \varphi_{n-k}$, one easily checks that $\Psi^\ast$ is strictly plurisubharmonic on $\C^n$ and continuous outside $\mathcal{E} = \{\Psi^\ast = -\infty\}$. Applying Richberg's smoothing procedure (see, for example, Theorem I.5.21 in \cite{DemaillyXX}), we will obtain a plurisubharmonic function $\Psi \colon \C^n \to [-\infty, \infty)$ such that $\Psi$ is smooth and strictly plurisubharmonic on $\C^n \setminus \mathcal{E}$ and $\mathcal{E} = \{\Psi = - \infty\}$. Fix a regular value $C^\ast$ of $\Psi$ and define
\[ \Omega \coloneqq \{(z,w) \in \C^n : \Psi(z,w) < C^\ast\}. \]
Then, after possibly replacing $\Omega$ by a suitable connected component, $\Omega$ is a strictly pseudoconvex domain with smooth boundary such that $\mathcal{E} \subset \Omega$. 

Now the general idea of the example is as follows: The set $\mathcal{E}$ is $k$-pseudoconcave but not $(k+1)$-pseudoconcave, since it is essentially a $k$-dimensional object. On the other hand, despite possibly large codimension of $\mathcal{E}$ in $\C^n$, for every $(p,q) \in \N_k \times \N_{n-k}$ there is an everywhere dense sequence of root branches along the $z_p$-axis originating in $w_q$-direction. This geometric property together with a Liouville type theorem for $\mathcal{E}$ will enforce that the Levi form of every smooth and bounded from above plurisubharmonic function on $\Omega$ has to vanish along all coordinate directions at every point of $\mathcal{E}$. The later assertion guarantees that $\mathfrak{c}_q(\Omega) = \mathcal{E}$ for every $q = 1, \ldots, n$. Letting $k$ vary between $1$ and $n-1$, this proves Theorem \ref{thm_coreconcavity}. The above considerations are made precise by the following three lemmas.

\begin{lemma} \label{thm_Ekpseudoconcave}
If $\{\varepsilon_l\}$ is decreasing fast enough, then $\mathcal{E}$ is $k$-pseudoconcave but not $(k+1)$-pseudoconcave.
\end{lemma}
\begin{proof}
Assume, to get a contradiction, that $\mathcal{E}$ is not $k$-pseudoconcave. Then there exists an $(n-k,k)$ Hartogs figure $H = \{(\zeta,\eta) \in \Delta^{n-k} \times \Delta^k : \norm{\zeta}_\infty > r_1 \text{ or } \norm{\eta}_\infty < r_2\}$ and an injective holomorphic mapping $\Phi \colon \hat{H} \to \C^n$ such that $\Phi(H) \subset \C^n \setminus \mathcal{E}$ but $\Phi(\hat{H}) \cap \mathcal{E} \neq \varnothing$. After possibly shrinking $H$, one can easily see that for $\nu \in \N$ large enough the pure $k$-dimensional varieties $E_\nu$ will also satisfy the conditions $\Phi(\overline{H}) \subset \C^n \setminus E_\nu$ and $\Phi(\hat{H}) \cap E_\nu \neq \varnothing$. Then $V \coloneqq \Phi(\hat{H})$ is a relatively compact subset of $\C^n$ such that the $(k-1)$-plurisubharmonic function $\varphi \coloneqq -\log\norm{\eta} \circ \Phi^{-1}$ satisfies $\max_{E_\nu \cap V} \varphi > \max_{E_\nu \cap bV} \varphi$. This contradicts the local maximum property of $(k-1)$-plurisubharmonic functions on $k$-dimensional analytic varieties, see Corollary 5.3 in \cite{Slodkowski86}. (An easy way to avoid the reference to \cite{Slodkowski86} is the following: For a fixed regular value $c \in \C^{n-k}$ of $P_\nu$ close enough to zero the complex $k$-dimensional manifold $\mathcal{M} \coloneqq \{P_\nu = c\}$ also satisfies $\Phi(\overline{H}) \subset \C^n \setminus \mathcal{M}$ and $\Phi(\hat{H}) \cap \mathcal{M} \neq \varnothing$. In particular, $\max_{\mathcal{M} \cap V} \varphi > \max_{\mathcal{M} \cap bV} \varphi$. Then for $\varepsilon > 0$ small enough the function $\widetilde{\varphi} \coloneqq (-\log\norm{\eta} +  \varepsilon(\norm{\zeta}^2 + \norm{\eta}^2)) \circ \Phi^{-1}$ attains a maximum along $\mathcal{M} \cap V$ which contradicts the fact that the Levi form of $\widetilde{\varphi}|_{\mathcal{M} \cap V}$ has at least one positive eigenvalue at every point of $\mathcal{M} \cap V$.)

To see that $\mathcal{E}$ is not $(k+1)$-pseudoconcave, let $z_0 \in \C^k$ be an arbitrary fixed point and let $\co \mathcal{E}_{z_0}$ denote the convex hull of $\mathcal{E}_{z_0}$. We claim that the set $A \coloneqq b(\co \mathcal{E}_{z_0}) \cap \mathcal{E}_{z_0}$ is nonempty. Indeed, by compactness of $\mathcal{E}_{z_0}$, we conclude that $\co \mathcal{E}_{z_0}$ is compact too, and thus it follows easily from Minkowski's theorem that A contains the nonempty set of extreme points of $\co \mathcal{E}_{z_0}$. Hence we can find a supporting real hyperplane $L \subset \C^{n-k}_w$ for $\mathcal{E}_{z_0}$ that contains at least one point $w_0 \in \mathcal{E}_{z_0}$. Since $L$ contains an $(n-(k+1))$-dimensional complex subspace, one now constructs easily an $(n-(k+1),k+1)$ Hartogs figure $H = \{(\zeta,\eta) \in \Delta^{n-(k+1)} \times \Delta^{k+1} : \norm{\zeta}_\infty > r_1 \text{ or } \norm{\eta}_\infty < r_2\}$ and an injective holomorphic mapping $\Phi \colon \hat{H} \to \C^n$ such that $\Phi(H) \subset \C^n \setminus \mathcal{E}$ but $\Phi(\hat{H}) \cap \mathcal{E} \neq \varnothing$.
%(see also Theorem 2 in [Fu] for more details)
\end{proof}

\begin{lemma} \label{thm_liouville2}
Let $\varphi$ be a continuous plurisubharmonic function defined on an open neighbourhood $U \subset \C^n$ of $\mathcal{E}$. If $\varphi$ is bounded from above, then $\varphi \equiv C$ on $\mathcal{E}$ for some $C \in \R$. 
\end{lemma}
\begin{proof}
Using the same argument as in Step 2 of the proof of Theorem \ref{thm_liouville}, we can restrict ourselves to the case $k = 1$.  Choose an increasing sequence $\{B_\nu\}_{\nu = \nu_0}^\infty$ of open sets $B_\nu \subset \C$ such that $\bigcup_{\nu = \nu_0}^\infty B_\nu = \C$ and such that $E_\nu \cap (B_\nu \times \C^{n-1}) \subset U$ for every $\nu \ge \nu_0$. Moreover, define functions $\phi_\nu \colon B_\nu \to \R$, $\nu \ge \nu_0$, as $\phi_\nu(z) \coloneqq \max_{1 \le j \le 2^\nu} \varphi(z, w_j^{(\nu)}(z))$ and let $\phi(z) \coloneqq \sup_{w \in \mathcal{E}(z)} \varphi(z,w)$. Since on compact subsets of $\C^n$ the sequence $\{E_\nu\}$ converges in the Hausdorff metric to $\mathcal{E}$, and since $\varphi$ is continuous, one easily sees that $\lim_{\nu \to \infty} \phi_\nu = \phi$ uniformly on compact subsets of $\C$. Moreover, every function $\varphi_\nu$ is subharmonic, since on each convex set in the complement of the polar set $\{a_1, \ldots, a_\nu\}$ the functions $w_1^{(\nu)}, \ldots, w_{2^\nu}^{(\nu)}$ can be chosen to be holomorphic. In particular, $\phi$ is a subharmonic function on $\C$ that is bounded from above, hence, in view of Liouville's theorem, $\phi \equiv C$ for some $C \in \R$. The proof can now be completed in the same way as in Step 1 of Theorem \ref{thm_liouville}.
\end{proof}

\noindent \textbf{Remark.} In the two-dimensional case of Step 1 in the proof of Theorem \ref{thm_liouville} the subharmonicity of the function $\phi$ was obtained by using Theorem II from \cite{Slodkowski81}. A more general version of this result, which also works for the case $n >2$, was claimed in Theorem 2.3 of \cite{Slodkowski83}, but since it does not have a proof, and since we were not able to find a reference with the proof, we have included the above argument. Observe that if we replace our argument by the result from \cite{Slodkowski83}, then we can drop the assumption on continuity of the function $\varphi$.

\begin{lemma} \label{thm_OmegacoresE}
If $\{\varepsilon_l\}$ is decreasing fast enough, then $\mathfrak{c}_q(\Omega) = \mathcal{E}$ for every $q =1, \ldots, n$.
\end{lemma}
\begin{proof}
If $\{\eta_j\}_{j=1}^\infty$ is a sequence of positive numbers that is converging to zero fast enough, then $\widetilde{\Psi}(z,w) \coloneqq \sum_{j=1}^\infty \eta_j \, \widetilde{\max}_1(\Psi-C^\ast, -j)$ is a smooth global defining function for $\Omega$ that is strictly plurisubharmonic outside $\mathcal{E}$. Hence we only have to show that the Levi form of every smooth and bounded from above plurisubharmonic function on $\Omega$ vanishes identically on $\mathcal{E}$. In order to do so, observe first that it suffices to prove the claim in the case $k = 1$. Indeed, for every $p \in \N_k$ and every $\xi = (\xi', \xi'') \in \C^{p-1} \times \C^{k-p}$ the set 
\[ \mathcal{E}_{p,\xi} \coloneqq \mathcal{E} \cap \big[(\{\xi'\} \times \C_{z_p} \times \{\xi''\})  \times \C^{n-k}_w \big] \]
is, up to inclusion of $\C_{z_p} \times \C^{n-k}_w$ into $\C^n$, of the form $\mathcal{E}_{p,\xi} = \bigcup_{w \in W} [\mathcal{F}_p + (0,w)]$ for a Wermer type set $\mathcal{F}_p \subset \C_{z_p} \times \C^{n-k}_w$ and a suitable set $W = W(p,\xi) \subset \C^{n-k}$. Thus it is enough to choose $\{\varepsilon_l\}$ in such a way that the assertion of the lemma holds true simultaneously for all sets $\mathcal{F}_1, \ldots, \mathcal{F}_k$.

Let $k = 1$ and fix $\Phi \in \mathcal{B}^\infty_{psh}(\Omega)$. By Lemma \ref{thm_liouville2}, we know that $\Phi \equiv C$ on $\mathcal{E}$ for some constant $C \in \R$. Now let $\{B_\nu\}_{\nu = \nu_0}^\infty$ be an exhaustion of $\C$ by open sets $B_\nu \subset \C$ such that $E_\nu \cap (B_\nu \times \C^{n-1}) \subset \Omega$ for every $\nu \ge \nu_0$, and consider the following sequence of functions $\varphi_\nu \colon B_\nu \to \R$,
\[ \varphi_\nu(z) \coloneqq \frac{1}{2^\nu} \sum_{j=1}^{2^\nu} \Phi\big(z, w_j^{(\nu)}(z)\big). \]
Since on compact subsets of $\C^n$ the sequence $\{E_\nu\}$ converges in the Hausdorff metric to $\mathcal{E}$, one can easily see that $\{\varphi_\nu\}$ converges locally uniformly to the function $\varphi \equiv C$. Now recall that for fixed $z_0 \in \C$ and for $\nu \ge \nu_0$ large enough the Poisson-Jensen formula  for $\varphi_\nu$ on $\Delta(z_0,1)$ states that
\[ \frac{1}{2\pi} \int_0^{2\pi} \varphi_\nu(z_0 + e^{i\theta}) \,d\theta - \varphi_\nu(z_0) = - \frac{1}{2\pi} \int_{\Delta(z_0,1)} \log\abs{z-z_0} \Delta\varphi_\nu(z) \,d\mu. \]
Assume, to get a contradiction, that we can find $\nu_0 \in \N$, a positive constant $L>0$ and a subset $M \subset \Delta(z_0,1)$ of positive Lebesgue measure such that $\Delta \varphi_\nu > L$ on $M$ for every $\nu \ge \nu_0$. Then from the above formula and locally uniform convergence of $\{\varphi_\nu\}$ we get that
\begin{equation} \label{equ_poissonjensen} \begin{split}
 0 = C - C &= \frac{1}{2\pi} \int_0^{2\pi} \varphi(z_0 + e^{i\theta}) \,d\theta - \varphi(z_0) \\& = \lim_{\nu \to \infty} - \frac{1}{2\pi} \int_{\Delta(z_0,1)} \log\abs{z-z_0} \Delta\varphi_\nu(z) \,d\mu > 0, 
\end{split} \end{equation}
which is a contradiction. We will use this observation to show that for a suitable choice of $\{\varepsilon_l\}$ the Levi form of every $\Phi \in \mathcal{B}^\infty_{psh}(\Omega)$ has to vanish identically on $\mathcal{E}$.

We first specify the choice of the sequence $\{\varepsilon_l\}$. For every $\nu \in \N$, let $\Reg E_\nu \subset \{(z,w) \in E_\nu : z \neq a_l \text{ for every } l \in \N_\nu\}$ denote the regular part of $E_\nu$ and for every $(z,w) \in \Reg E_\nu$ let $\lambda_\nu(z,w) \subset \C^n$ be the complex $1$-dimensional subspace that is tangent to $E_\nu$ at $(z,w)$. Further, for every $q \in \N_{n-1}$ and every $\alpha \ge 1$, let $\Gamma^q(\alpha) \subset \C^n$ denote the closed cone $\Gamma^q(\alpha) \coloneqq \{(z,w) \in \C^n : \abs{w_q} \ge (1-1/\alpha)\norm{(z,w)}\}$. Since $\lambda_\nu(z,w)$ converges in $\CP^{n-1}$ to the $w_{\langle l \rangle}$-axis for $z \to a_l$ (here we use the fact that $a_l \neq a_{l'}$ for $l \neq l'$), it is then easy to see that we can choose inductively the sequence $\{\varepsilon_l\}_{l=1}^\infty$ and a second sequence $\{\delta_l\}_{l=1}^\infty$ of positive numbers both converging to zero so fast that the following assertion is satisfied for every $\nu \in \N$:
\begin{equation*} 
\begin{array}{lc} 
\begin{split}& \textstyle \text{for each } l \in \N_\nu \text{ one has that } \lambda_\nu(z,w) \subset \Gamma^{\langle l \rangle}\big(l + 1 - \sum_{l'=l+1}^\nu 1/2^{l'}\big) \\ & \textstyle \text{for } (z,w) \in \big[\big(\Delta(a_l,\delta_l) \setminus \bigcup_{l'=l+1}^\nu \Delta(a_{l'}, \delta_{l}/2^{l'+1})\big) \times \C^{n-1}\big] \cap \Reg E_\nu.
\end{split} & \quad (I_\nu)
\end{array}
\end{equation*}
Indeed, for the case $\nu =1$ fix arbitrary $\varepsilon_1 > 0$ and then choose $\delta_1 > 0$ so small that $\lambda_1(z,w) \subset \Gamma^{\langle 1 \rangle}(2)$ for every $(z,w) \in [\Delta(a_1,\delta_1) \times \C^{n-1}] \cap \Reg E_1$. Assume now that $\varepsilon_1, \ldots, \varepsilon_\nu$ and $\delta_1, \ldots, \delta_\nu$ are already chosen in such a way that $(I_\nu)$ holds true. Since $E_\nu$ and $E_{\nu + 1}$, viewed as set-valued functions over $\C_z$, differ only by the term $\varepsilon_{\nu+1} \textbf{e}_{\nu+1} \sqrt{z-a_{\nu+1}}$, we can now choose $\varepsilon_{\nu+1} > 0$ so small that $\lambda_{\nu+1}(z,w) \subset \Gamma^{\langle l \rangle}\big(l + 1 - \sum_{l'=l+1}^{\nu+1} 1/2^{l'}\big)$ for every $l \in \N_\nu$ and $(z,w) \in \big[\big(\Delta(a_l,\delta_l) \setminus \bigcup_{l'=l+1}^{\nu+1} \Delta(a_{l'}, \delta_{l}/2^{l'+1})\big) \times \C^{n-1}\big] \cap \Reg E_{\nu+1}$ (observe that $\Reg E_{\nu+1} \subset \{(z,w) \in \C \times \C^{n-1} : \exists (z,w') \in \Reg E_\nu \text{ such that } w = w' + \varepsilon_{\nu+1} \textbf{e}_{\nu+1}\sqrt{z-a_{\nu+1}}\}$). With $\varepsilon_{\nu+1}$ now being fixed we can then choose $\delta_{\nu+1} > 0$ so small that $\lambda_{\nu+1}(z,w) \subset \Gamma^{\langle \nu+1 \rangle}(\nu+2)$ for every $(z,w) \in [\Delta(a_{\nu+1}, \delta_{\nu+1}) \times \C^{n-1}] \cap \Reg E_{\nu + 1}$. But then $(I_{\nu+1})$ is satisfied which completes our induction on $\nu$. Hence for ${M}_l' \coloneqq \Delta(a_l,\delta_l) \setminus \bigcup_{l'=l+1}^\infty \Delta(a_{l'}, \delta_{l}/2^{l'+1})$ we now have
\begin{equation} \label{equ_slopeE} \lambda_{\nu+\mu}(z,w)  \subset \Gamma^{\langle \nu \rangle}(\nu) \text{ for every } \mu,\nu \ge 1 \text{ and } (z,w) \in (M_\nu' \times \C^{n-1}) \cap \Reg E_{\nu+\mu}, \end{equation}
and $M_\nu' \subset \C$ has positive Lebesgue measure. Moreover, as above, we choose $\{\varepsilon_l\}$ in such a way that $\varepsilon_l\sqrt{\abs{z-a_l}} < 1/2^l$ on $\Delta(0,l)$ for every $l \in \N$.

We now want to show that with the above choice of the sequence  $\{\varepsilon_l\}$ the Levi form $\Lev(\Phi)((z_0,w_0),\,\cdot\,)$ vanishes on $\C_{w_q}$ for every $\Phi \in \mathcal{B}^\infty_{psh}(\Omega)$, $(z_0,w_0) \in \mathcal{E}$ and $q \in \N_{n-1}$. Indeed, assume, to get a contradiction, that $\Lev(\Phi)((z_0,w_0),\,\cdot\,) > 0$ on $\C_{w_q}$ for some fixed data $\Phi \in \mathcal{B}^\infty_{psh}(\Omega)$, $(z_0,w_0) \in \mathcal{E}$ and $q \in \N_{n-1}$. By smoothness of $\Phi$, we can then find positive constants $r, \alpha, \tilde{L} > 0$ such that $\Lev(\Phi)((z,w),\xi) \ge \tilde{L}\cdot\norm{\xi}^2$ for every $(z,w) \in B^n((z_0,w_0),r)$ and every $\xi \in \Gamma^q(\alpha)$. For every $\nu,\mu \in \N$, $j \in \N_{2^\nu}$ and $z \in \C$ let $\{w_1^{(\mu)}(\nu,j;z), \ldots, w_{2^\mu}^{(\mu)}(\nu,j;z)\} = (z, w_j^{(\nu)}(z)) + \sum_{l=\nu+1}^{\nu+\mu} \varepsilon_l \textbf{e}_{\langle l \rangle} \sqrt{z - a_l}$. Since $\{\varepsilon_l\}$ is converging to zero so fast that $(\ref{equ_slopeE})$ holds true, it is easy to see that we can find $\nu \in \N$, $\langle\nu\rangle = q$, and $j_0 \in \N_{2^\nu}$ such that the graphs of the functions $w_1^{(\mu)}(\nu,j_0; \,\cdot\,), \ldots, w_{2^\mu}^{(\mu)}(\nu,j_0; \,\cdot\,)$ over $M_\nu'$ are contained in $B^n((z_0,w_0),r)$ and such that $\lambda_{\nu+\mu}(z,w_k^{(\mu)}(\nu,j_0; z)) \subset \Gamma^q(\alpha)$ for every $\mu \in \N$, $k \in \N_{2^\mu}$ and $z \in M_\nu' \setminus \pi(\Sing E_{\nu+\mu})$, where $\pi \colon \C^n \to \C_z$ is the canonical projection. Now if we define the functions $\varphi_\nu$ as before, then we get
\begin{equation} \label{equ_Laplaceestimate}\begin{split}
\Delta \varphi_{\nu+\mu}(z) &= \frac{1}{2^{\nu+\mu}} \sum_{j=1}^{2^\nu}\sum_{k=1}^{2^\mu} \Delta_z\big[ \Phi(z,w_k^{(\mu)}(\nu,j;z))\big] \ge \frac{1}{2^{\nu+\mu}} \sum_{k=1}^{2^\mu} \Delta_z \big[\Phi(z,w_k^{(\mu)}(\nu,j_0;z))\big] \\ &\ge \frac{\tilde{L}}{2^\nu} \eqqcolon L
 \end{split} \end{equation}
for every $\mu \in \N$ and $z \in M_\nu' \setminus \pi(\Sing E_{\nu+\mu})$, since on every convex subset of $\C \setminus \pi(\Sing E_{\nu+\mu})$ we can assume the functions $w_k^{(\mu)}(\nu,j_0; \,\cdot\,)$ to be holomorphic. But it is clear from the construction of the sequence $\{E_{\nu+\mu}\}$ that each of the sets $\pi(\Sing E_{\nu+\mu}) \subset \C$ has Lebesgue measure zero. Hence the Lebesgue measure of $M_\nu \coloneqq M_\nu' \setminus \bigcup_{\mu=1}^\infty \pi(\Sing E_{\nu+\mu})$ is positive, and we have already seen in $(\ref{equ_poissonjensen})$ that this leads to a contradiction.

We already know that for every $\Phi \in \mathcal{B}^\infty_{psh}(\Omega)$ and $(z_0,w_0) \in \mathcal{E}$ the Levi form $\Lev(\Phi)((z_0,w_0),\,\cdot\,)$ vanishes on $\C^{n-1}_w$. Assume, to get a contradiction, that there exist $\Phi \in \mathcal{B}^\infty_{psh}(\Omega)$ and $(z_0,w_0) \in \mathcal{E}$ such that the Levi form $\Lev(\Phi)((z_0,w_0),\,\cdot\,)$ is not identically zero. Then $\Lev(\Phi)((z_0,w_0),\xi) = \tilde{c}\cdot\abs{\xi_z}^2$ for some constant $\tilde{c}$, where $\xi = (\xi_z,\xi_w) \in \C_z \times \C^{n-1}_w$. Hence, by smoothness of $\Phi$, we can find $r,c > 0$ such that $\Lev(\Phi)((z,w),\xi) \ge c\cdot\abs{\xi_z}^2$ for every $(z,w) \in B^n((z_0,w_0),r)$. Thus whenever $f$ is a holomorphic mapping from an open subset of $\Delta(z_0,r)$ to $\C^{n-1}_w$ such that its graph is completely contained in $B^n((z_0,w_0),r)$, we have $\Delta_z[\Phi(z,f(z))] \ge c$. Then we can argue as in $(\ref{equ_Laplaceestimate})$ to conclude that there exists $\nu \in \N$ such that $\Delta \varphi_{\nu+\mu}(z) \ge c/2^\nu \eqqcolon L$ for every $\mu \in \N$ and $z \in M_\nu$, where $\varphi_{\nu+\mu}(z) \coloneqq \frac{1}{2^{\nu+\mu}} \sum_{j=1}^{2^{\nu+\mu}} \Phi(z, w_j^{(\nu+\mu)}(z))$ as before. In view of $(\ref{equ_poissonjensen})$, this again leads to a contradiction.
\end{proof}

Observe that Lemma \ref{thm_Ekpseudoconcave} and Lemma \ref{thm_OmegacoresE} together prove Theorem \ref{thm_coreconcavity} in the case where $q' \in \{1, \ldots, n-1\}$. It only remains to show that the set $\mathcal{E}$ has the additional properties remarked after the formulation of the theorem.

\begin{lemma} If $\{\varepsilon_l\}$ is decreasing fast enough, then $\mathcal{E}$ contains no analytic variety of positive dimension. Moreover, $\widehat{bB^n(0,R) \cap \mathcal{E}} = \overline{B^n(0,R)} \cap \mathcal{E}$ for every $R > 0$, where $\widehat{bB^n(0,R) \cap \mathcal{E}}$ denotes the polynomial hull of $bB^n(0,R) \cap \mathcal{E}$.
\end{lemma}
\begin{proof}
By Lemma 3.7 in \cite{HarzShcherbinaTomassini12}, we can choose $\{\varepsilon_l\}$ such that $\pi_q(\mathcal{E}) = \mathcal{E}_q \cap (\C^k_z \times \C_{w_q})$ contains no analytic variety of positive dimension for every $q\in\N_{n-k}$, where $\pi_q \colon \C^n \to \C^k_z \times \C_{w_q}$ is the canonical projection. Thus for every analytic set $A \subset \mathcal{E}$ all projections $\pi_q(A)$, $q = 1, \ldots, n-k$, consist of only one point, i.e., $A=\{P\}$ for some $P \in \mathcal{E}$. The second assertion follows as in the proof of Theorem 1.1 in \cite{HarzShcherbinaTomassini12}.
\end{proof}

\section{Open questions} \label{sec_questions}
In this last section we state some open questions related to the content of the paper.

\noindent \textbf{1. Existence of global defining functions}

\begin{question}
Let $X$ be a complex space and let $\Omega \subset X$ be a smoothly strictly pseudoconvex domain. Does there exist a minimal global defining function for $\Omega$, i.e., does there exist a smoothly plurisubharmonic function $\varphi \colon U \to \R$ defined on an open neighbourhood $U \subset X$ of $\overline{\Omega}$ such that $\Omega = \{\varphi < 0\}$ and $\varphi$ is strictly smoothly plurisubharmonic outside $\mathfrak{c}(\Omega)$? (For the definition of $\mathfrak{c}(\Omega)$ in the setting of complex spaces see p.\ \pageref{def_corespaces}.)
\end{question}

\begin{question}
Let $X$ be a complex space and let $\Omega \subset X$ be a smoothly strictly $q$-pseudoconvex domain. Does there exist a smoothly $q$-plurisubharmonic function $\varphi \colon U \to \R$ defined on an open neighbourhood $U \subset X$ of $\overline{\Omega}$ such that $\Omega = \{\varphi < 0\}$ and $\varphi$ is  smoothly strictly $q$-plurisubharmonic near $b\Omega$? 
\end{question}

\noindent \textbf{2. The core of a domain}

\begin{question}
Let $\Omega \subset \C^n$ be a domain and let $\omega \subset \Omega$ be a domain such that $\mathfrak{c}(\Omega) \subset \omega$. Does it follow that $\mathfrak{c}(\omega) = \mathfrak{c}(\Omega)$?
\end{question}

\begin{question}
Let $\mathcal{M}$ be a complex manifold. Is it possible to characterize the core type subsets  $E \subset \mathcal{M}$? (For the definition of core type sets see p.\ \pageref{def_coretype}.)
\end{question}

\begin{question}
Let $\Omega \subset \C^n$ be a domain. Is it true that $\widehat{bB^n(0,R) \cap \mathfrak{c}(\Omega)} = \overline{B^n(0,R)} \cap \mathfrak{c}(\Omega)$ for every $R>0$? Here $\widehat{bB^n(0,R) \cap \mathfrak{c}(\Omega)}$ denotes the polynomial hull of the set $bB^n(0,R) \cap \mathfrak{c}(\Omega)$.
\end{question}

\begin{question}
Is it true that $\mathfrak{c}^{s_1}(\Omega) = \mathfrak{c}^{s_2}(\Omega)$ for every domain $\Omega \subset \C^n$ and every $s_1,s_2$ such that the corresponding cores $\mathfrak{c}^{s_1}(\Omega)$ and $\mathfrak{c}^{s_2}(\Omega)$ are defined? (For the definition of $\mathfrak{c}^s(\Omega)$ see p.\ \pageref{def_coreweaklysmooth}.)
\end{question}

\begin{question}
Let $\Omega \subset \C^2$ be a strictly pseudoconvex domain with smooth boundary. Is it true that every smooth and bounded from above plurisubharmonic function on $\Omega$ is constant on each connected component of $\mathfrak{c}(\Omega)$?
\end{question}

\begin{question}
Let $\mathcal{M}$ be a complex manifold and let $\Omega \subset \mathcal{M}$ be a domain. Can it happen that a maximal component of $\mathfrak{c}(\Omega)$ consists of only one point? Moreover, in the case when $\mathcal{M}$ is Stein, is it true that no maximal component of $\mathfrak{c}(\Omega)$ is relatively compact in $\mathcal{M}$? (For the definition of maximal components of $\mathfrak{c}(\Omega)$ see p.\ \pageref{def_maximalcomponent}.)
\end{question}

\begin{question}
Let $\Omega \subset \C^n$ be a strictly pseudoconvex domain with smooth boundary. Can it happen that the set $\mathfrak{c}(\Omega)$ is not pluripolar? Or, even more, can it happen that $\mathfrak{c}(\Omega)$ has a nonempty interior? And finally the strongest version of this question: Does there exist a strictly pseudoconvex domain $\Omega \subset \C^n$ containing a Fatou-Bieberbach domain?
\end{question}

\begin{question}
Let $\Omega \subset \C^n$ be a strictly pseudoconvex domain with smooth boundary. Is it always true that $\Omega \setminus \mathfrak{c}(\Omega)$ is connected?
\end{question}

%
%
%%% bibliography %%%
 \vspace{1truecm}

%
%
%
%%% contact addresses %%%
{\sc T. Harz: Department of Mathematics, University of Wuppertal --- 42119 Wuppertal, Germany}
  
{\em e-mail address}: {\texttt harz@math.uni-wuppertal.de}\vspace{0.3cm}
  
{\sc N. Shcherbina: Department of Mathematics, University of Wuppertal --- 42119 Wuppertal, Germany}
  
{\em e-mail address}: {\texttt shcherbina@math.uni-wuppertal.de}\vspace{0.3cm}
  
{\sc G. Tomassini: Scuola Normale Superiore, Piazza dei Cavalieri, 7 --- 56126 Pisa, Italy}
  
{\em e-mail address}: {\texttt g.tomassini@sns.it}

\end{document}